\documentclass[english]{article}
\usepackage[T1]{fontenc}
\usepackage[latin9]{inputenc}
\usepackage{geometry}
\usepackage{color}
\usepackage{babel}
\usepackage{mathrsfs}
\usepackage{amsmath}
\usepackage{amsthm}
\usepackage{amssymb}
\usepackage{stmaryrd}
\usepackage{graphicx}
\usepackage{esint}
\usepackage[all]{xy}
\usepackage[unicode=true,pdfusetitle,
 bookmarks=true,bookmarksnumbered=false,bookmarksopen=false,
 breaklinks=false,pdfborder={0 0 1},backref=false,colorlinks=true]
 {hyperref}

\makeatletter
\newcommand\thmsname{\protect\theoremname}
\newcommand\nm@thmtype{theorem}
\theoremstyle{plain}

\newenvironment{namedthm}[1][Undefined Theorem Name]{
  \ifx{#1}{Undefined Theorem Name}\renewcommand\nm@thmtype{theorem*}
  \else\renewcommand\thmsname{#1}\renewcommand\nm@thmtype{namedtheorem}
  \fi
  \begin{\nm@thmtype}}
  {\end{\nm@thmtype}}
\theoremstyle{plain}
\newtheorem{thm}{\protect\theoremname}[section]
\theoremstyle{remark}
\newtheorem{rem}[thm]{\protect\remarkname}
\theoremstyle{plain}
\newtheorem{lem}[thm]{\protect\lemmaname}
\theoremstyle{plain}
\newtheorem{prop}[thm]{\protect\propositionname}
\theoremstyle{definition}
\newtheorem{defn}[thm]{\protect\definitionname}
\theoremstyle{remark}
\newtheorem{notation}[thm]{\protect\notationname}
\theoremstyle{definition}
\newtheorem{example}[thm]{\protect\examplename}
\theoremstyle{plain}
\newtheorem{cor}[thm]{\protect\corollaryname}
\theoremstyle{plain}
\newtheorem{assumption}[thm]{\protect\assumptionname}
\newcommand{\lyxaddress}[1]{
	\par {\raggedright #1
	\vspace{1.4em}
	\noindent\par}
}

\usepackage[all]{xy}

\AtBeginDocument{
  
}

\makeatother

\providecommand{\assumptionname}{Assumption}
\providecommand{\corollaryname}{Corollary}
\providecommand{\definitionname}{Definition}
\providecommand{\examplename}{Example}
\providecommand{\lemmaname}{Lemma}
\providecommand{\notationname}{Notation}
\providecommand{\propositionname}{Proposition}
\providecommand{\remarkname}{Remark}
\providecommand{\theoremname}{Theorem}

\begin{document}
\title{\textbf{\textsc{Generalized Long-Moody functors}}}
\author{\textsc{\textcolor{black}{Arthur Soulié}}}
\maketitle
\begin{abstract}
In this paper, we generalize the principle of the Long-Moody construction
for representations of braid groups to other groups, such as mapping
class groups of surfaces. Namely, we introduce endofunctors over a
functor category that encodes representations of a family of groups.
They are called Long-Moody functors and provide new representations.
In this context, notions of polynomial functors are defined and play
an important role in the study of homological stability. We prove
that, under additional assumptions, a Long-Moody functor increases
the \textit{very strong} and \textit{weak} polynomial degrees of functors
by one.
\end{abstract}
{\let\thefootnote\relax\footnotetext{{This work was partially supported by the ANR Project {\em ChroK}, {\tt ANR-16-CE40-0003} and by a JSPS Postdoctoral Fellowship Short Term.\\2010 \textit{Mathematics Subject Classification}: 
 18D10, 18A25, 20C07, 20C99, 20J99, 20F36, 20F38, 57M07, 57N05.\\Keywords: Long-Moody construction, Functor categories, Mapping class groups, Polynomial functors.\\E-mail address: artsou@hotmail.fr}}

\section*{Introduction}

In $1994$, as a result of a collaboration with Moody, Long \cite{Long1}
gave a method to construct a new linear representation of the braid
group on $n$ strands $\mathbf{B}_{n}$ from a representation of $\mathbf{B}_{n+1}$.
The underlying framework of this construction, called the \textit{Long-Moody
construction}, naturally arises in many situations in connection with
topology: the first aim of this paper is to extend this construction
to these settings.

Namely, for a family of groups $\{G_{n}\}_{n\in\mathbb{N}}$ equipped
with injections $\gamma_{n}:G_{n}\hookrightarrow G_{n+1}$, we give
a method to construct a representation of $G_{n}$ from a representation
of $G_{n+1}$, which generalizes the underlying idea of the original
Long-Moody construction. This machinery is based on the following
ingredients: a family of groups $\{H_{n}\}_{n\in\mathbb{N}}$, an
action $\mathcal{A}_{n}:G_{n}\to\textrm{Aut}(H_{n})$ and a group
morphism $\chi_{n}:H_{n}\rightarrow G_{n+1}$ for all natural numbers
$n$. For instance, we can consider the following situation:

\begin{itemize}
\item Let $\mathbf{\varSigma}_{g,1}$ be a smooth connected compact surface
with genus $g$ and one boundary component. We consider the surfaces
$\{\mathbf{\varSigma}_{g,1}^{n}\}_{n\in\mathbb{N}}$ obtained from
$\mathbf{\varSigma}_{g,1}$ by removing $n$ points from its interior.
We take the family of groups $\{G_{n}\}_{n\in\mathbb{N}}$ to be the
family of \textit{mapping class groups} $\{\mathbf{\Gamma}_{g,1}^{n}\}_{n\in\mathbb{N}}$,
that is the group of isotopy classes of homeomorphisms fixing the
punctures setwise and restricting to the identity on the boundary
component.
\item We define the injection $\gamma_{n}:\mathbf{\Gamma}_{g,1}^{n}\hookrightarrow\mathbf{\Gamma}_{g,1}^{n+1}$
by extending an element $\varphi\in\mathbf{\Gamma}_{g,1}^{n}$ to
a mapping class of $\mathbf{\varSigma}_{g,1}^{n+1}$ by the identity
on the complement $\mathbf{\varSigma}_{0,1}^{1}$ of $\mathbf{\varSigma}_{g,1}^{n}\hookrightarrow\mathbf{\varSigma}_{g,1}^{n+1}$.
\item We take the family of groups $\{H_{n}\}_{n\in\mathbb{N}}$ to be the
fundamental groups of the surfaces $\{\pi_{1}(\mathbf{\varSigma}_{g,1}^{n},p)\}_{n\in\mathbb{N}}$,
where $p$ is a point in the boundary component which lies on the
part of the boundary of each $\mathbf{\varSigma}_{g,1}^{n}$ that
is common to all of the surfaces $\{\mathbf{\varSigma}_{g,1}^{n}\}_{n\in\mathbb{N}}$
(and thus $p$ is independent of $n$); see Figure \ref{fig:Generators-and-paths}.
\item We consider the natural action $\mathcal{A}_{n}$ of $\mathbf{\Gamma}_{g,1}^{n}$
on $\pi_{1}(\mathbf{\varSigma}_{g,1}^{n},p)$.
\item We consider a morphism $\chi_{n,1}:\pi_{1}(\varSigma_{g,1}^{n},p)\rightarrow\mathbf{\Gamma}_{g,1}^{n+1}$
induced by the \textit{point pushing map} as it called in Farb and
Margalit \cite[Section 4.2.1]{farb2011primer}; it is explicitly defined
by the formula (\ref{eq:push_formula}).
\end{itemize}
The input of the Long-Moody construction is a representation $\rho_{n+1}:G_{n+1}\rightarrow GL_{R}(M_{n+1})$,
for $R$ a commutative ring and $M_{n+1}$ an $R$-module. Then we
note that $M_{n+1}$ is endowed with an $H_{n}$-module structure
via the composition $\rho_{n+1}\circ\chi_{n}:H_{n}\rightarrow G_{n+1}\rightarrow GL_{R}(M_{n+1})$.
We denote by $\mathcal{I}_{R\left[H_{n}\right]}$ the augmentation
ideal of the group $H_{n}$ -- which has a canonical $R[H_{n}]$-module
structure. The key idea of the Long-Moody construction is to give
the tensor product
\[
\mathcal{I}_{R[H_{n}]}\underset{R[H_{n}]}{\varotimes}M_{n+1}
\]
a $G_{n}$-module structure as follows. The action $\mathcal{A}_{n}:G_{n}\to\textrm{Aut}(H_{n})$
linearly extends to an action on $\mathcal{I}_{R[H_{n}]}$ and $M_{n+1}$
is a $G_{n}$-module structure via the restriction of the morphism
$\rho_{n+1}$ along $\gamma_{n}:G_{n}\hookrightarrow G_{n+1}$. Then,
for all $g\in G_{n}$, $x\in\mathcal{I}_{R[H_{n}]}$ and $m\in M_{n+1}$,
we define a diagonal action of $G_{n}$ on the above tensor product
by
\[
g\cdot(x\underset{R[H_{n}]}{\otimes}m)=\mathcal{A}_{n}(g)(x)\underset{R[H_{n}]}{\varotimes}\rho_{n+1}(\gamma_{n}(g))(m).
\]
 For this action to be well-defined, it has to be compatible with
the $H_{n}$-module structures on both sides of the tensor product:
this requires technical compatibilities between the morphisms $\mathcal{A}_{n}$,
$\chi_{n}$ and $\gamma_{n}$. This is the delicate point of the construction.
However these compatibilities can easily be described by using categorical
tools, in particular the \textit{Grothendieck construction}, as follows.

We consider the  groupoid \textit{$\mathcal{G}$} with objects indexed
by the natural numbers (and denoted by $\underline{n}$) and with
the groups $\{G_{n}\}_{n\in\mathbb{N}}$ as automorphism groups. We
denote by $\mathfrak{Gr}$ the category of groups and $\intop^{\mathcal{G}}F$
is the Grothendieck construction on a functor $F:\mathcal{G}\rightarrow\mathfrak{Gr}$;
see Section \ref{sec:Recollections-on-Quillen's}. The cornerstone
to define the Long-Moody construction is to assume that the families
of morphisms $\{\mathcal{A}_{n}\}_{n\in\mathbb{N}}$ and $\{\chi_{n}\}_{n\in\mathbb{N}}$
assemble to define functors $\mathcal{A}:\mathcal{G}\rightarrow\mathfrak{Gr}$
such that $\mathcal{A}(\underline{n})=H_{n}$ and $\chi:\intop^{\mathcal{G}}\mathcal{A}\rightarrow\mathcal{G}$.
In addition, we require that the following diagram be commutative:
\[
\xymatrix{\mathcal{G}\,\ar@{^{(}->}[r]^{i}\ar@{->}[dr]_{\gamma} & \intop^{\mathcal{G}}\mathcal{A}\ar@{->}[d]^{\chi}\\
 & \mathcal{G},
}
\]
where $i$ is the evident section of the projection functor induced
by the Grothendieck construction and $\gamma$ is the functor defined
by $\underline{n}\mapsto\underline{n+1}$ and induced by the canonical
injections $\{\gamma_{n}:G_{n}\hookrightarrow G_{n+1}\}_{n\in\mathbb{N}}$.
The Grothendieck construction encodes semidirect product structures,
so this condition actually reflects the factorization of the injections
$\gamma_{n}$ through $H_{n}\rtimes_{\mathcal{A}_{n}}G_{n}$; see
Lemma \ref{lem:equivcond:coherenceconditionsigmanan}.

The necessary coherence conditions on the group morphisms $\{\chi_{n}\}_{n\in\mathbb{N}}$
are restrictive. However, the trivial morphisms $\{\chi_{n,tr}:H_{n}\rightarrow G_{n+1}\}_{n\in\mathbb{N}}$
always satisfy the necessary technical condition, and there are many
situations where non-trivial $\{\chi_{n}\}_{n\in\mathbb{N}}$ arise
naturally, in particular for some families of groups in connection
with topology. For example, the above non-trivial morphisms $\{\chi_{n,1}:\pi_{1}(\varSigma_{g,1}^{n},p)\rightarrow\mathbf{\Gamma}_{g,1}^{n+1}\}_{n\in\mathbb{N}}$
satisfy the appropriate conditions; see Lemma \ref{lem:sigman1satisfiesfirtcond}.

Furthermore, instead of considering these constructions for only one
group $G_{n}$, a natural question is how to extend these constructions
to \textit{families of representations} of $\text{\{}G_{n}\}_{n\in\mathbb{N}}$.
Denoting by $\mathbf{Fct}(\mathfrak{C},R\textrm{-}\mathfrak{Mod})$
the category of functors from a small category $\mathfrak{C}$ to
the category of $R$-modules $R\textrm{-}\mathfrak{Mod}$, an object
$M$ of $\mathbf{Fct}(\mathcal{G},R\textrm{-}\mathfrak{Mod})$ is
a collection of linear representations $\text{\{}\rho_{n}:G_{n}\rightarrow GL_{R}(M_{n})\}_{n\in\mathbb{N}}$.
For instance, the groupoid $\mathcal{G}$ for the family of mapping
class groups $\{\mathbf{\Gamma}_{g,1}^{n}\}_{n\in\mathbb{N}}$ is
introduced in Section \ref{subsec:The-monoidal-groupoid} and we denote
it by $\mathfrak{M}_{2}^{+,g}$. The condition for the functor $M$
to encode a family of representations consists of requiring that there
exist maps $m_{n}:GL_{R}(M_{n})\rightarrow GL_{R}(M_{n+1})$ for all
natural numbers $n$ such that:
\begin{equation}
m_{n}\circ\rho_{n}(g)=\rho_{n+1}(\gamma_{n}(g))\label{eq:}
\end{equation}
for all $g\in G_{n}$. Then, we say that the representations $\{M_{n}\}_{n\in\mathbb{N}}$
form a \textit{family} of linear representations of the groups $\{G_{n}\}_{n\in\mathbb{N}}$.
This notion has been described as a \textit{consistent sequence of
representations }in Church, Ellenberg and Farb \cite{ChurchEllenbergFarbstabilityFI}.
However, the extra information (\ref{eq:}) on $M$ is not encoded
by the fact that $M$ is an object of $\mathbf{Fct}(\mathcal{G},R\textrm{-}\mathfrak{Mod})$.
Quillen's bracket construction (see Grayson \cite[p.219]{graysonQuillen})
defines a new category $\mathfrak{U}\mathcal{G}$ with enough additional
morphisms to resolve this failure. In particular, the groupoid $\mathcal{G}$
is its maximal subgroupoid, $\mathfrak{U}\mathcal{G}$ contains extra
morphisms from each object $\underline{n}$ to the object $\underline{n+1}$
allowing to encode the compatibility condition (\ref{eq:}) and $\gamma$
canonically extends to $\mathfrak{U}\mathcal{G}$. Actually the category
$\mathfrak{U}\mathcal{G}$ forces some additional relations on the
morphisms which are not used to describe the above compatibilities.
For example, the family of symplectic representations of the mapping
class groups defines an object of $\mathbf{Fct}(\mathfrak{U}\mathfrak{M}_{2}^{+,g},R\textrm{-}\mathfrak{Mod})$.
Also the category $\mathfrak{U}\mathcal{G}$ is fundamental since
it provides a natural setting to study coefficient systems for homological
stability; see Wahl and Randal-Williams \cite[Sections 1 and 4]{WahlRandal-Williams}.

Therefore, our goal is to define the Long-Moody construction as an
endofunctor of the functor category $\mathbf{Fct}(\mathfrak{U}\mathcal{G},R\textrm{-}\mathfrak{Mod})$.
We thus deal with an extension problem: we require the functors $\mathcal{A}$
and $\chi$ to respectively extend along the inclusions $\mathcal{G}\hookrightarrow\mathfrak{U}\mathcal{G}$
and $\intop^{\mathcal{G}}\mathcal{A}\rightarrow\intop^{\mathcal{\mathfrak{U}G}}\mathcal{A}$
so that the following diagram is commutative
\[
\xymatrix{\mathfrak{U}\mathcal{G}\,\ar@{^{(}->}[r]^{i}\ar@{->}[dr]_{\gamma} & \intop^{\mathfrak{U}\mathcal{G}}\mathcal{A}\ar@{->}[d]^{\chi}\\
 & \mathfrak{U}\mathcal{G}.
}
\]
Under these assumptions, we prove:
\begin{namedthm}[\textit{Theorem }$\boldsymbol{A}$ (Definition \ref{Thm:LMFunctor})]
 There is a right-exact functor $\mathbf{LM}_{\{\mathcal{G},\chi\}}:\mathbf{Fct}(\mathfrak{U}\mathcal{G},R\textrm{-}\mathfrak{Mod})\rightarrow\mathbf{Fct}(\mathfrak{U}\mathcal{G},R\textrm{-}\mathfrak{Mod})$,
called the Long-Moody functor associated to $\{\mathcal{G},\chi\}$,
that assigns to all objects $F$ of $\mathbf{Fct}(\mathfrak{U}\mathcal{G},R\textrm{-}\mathfrak{Mod})$
and objects $\underline{n}$ of $\mathfrak{U}\mathcal{G}$ the $R$-module
\begin{equation}
\mathbf{LM}_{\{\mathcal{G},\chi\}}(F)(\underline{n})=\mathcal{I}_{R[H_{n}]}\underset{R[H_{n}]}{\otimes}F(\underline{n+1})\label{eq:def_LM_1}
\end{equation}
and to all elements $g$ of $G_{n}$ the morphism $\mathbf{LM}_{\{\mathcal{G},\chi\}}(F)(g)=\mathcal{A}_{n}(g)\otimes_{R[H_{n}]}F(\gamma_{n}(g))$.
\end{namedthm}
In particular, if we take the groups $\{G_{n}\}_{n\in\mathbb{N}}$
to be the family of braid groups $\{\mathbf{B}_{n}\}_{n\in\mathbb{N}}$
and the groups $\{H_{n}\}_{n\in\mathbb{N}}$ to be the family of free
groups $\{\mathbf{F}_{n}\}_{n\in\mathbb{N}}$, Theorem $A$ recovers
a previous result of the author \cite[Theorem A]{soulieLMBilan}.
However, the present framework improves the previous one since it
allows us to define and recover more representations for the braid
groups; see Example \ref{exa:recoverCallegaro} for instance. Additionally,
the families of mapping class groups of surfaces also fit into this
framework and are studied in Section \ref{subsec:Examples}. Further
families of groups also fit into the present framework such as mapping
class groups of compact connected oriented $3$-manifolds with boundary,
loop braid groups or automorphism groups of free products of groups,
although they are not presented in this paper. Also, the new abstract
approach to the Long-Moody functors presented in this paper indicates
that these are particular examples of more general \textit{tensorial
constructions} introduced in Section \ref{subsec:Framework-of-the}:
further generalizations and wider applications should emerge from
these last constructions. They are however not addressed in the present
paper.

In addition, we can relate a Long-Moody functor $\mathbf{LM}_{\{\mathcal{G},\chi\}}$
to the first homology groups of the family of groups $\{H_{n}\}_{n\in\mathbb{N}}$.
We recall that $\mathcal{A}$ denotes a functor $\mathfrak{U}\mathcal{G}\rightarrow\mathfrak{Gr}$
such that $\mathcal{A}(\underline{n})=H_{n}$ for all $n$. We denote
by $H_{1}(\mathcal{A};R)$ the composition $H_{1}(-;R)\circ\mathcal{A}$,
where $H_{1}(-;R)$ denotes the first homology group functor. Then,
denoting by $R:\mathfrak{U}\mathcal{G}\rightarrow R\textrm{-}\mathfrak{Mod}$
the constant functor at $R$, the functor $\mathbf{LM}_{\{\mathcal{G},\chi\}}(R)$
is equivalent to the composition $H_{1}(\mathcal{A};R)$; see Proposition
\ref{prop:casesigmatrivial}. For instance, for any functor $\chi$
defining a Long-Moody functor for the category $\mathfrak{M}_{2}^{+,g}$
associated to the mapping class groups $\{\mathbf{\Gamma}_{g,1}^{n}\}_{n\in\mathbb{N}}$,
the functor $\mathbf{LM}_{\{\mathfrak{M}_{2}^{+,g},\chi\}}(R)$ thus
encodes the family of symplectic representations. Moreover, considering
the family of trivial morphisms $\{\chi_{n,tr}:H_{n}\rightarrow G_{n+1}\}_{n\in\mathbb{N}}$,
we also prove that for all objects $F$ of $\mathbf{Fct}(\mathfrak{U}\mathcal{G},R\textrm{-}\mathfrak{Mod})$,
there is a natural equivalence for the associated Long-Moody functor
(see Proposition \ref{prop:casesigmatrivial}):
\begin{equation}
\mathbf{LM}_{\{\mathcal{G},\chi_{tr}\}}(F)\cong H_{1}(H_{n};R)\underset{R}{\otimes}F(1+-).\label{eq:introtrivial}
\end{equation}
Hence $\mathbf{LM}_{\{\mathcal{G},\chi_{tr}\}}$ is defined as the
tensor product of $F$ with the functor $H_{1}(\mathcal{A};R)$. Nevertheless,
the equivalence (\ref{eq:introtrivial}) does not hold generally speaking
for an endofunctor $\mathbf{LM}_{\{\mathcal{G},\chi\}}$ where $\chi$
is not a functor induced by the morphisms $\{\chi_{n,tr}\}_{n\in\mathbb{N}}$.
Indeed, for each $n\in\mathbb{N}$, the $G_{n}$-module $\mathbf{LM}_{\{\mathcal{G},\chi\}}(F)(\underline{n})$
is defined by the twisted tensor product (\ref{eq:def_LM_1}), which
is not isomorphic in general to $H_{1}(\mathcal{A};R)\otimes_{R}F(\underline{n+1})$
if $F(\underline{n+1})$ is not a trivial $H_{n}$-module. Also, for
each $n\in\mathbb{N}$, the twisted first homology group $H_{1}(H_{n};F(\underline{n+1}))$
is canonically a submodule $\mathbf{LM}_{\{\mathcal{G},\chi\}}(F)(\underline{n})$;
see the $4$-term short exact sequence (\ref{eq:4_term_sequence_LM}).
For instance, let $\mathbf{LM}_{\{\mathfrak{M}_{2}^{+,g},\chi_{1}\}}$
be the Long-Moody functor defined by the above non-trivial morphisms
$\{\chi_{n,1}:\pi_{1}(\varSigma_{g,1}^{n},p)\rightarrow\mathbf{\Gamma}_{g,1}^{n+1}\}_{n\in\mathbb{N}}$.
The application of this functor provides linear representations of
the family of groups $\{\mathbf{\Gamma}_{g,1}^{n}\}_{n\in\mathbb{N}}$
which, as far as the author knows, are unknown in the literature;
see Section \ref{subsec:Applications}. In the same vein, new interesting
linear representations for the braid groups on orientable surfaces
$\{\mathbf{B}_{n}(\varSigma_{g,1})\}_{n\in\mathbb{N}}$ and for the
mapping class groups $\{\mathbf{\Gamma}_{n,1}\}_{n\in\mathbb{N}}$
are defined by the applications of some appropriate Long-Moody functors;
see Sections \ref{subsec:Applications} and \ref{par:Orientable-surfaces}.\\

Furthermore, among the objects in the category $\mathbf{Fct}(\mathfrak{U}\mathcal{G},R\textrm{-}\mathfrak{Mod})$
of particular importance are the \textit{strong }and\textit{ very
strong polynomial functors}. The first notions of polynomial functors
date back to Eilenberg and Mac Lane \cite{EilenbergMacLane} for functors
on module categories. Polynomial functors have been the object of
an intensive study because of their applications in algebraic topology
(see Henn, Lannes and Schwartz \cite{HennLannesSchwartz}), representation
theory (see Kuhn \cite{Kuhn}) and group cohomology (see Franjou,
Friedlander, Scorichenko and Suslin \cite{FranjouFriedlanderScorichenkoSuslin}).
Also this notion has progressively been extended to deal with a more
general framework than module categories; see Pirashvili \cite{Pirashvili}
for instance. Djament and Vespa recently introduced in \cite{DV3}
strong polynomial functors for symmetric monoidal categories in which
the monoidal unit is initial. This new definition encompasses all
the previous ones. This definition is extended to pre-braided monoidal
categories in which the monoidal unit is initial in \cite[Section 3]{soulieLMBilan}.
The notion of a very strong polynomial functor in this context is
introduced in \cite[Section 3]{soulieLMBilan}; it is equivalent to
that of coefficient systems of finite degree of Wahl and Randal-Williams
\cite[Section 4.4]{WahlRandal-Williams}. These notions of strong
and very strong polynomial functors extend to the more general context
of the present paper; see Section \ref{subsec:Prerequisite-on-strong}.

One reason for our interest in very strong polynomial functors is
their homological stability properties: in \cite{WahlRandal-Williams},
Randal-Williams and Wahl prove homological stability results for certain
families of groups $\{G_{n}\}_{n\in\mathbb{N}}$ with twisted coefficients
given by very strong polynomial objects of $\mathbf{Fct}(\mathfrak{U}\mathcal{G},\mathbb{Z}\textrm{-}\mathfrak{Mod})$.
Their results hold for surface braid groups and mapping class groups
of orientable and non-orientable surfaces. The representation theory
of these groups is complicated and an active research topic; see Birman
and Brendle's survey \cite[Section 4.6]{BirmanBrendlesurvey} or Margalit's
expository paper \cite{Margalit}. A fortiori, the very strong polynomial
functors associated with these groups are not well-understood.

In addition, we are interested in \textit{weak} polynomial functors,
a notion introduced by Djament and Vespa \cite[Section 3.1]{DV3}
for symmetric monoidal categories. We prove that this concept extends
to the present framework in Section \ref{subsec:Weak-polynomiality}.
A first matter of interest in this last notion is that it is more
appropriate for understanding the stable behaviour of a given functor:
it reflects more accurately than the strong polynomiality the behaviour
of functors for large values. Also weak polynomial functors of degree
less or equal to some $d\in\mathbb{N}$ form a category $\mathcal{P}ol_{d}(\mathfrak{U}\mathcal{G})$
that is \textit{localizing}, which allows one to define quotient categories
$\mathcal{P}ol_{d+1}(\mathfrak{U}\mathcal{G})/\mathcal{P}ol_{d}(\mathfrak{U}\mathcal{G})$.
These quotients provide an organizing tool for families of representations
of the groups $\{G_{n}\}_{n\in\mathbb{N}}$; see Section \ref{subsec:Weak-polynomial-functors-1}.

We then investigate the effects of Long-Moody functors on polynomial
functors. This analysis requires to make further assumptions on the
parameters $\mathcal{A}$ and $\chi$, which are detailed in Assumption
\ref{assu:decomposeAfreeproduct}. Namely we assume that there exist
two groups $H_{0}$ and $H$ such that $\mathcal{A}(\underline{n})=H^{*n}*H_{0}$
for all objects $\text{\ensuremath{\underline{n}}}$ of $\mathcal{G}$,
that $\mathcal{A}$ satisfies some compatibility conditions with respect
to the morphisms and braiding of $\mathfrak{U}\mathcal{G}$ and that
a coherence equality between the group morphisms $\{\chi_{n}:H_{n}\rightarrow G_{n+1}\}_{n\in\mathbb{N}}$
and the braiding of $\mathfrak{U}\mathcal{G}$ is checked. Assumption
\ref{assu:decomposeAfreeproduct} is a technical (but quite natural)
hypothesis, which is satisfied in many of the examples of interest,
such as mapping class groups of surfaces or surface braid groups;
see Section \ref{sec:Examples-and-applications}. Then we prove:
\begin{namedthm}[\textit{Theorem $\boldsymbol{B}$ }(Theorems \ref{Thm:emairesult}
and \ref{thm:ResultWeakpoly})]
 Under the hypotheses of Theorem $A$ and Assumption \ref{assu:decomposeAfreeproduct}
and assuming that the groups $\{H_{n}\}_{n\in\mathbb{N}}$ are free,
the Long-Moody functor $\mathbf{LM}_{\{\mathcal{G},\chi\}}$ increases
by one both the very strong and the weak polynomial degrees.
\end{namedthm}
The proof of this result relies on keystone relations for the action
of the \textit{difference} and \textit{evanescence} functors on Long-Moody
functors. For the family of braid groups $\{\mathbf{B}_{n}\}_{n\in\mathbb{N}}$,
Theorem $B$ recovers in particular the previous result of the author
for braid groups \cite[Theorem B]{soulieLMBilan} and gives the new
result for the weak polynomial degree.

Therefore the Long-Moody functors provide new families of (very) strong
polynomial and weak polynomial functors of \textit{$\mathbf{Fct}(\mathfrak{U}\mathcal{G},R\textrm{-}\mathfrak{Mod})$}
in any degree. This result allows one to gain a better understanding
of polynomial functors for mapping class groups and extends the scope
of twisted homological stability to more sophisticated sequences of
representations. It is a precise measure of the fact that the Long-Moody
functors produce more complicated (hence interesting) representations
of families of groups, but not overly complicated since polynomiality
is preserved. These methods also introduce new tools to clarify the
structures of weak polynomial functors in this context; see Proposition
\ref{prop:classificationweakpoly}.

\paragraph{Outline.}

The paper is organized as follows. In Section \ref{sec:Recollections-on-Quillen's},
we recall Quillen's bracket construction, pre-braided monoidal categories
and the Grothendieck construction. In Section \ref{sec:The-generalized-Long-Moody},
after setting up the general framework of the families of groups,
we define the generalized Long-Moody functors and give some of their
properties. Section \ref{subsec:Examples} is devoted to the application
of Long-Moody functors to the mapping class groups of surfaces and
surface braid groups (recovering in particular the case of classical
braid groups). Section \ref{sec:Strong-and-weak} introduces the notions
of strong, very strong and weak polynomial functors in the present
framework. In Section \ref{sec:Behaviour-of-the}, we consider the
effect of Long-Moody functors on strong and weak polynomial functors.
Finally, in Section \ref{sec:Examples-and-applications}, we explain
the applications of the effect of Long-Moody functors on polynomiality,
in particular their interest for homological stability results.

\paragraph{General notation.}

We fix a commutative (associative, unital, non-trivial) ring $R$
throughout this paper. We denote by $R\textrm{-}\mathfrak{Mod}$ the
category of $R$-modules. We denote by $\mathfrak{Gr}$ the category
of groups and by $*$ the coproduct in this category. For $G$ a group,
we denote its unit element by $e_{G}$.

Let $\mathfrak{Cat}$ denote the category of  small categories. Let
$\mathfrak{C}$ be an object of $\mathfrak{Cat}$. We use the abbreviation
$\textrm{Obj}(\mathfrak{C})$ to denote the set of objects of $\mathfrak{C}$.
If there exists an initial object $\textrm{Ø}$ in the category $\mathfrak{C}$,
then we denote by $\iota_{A}:\textrm{Ø}\rightarrow A$ the unique
morphism from $\textrm{Ø}$ to $A$. If $t$ is a terminal object
in the category $\mathfrak{C}$, then we denote by $t_{A}:A\rightarrow t$
the unique morphism from $A$ to $t$. The maximal subgroupoid $\mathscr{G}\mathfrak{r}(\mathfrak{C})$
is the subcategory of $\mathfrak{C}$ which has the same objects as
$\mathfrak{C}$ and whose morphisms are the isomorphisms of $\mathfrak{C}$.
We denote by $\mathscr{G}\mathfrak{r}:\mathfrak{Cat}\rightarrow\mathfrak{Cat}$
the functor which associates to a category its maximal subgroupoid.
For $\mathfrak{D}$ a category and $\mathfrak{C}$ a small category,
we denote by $\mathbf{Fct}(\mathfrak{C},\mathfrak{D})$ the category
of functors from $\mathfrak{C}$ to $\mathfrak{D}$. We take the convention
that the set of natural numbers $\mathbb{N}$ is the set of nonnegative
integers $\{0,1,2,\ldots\text{\}}$.

\paragraph{Acknowledgements.}

The author wishes to thank most sincerely Christine Vespa, Geoffrey
Powell, Tara Brendle, Antoine Touzé and Nathalie Wahl for their careful
reading, comments, suggestions and expert advice. He would also especially
like to thank Aurélien Djament, Nariya Kawazumi, Gwénaël Massuyeau,
Martin Palmer and Oscar Randal-Williams for the attention they have
paid to this work and helpful discussions. Additionally, he would
like to thank the anonymous referees for their reading and work on
this paper.

\tableofcontents{}

\section{Preliminaries on some categorical tools\label{sec:Recollections-on-Quillen's}}

The aim of this section is to introduce the categorical framework
necessary for our study. In particular, we recall notions and properties
of Quillen's bracket construction introduced in \cite[p.219]{graysonQuillen}
and pre-braided monoidal categories, based on \cite[Section 1]{WahlRandal-Williams}
to which we refer the reader for further details. Also, we recall
a construction for functors from a small category to the category
of small categories, called the \textit{Grothendieck construction}.

Beforehand, we take this opportunity to recall some terminology about
 monoidal categories, referring to \cite{MacLane1} for an introduction
to (braided) strict monoidal categories. A strict monoidal category
will be denoted by $(\mathfrak{C},\natural,0)$, where $\mathfrak{C}$
is a category, $\natural$ is the monoidal product and $0$ is the
monoidal unit. If it is braided, then we denote its braiding by $b_{A,B}^{\mathfrak{C}}:A\natural B\overset{\sim}{\rightarrow}B\natural A$
for all objects $A$ and $B$ of $\mathfrak{C}$. \textbf{We fix a
strict monoidal groupoid $(\mathfrak{G},\natural,0)$ throughout this
section.}

\paragraph{Quillen's bracket construction.}

The following definition is a particular case of a more general construction
of \cite{graysonQuillen} and we refer the reader to \cite[Section 1.1]{WahlRandal-Williams}
for further details.\textit{ Quillen's bracket construction} on the
groupoid $\mathfrak{G}$, denoted by $\mathfrak{UG}$, is the category
with the same objects as $\mathfrak{G}$ and for morphisms
\[
\textrm{Hom}_{\mathfrak{UG}}(A,B)=\textrm{Colim}_{\mathfrak{G}}[\textrm{Hom}_{\mathfrak{G}}(-\natural A,B)]
\]
for $A$ and $B$ two objects of $\mathfrak{G}$. Thus, a morphism
from $A$ to $B$ in the category $\mathfrak{UG}$ is an equivalence
class of pairs $(X,f)$, where $X$ is an object of $\mathfrak{G}$
and $f:X\natural A\rightarrow B$ is a morphism of $\mathfrak{G}$;
we denote this by $[X,f]:A\rightarrow B$: two morphisms $[X,f]$
and $[X',f']$ of $\textrm{Hom}_{\mathfrak{UG}}(A,B)$ are equivalent
if there exists an isomorphism $\psi\in\textrm{Hom}_{\mathfrak{G}}(X,X')$
such that
\begin{equation}
f'\circ(\psi\natural id_{A})=f.\label{eq:equivalence relation}
\end{equation}
For all objects $X$ of $\mathfrak{UG}$, the identity morphism in
the category $\mathfrak{UG}$ is given by $[0,id_{X}]:X\rightarrow X$.
Then, the composition in the category $\mathfrak{UG}$ is defined
by $[Y,g]\circ[X,f]=[Y\natural X,g\circ(id_{Y}\natural f)]$ for $[X,f]:A\rightarrow B$
and $[Y,g]:B\rightarrow C$ two morphisms in the category $\mathfrak{UG}$.
In particular, we note that the unit $0$ of the monoidal structure
is an initial object in the category $\mathfrak{U\mathfrak{G}}$.

There is a relationship between the automorphisms of the groupoid
$\mathfrak{G}$ and those of its associated Quillen's bracket construction
$\mathfrak{U}\mathfrak{G}$. We recall that the strict monoidal groupoid
$(\mathfrak{G},\natural,0)$ is said to have \textit{no zero divisors}
if, for all objects $A$ and $B$ of $\mathfrak{G}$, $A\natural B\cong0$
if and only if $A\cong B\cong0$. In particular, if the strict monoidal
groupoid $(\mathfrak{G},\natural,0)$ has no zero divisors and $\textrm{Aut}_{\mathfrak{G}}(0)=\{id_{0}\}$,
then \cite[Proposition 1.7]{WahlRandal-Williams} proves that the
canonical functor $\mathfrak{c}_{\mathfrak{U}\mathfrak{G}}:\mathfrak{G}\hookrightarrow\mathfrak{UG}$
defined as the identity on objects and by $\mathfrak{c}_{\mathfrak{U}\mathfrak{G}}(\phi)=[0,\phi]$
for all $\phi\in\textrm{Aut}_{\mathfrak{G}}(X)$ is fully faithful.
The maximal subgroupoid $\mathscr{G}\mathfrak{r}(\mathfrak{U\mathfrak{G}})$
of $\mathfrak{U\mathfrak{G}}$ then corresponds to $\mathfrak{G}$
in this situation.\textbf{ Henceforth in this section, we assume that
the strict monoidal groupoid $(\mathfrak{G},\natural,0)$ has no zero
divisors and that $\textrm{Aut}_{\mathfrak{G}}(0)=\{id_{0}\}$.}
\begin{rem}
Let $X$ be an object of $\mathfrak{G}$ and $\phi\in\textrm{Aut}_{\mathfrak{G}}(X)$.
Then, as an element of $\textrm{Hom}_{\mathfrak{UG}}(X,X)$, we abuse
the notation and write $\phi$ for $[0,\phi]$. We also note from
the definition of Quillen's bracket construction that if $(\mathfrak{G},\natural,0)$
is locally small, then so is $\mathfrak{UG}$.
\end{rem}

A natural question is to wonder when an object of $\mathbf{Fct}(\mathfrak{G},\mathscr{C})$
extends to an object of $\mathbf{Fct}(\mathfrak{U}\mathfrak{G},\mathscr{C})$
for a given category $\mathscr{C}$, which is the aim of the following
lemma. Analogous statements can be found in \cite[Proposition 2.4]{WahlRandal-Williams}
and \cite[Lemma 1.12]{soulieLMBilan} (for the category $\mathfrak{U}\boldsymbol{\beta}$
for this last reference).
\begin{lem}
\label{lem:criterionfamilymorphismsfunctor}Let $\mathscr{C}$ be
a category and $F$ an object of $\mathbf{Fct}(\mathfrak{G},\mathscr{C})$.
Assume that there exist a morphism $\eta_{Q,P}:F(P)\rightarrow F(Q\natural P)$
for each pair $(P,Q)$ of objects of $\mathfrak{G}$, such that for
all $A,X,Y\in\textrm{Obj}(\mathfrak{G})$:
\begin{equation}
\eta_{Y,X\natural A}\circ\eta_{X,A}=\eta_{Y\natural X,A}\label{eq:criterion}
\end{equation}
and $\eta_{0,B}=id_{F(B)}$ for all $B\in\textrm{Obj}(\mathfrak{G})$.
Then, the assignments $F([X,\gamma])=F(\gamma)\circ\eta_{X,A}$ to
all morphisms $[X,\gamma]:A\rightarrow X\natural A$ of $\mathfrak{U}\mathfrak{G}$
extend the functor $F:\mathfrak{G}\rightarrow\mathscr{C}$ to a functor
$F:\mathfrak{U}\mathfrak{G}\rightarrow\mathscr{C}$ if and only if
for all $A,X\in\textrm{Obj}(\mathfrak{G})$, for all $\gamma''\in\textrm{Aut}_{\mathfrak{G}}(A)$
and all $\gamma'\in\textrm{Aut}_{\mathfrak{G}}(X)$:
\begin{equation}
F([X,id_{X\natural A}])\circ F(\gamma'')=F(\gamma'\natural\gamma'')\circ F([X,id_{X\natural A}]).\label{eq:criterion'}
\end{equation}
\end{lem}

\begin{proof}
Assume that relation (\ref{eq:criterion'}) is satisfied. Our assignments
imply that $F([0,id_{A}])=id_{F(A)}$ for all objects $A$. First,
let us prove that our assignment conforms with the defining equivalence
relation (\ref{eq:equivalence relation}) of $\mathfrak{UG}$. Let
$A,X\in\textrm{Obj}(\mathfrak{G})$. Let $\gamma,\gamma'\in\textrm{Aut}_{\mathfrak{G}}(X\natural A)$
such that there exists $\psi\in\textrm{Aut}_{\mathfrak{G}}(X)$ such
that $\gamma'\circ(\psi\natural id_{A})=\gamma$. According to the
relation (\ref{eq:criterion'}) and since $F$ is a functor over $\mathfrak{G}$,
we deduce that $F([X,\gamma])=F(\gamma')\circ F([X,id_{X\natural A}])\circ F(id_{A})=F([X,\gamma'])$.
Now, let us check the composition axiom. Let $A,X,Y\in\textrm{Obj}(\mathfrak{G})$,
let $[X,\gamma]\in\textrm{Hom}_{\mathfrak{UG}}(A,X\natural A)$ and
$[Y,\gamma']\in\textrm{Hom}_{\mathfrak{UG}}(X\natural A,Y\natural X\natural A)$.
We deduce from relation (\ref{eq:criterion'}) that $F([Y,\gamma'])\circ F([X,\gamma])=F(\gamma')\circ(F(id_{Y}\natural\gamma)\circ F([Y,id_{Y\natural X\natural A}]))\circ F([X,id_{X\natural A}])$.
So, it follows from relation (\ref{eq:criterion}) that 
\[
F([Y,\gamma'])\circ F([X,\gamma])=F(\gamma'\circ(id_{Y}\natural\gamma))\circ F([Y\natural X,id_{Y\natural X\natural A}])=F([Y,\gamma']\circ[X,\gamma]).
\]

Conversely, assume that the functor $F:\mathfrak{U}\mathfrak{G}\rightarrow\mathscr{C}$
is well-defined. In particular, the composition axiom in $\mathfrak{U}\mathfrak{G}$
is satisfied and implies that for all $A,X\in\textrm{Obj}(\mathfrak{G})$,
for all $\gamma\in\textrm{Aut}_{\mathfrak{G}}(A)$, $F([X,id_{X\natural A}])\circ F(\gamma)=F([X,id_{X}\natural\gamma])$.
So it follows from the defining equivalence relation (\ref{eq:equivalence relation})
of $\mathfrak{UG}$ that relation (\ref{eq:criterion'}) is satisfied.
\end{proof}

\paragraph{Pre-braided monoidal categories.}

If the strict monoidal groupoid $(\mathfrak{G},\natural,0)$ is braided,
Quillen's bracket construction $\mathfrak{UG}$ inherits a strict
monoidal structure; see Proposition \ref{prop:Quillen'sconstructionprebraided}.
However, the braiding $b_{-,-}^{\mathfrak{G}}$ does not extend in
general to $\mathfrak{UG}$.

First, we recall the notion of a pre-braided monoidal category introduced
by Randal-Williams and Wahl in \cite[Definition 1.5]{WahlRandal-Williams},
generalizing that of a strict braided monoidal category. For $(\mathfrak{C},\natural,0)$
a strict monoidal category such that the unit $0$ is initial, we
say that $(\mathfrak{C},\natural,0)$ is \textit{pre-braided} if the
maximal subgroupoid $\mathscr{G}\mathfrak{r}(\mathfrak{C})$ is a
braided monoidal category -- where the monoidal structure is induced
by that of $(\mathfrak{C},\natural,0)$ -- and the braiding associated
with the maximal subgroupoid $b_{A,B}^{\mathscr{G}\mathfrak{r}(\mathfrak{C})}:A\natural B\longrightarrow B\natural A$
satisfies
\begin{equation}
b_{A,B}^{\mathscr{G}\mathfrak{r}(\mathfrak{C})}\circ(id_{A}\natural\iota_{B})=\iota_{B}\natural id_{A}:A\longrightarrow B\natural A\label{eq:defbraid}
\end{equation}
for all objects $A$ and $B$ of $\mathfrak{C}$ (recall that $\iota_{B}:0\rightarrow B$
denotes the unique morphism from $0$ to $B$).
\begin{rem}
A braided monoidal category is automatically pre-braided. However,
a pre-braided monoidal category is not necessarily braided; see \cite[Remark 5.24]{WahlRandal-Williams}
or \cite[Remark 1.15]{soulieLMBilan}. Namely, for a pre-braided monoidal
category, the opposite of condition (\ref{eq:defbraid}), that is
$b_{A,B}^{\mathfrak{C}}\circ(\iota_{B}\natural id_{A})=id_{A}\natural\iota_{B}$,
does not hold generally speaking, whereas this is a necessary property
for a braided monoidal category.
\end{rem}

Finally, let us give the following key property when Quillen's bracket
construction is applied on a strict braided monoidal groupoid $(\mathfrak{G},\natural,0,b_{-,-}^{\mathfrak{G}})$.
\begin{prop}
\cite[Proposition 1.8]{WahlRandal-Williams}\label{prop:Quillen'sconstructionprebraided}
If the groupoid $(\mathfrak{G},\natural,0)$ is braided, then the
category $(\mathfrak{UG},\natural,0)$ is pre-braided monoidal. Namely,
the monoidal structure on the category $(\mathfrak{UG},\natural,0)$
is defined on objects by that of $(\mathfrak{G},\natural,0)$ and
on morphisms by letting
\begin{equation}
[X,f]\natural[Y,g]=[X\natural Y,(f\natural g)\circ(id_{X}\natural(b_{A,Y}^{\mathfrak{G}})^{-1}\natural id_{C})]\label{eq:def_monoidal_UG}
\end{equation}
for $[X,f]\in\textrm{Hom}_{\mathfrak{UG}}(A,B)$ and $[Y,g]\in\textrm{Hom}_{\mathfrak{UG}}(C,D)$.
In particular, the canonical functor $\mathfrak{c}_{\mathfrak{U}\mathfrak{G}}:\mathfrak{G}\hookrightarrow\mathfrak{UG}$
is monoidal.
\end{prop}

\paragraph{The Grothendieck construction.}

We present here the Grothendieck construction for a functor from a
small category to the category of small categories. We refer the reader
to \cite[Chapter 1, Section 5]{MacLaneMoerdijk} for further details.

Let $\mathfrak{C}$ be a \textit{small} category and $\mathcal{A}:\mathfrak{C}\rightarrow\mathfrak{Cat}$
a functor. The \textit{Grothendieck construction} for $\mathcal{A}$,
also known as the category of elements of $\mathcal{A}$ and denoted
by $\intop^{\mathfrak{C}}\mathcal{A}$, is defined as follows. The
objects are pairs $(x,c)$ where $c\in\textrm{Obj}(\mathfrak{C})$
and $x\in\textrm{Obj}(\mathcal{A}(c))$, and a morphism from $(x,c)$
to $(x',c')$ in $\intop^{\mathfrak{C}}\mathcal{A}$ is a pair $(\alpha,f)$
where $f\in\textrm{Hom}_{\mathfrak{C}}(c,c')$ and $\alpha\in\textrm{Hom}_{\mathcal{A}(c')}(\mathcal{A}(f)(x),x')$.
For $(\alpha,f)\in\textrm{Hom}_{\intop^{\mathfrak{C}}\mathcal{A}}((x_{1},c_{1}),(x_{2},c_{2}))$
and $(\beta,g)\in\textrm{Hom}_{\intop^{\mathfrak{C}}\mathcal{A}}((x_{2},c_{2}),(x_{3},c_{3}))$,
the composition in the category $\intop^{\mathfrak{C}}\mathcal{A}$
is defined by
\[
(\beta,g)\circ(\alpha,f)=(\beta\circ\mathcal{A}(g)(\alpha),g\circ f).
\]
There is a canonical projection functor $\intop^{\mathfrak{C}}\mathcal{A}\rightarrow\mathfrak{C}$,
given by sending an object $(x,c)$ to $c$.

Recall that $\mathfrak{Gr}$ denotes the category of groups. For a
functor $\mathcal{A}:\mathfrak{C}\rightarrow\mathfrak{Gr}$, the associated
Grothendieck construction $\intop^{\mathfrak{C}}\mathcal{A}$ is then
defined by considering a group $G$ as a category with one object
denoted by $\cdot_{G}$. Denoting by $0$ the functor $\mathfrak{C}\rightarrow\mathfrak{Gr}$
sending all $c\in\textrm{Obj}(\mathfrak{C})$ to the trivial group
$0_{\mathfrak{Gr}}$, there exists a unique natural transformation
$0\rightarrow\mathcal{A}$. Applying the Grothendieck construction,
this induces a section $\mathfrak{C}=\intop^{\mathfrak{C}}0\hookrightarrow\intop^{\mathfrak{C}}\mathcal{A}$
to the projection functor $\intop^{\mathfrak{C}}\mathcal{A}\rightarrow\mathfrak{C}$.
We denote this section functor by $\mathfrak{s}_{\mathcal{A}}$.

With the category of groups as the target category of $\mathcal{A}$,
the associated Grothendieck construction actually encodes a semidirect
product structure:
\begin{lem}
\label{exa:casegroupcategory} For a functor $\mathcal{A}:\mathfrak{C}\rightarrow\mathfrak{Gr}$
where $\mathfrak{C}$ is a small groupoid, then $\intop^{\mathfrak{C}}\mathcal{A}$
is a groupoid where the automorphism group of any object $c$ is the
semidirect product $\mathcal{A}(c)\rtimes\textrm{Aut}_{\mathfrak{C}}(c)$
induced by $\mathcal{A}$.
\end{lem}

\begin{proof}
Considering such a functor $\mathcal{A}$ is equivalent to defining
a group homomorphism $\mathcal{A}_{c}:\textrm{Aut}_{\mathfrak{C}}(c)\rightarrow\textrm{Aut}_{\mathfrak{Gr}}(\mathcal{A}(c))$
for each object $c$. It follows from the definition of the Grothendieck
construction that all the morphisms of $\intop^{\mathfrak{C}}\mathcal{A}$
are invertible, and that there is a clear group isomorphism $\textrm{Aut}_{\intop^{\mathfrak{C}}\mathcal{A}}((\cdot_{\mathcal{A}(c)},c))\cong\mathcal{A}(c)\rtimes_{\mathcal{A}_{c}}\textrm{Aut}_{\mathfrak{C}}(c)$,
the composition in the category $\intop^{\mathfrak{C}}\mathcal{A}$
following exactly the same rule as the group operation of a semidirect
product.
\end{proof}

\section{The generalized Long-Moody functors\label{sec:The-generalized-Long-Moody}}

In this section, we introduce the notion of Long-Moody functors for
an abstract family of groups, inspired by the Long-Moody construction
for braid groups of \cite[Theorem 2.1]{Long1}. First, we introduce
a general construction in Section \ref{subsec:Framework-of-the},
called \textit{the tensorial construction}, using tensor product of
functors and the required tools for our study. We then define \textit{the
generalized Long-Moody functors} and establish some of their first
properties in Section \ref{subsec:Functoriality-of-the}. In addition
to recovering all the results of \cite[Section 2]{soulieLMBilan},
we give a new approach to the tools and conditions previously considered
in \cite{soulieLMBilan}, allowing a deeper understanding of these
constructions and a wider application.

\subsection{A general construction\label{subsec:Framework-of-the}}

In this first subsection, we present a general construction based
on a tensor product for functor categories. The generalized Long-Moody
functors introduced in Section \ref{subsec:Functoriality-of-the}
are particular cases of this construction. We refer the reader to
\cite[Section VII.3]{MacLane1} for the notions of monoid objects
and modules in a monoidal category, which will be used in this section.

\paragraph{Tensor product over a monoid functors.}

First, let us introduce the notion of tensor product over a monoid
functor. \textbf{We fix a small category $\mathfrak{C}$ throughout
Section \ref{subsec:Framework-of-the}.}

Let $\otimes_{R}$ be the pointwise tensor product in the functor
category $\mathbf{Fct}(\mathfrak{C},R\textrm{-}\mathfrak{Mod})$ and
let $R$ denote the constant functor at $R$. These endow $\mathbf{Fct}(\mathfrak{C},R\textrm{-}\mathfrak{Mod})$
with a strict monoidal structure $(\mathbf{Fct}\left(\mathfrak{C},R\textrm{-}\mathfrak{Mod}\right),\otimes_{R},R)$.
Let $\mathcal{M}$ be a monoid object in $\mathbf{Fct}(\mathfrak{C},R\textrm{-}\mathfrak{Mod})$.
We denote by $\mathcal{M}\textrm{-}\mathfrak{Mod}$ and $\mathfrak{Mod}\textrm{-}\mathcal{M}$
the categories of left and right modules in $\mathbf{Fct}(\mathfrak{C},R\textrm{-}\mathfrak{Mod})$
over $\mathcal{M}$ respectively. The \textit{tensor product over
$\mathcal{M}$ functor} $-\otimes_{\mathcal{M}}-:\mathfrak{Mod}\textrm{-}\mathcal{M}\times\mathcal{M}\textrm{-}\mathfrak{Mod}\rightarrow\mathbf{Fct}(\mathfrak{C},R\textrm{-}\mathfrak{Mod})$
is defined by:

$\bullet$ Objects: for $F\in\textrm{Obj}(\mathfrak{Mod}\textrm{-}\mathcal{M})$
and $G\in\textrm{Obj}(\mathcal{M}\textrm{-}\mathfrak{Mod})$, denoting
$\rho_{F}$ and $\lambda_{G}$ the natural transformation actions
of $\mathcal{M}$ on $F$ and $G$ respectively, $F\otimes_{\mathcal{M}}G:\mathfrak{C}\rightarrow R\textrm{-}\mathfrak{Mod}$
is the coequalizer of the natural transformations $\rho_{F}\otimes_{R}id_{G}$
and $id_{F}\otimes_{R}\lambda_{G}$.

$\bullet$ Morphisms: let $F_{1}$ and $F_{2}$ be two objects of
$\mathfrak{Mod}\textrm{-}\mathcal{M}$ and $f:F_{1}\rightarrow F_{2}$
be a natural transformation in $\mathfrak{Mod}\textrm{-}\mathcal{M}$;
let $G_{1}$ and $G_{2}$ be two objects of $\mathcal{M}\textrm{-}\mathfrak{Mod}$
and $g:G_{1}\rightarrow G_{2}$ be a natural transformation in $\mathcal{M}\textrm{-}\mathfrak{Mod}$.
We define $f\otimes_{\mathcal{M}}g:F_{1}\otimes_{\mathcal{M}}G_{1}\rightarrow F_{2}\otimes_{\mathcal{M}}G_{2}$
to be the unique morphism induced from $f\otimes_{R}g:F_{1}\otimes_{R}G_{1}\rightarrow F_{2}\otimes_{R}G_{2}$
by the universal property of the coequalizer $F_{1}\otimes_{\mathcal{M}}G_{1}$.

The functor $-\otimes_{\mathcal{M}}-$ is called \textit{the tensor
product functor over $\mathcal{M}$}. In particular, fixing an object
$F$ of $\mathfrak{Mod}\textrm{-}\mathcal{M}$ defines a functor $F\otimes_{\mathcal{M}}-:\mathcal{M}\textrm{-}\mathfrak{Mod}\rightarrow\mathbf{Fct}(\mathfrak{C},R\textrm{-}\mathfrak{Mod})$.

\paragraph{Group algebra and augmentation ideal functors.}

We recall that $\mathfrak{Gr}$ denotes the category of groups.\textbf{
From now on, we fix a functor $\mathcal{A}:\mathfrak{C}\rightarrow\mathfrak{Gr}$
for the remainder of Section \ref{subsec:Framework-of-the}.}

Let $R\textrm{-}\mathfrak{Alg}$ be the category of unital $R$-algebras.
For all objects $G$ of $\mathfrak{Gr}$, the group rings $R[G]$
and augmentation ideals $\mathcal{I}_{R[G]}$ respectively assemble
to define the group algebra functor $R[-]:\mathfrak{Gr}\rightarrow R\textrm{-}\mathfrak{Alg}$
and the augmentation ideal functor $\mathcal{I}_{R[-]}:\mathfrak{Gr}\rightarrow R\textrm{-}\mathfrak{Mod}$.
Let $R[\mathcal{A}]$ be the composition functor $R[-]\circ\mathcal{A}:\mathfrak{C}\rightarrow R\textrm{-}\mathfrak{Alg}$,
called the group algebra functor induced by $\mathcal{A}$. Similarly,
let $\mathcal{I}_{R[\mathcal{A}]}$ be the composition functor $\mathcal{I}_{R[-]}\circ\mathcal{A}:\mathfrak{C}\rightarrow R\textrm{-}\mathfrak{Mod}$,
called the augmentation ideal functor induced by $\mathcal{A}$. The
unital $R$-algebra structures of $R[\mathcal{A}(c)]$ for all $c\in\textrm{Obj}(\mathfrak{C})$
induce an associative unital monoid object structure on $R[\mathcal{A}]$
with respect to the monoidal structure $(\mathbf{Fct}(\mathfrak{C},R\textrm{-}\mathfrak{Mod}),\otimes_{R},R)$.
Therefore:
\begin{lem}
\label{lem:The-augmentation-ideal_right_A}The augmentation ideal
functor $\mathcal{I}_{R[\mathcal{A}]}$ is a right $R[\mathcal{A}]$-module.
\end{lem}

\begin{proof}
The natural transformation $\mathcal{I}_{R[\mathcal{A}]}\otimes_{R}R[\mathcal{A}]\rightarrow\mathcal{I}_{R[\mathcal{A}]}$
is induced by the right $R[\mathcal{A}(c)]$-module structure of the
augmentation ideal $\mathcal{I}_{R[\mathcal{A}(c)]}$ for each object
$c$ of $\mathfrak{C}$, the associativity and unit axioms of a module
over a monoid object being straightforward to check.
\end{proof}

\paragraph{The tensorial construction.}

We present now a general construction for functor categories, using
a tensor product functor. First of all, we have the following key
property:
\begin{prop}
\label{prop:keypropertygroth}The precomposition by the section $\mathfrak{s}_{\mathcal{A}}:\mathfrak{C}\hookrightarrow\intop^{\mathfrak{C}}\mathcal{A}$
induces an equivalence of categories $\mathfrak{s}_{\mathcal{A}}^{*}:\mathbf{Fct}(\intop^{\mathfrak{C}}\mathcal{A},R\textrm{-}\mathfrak{Mod})\cong R[\mathcal{A}]\textrm{-}\mathfrak{Mod}.$
\end{prop}

\begin{proof}
Let $F$ be an object of $\mathbf{Fct}(\intop^{\mathfrak{C}}\mathcal{A},R\textrm{-}\mathfrak{Mod})$.
We consider the functor $\mathfrak{s}_{\mathcal{A}}^{*}(F):\mathfrak{C}\to R\textrm{-}\mathfrak{Mod}$.
We recall that $e_{G}$ denotes the unit element of a group $G$.
For $c$ and $c'$ two objects of $\mathfrak{C}$, a morphism from
$(\cdot_{\mathcal{A}(c)},c)$ to $(\cdot_{\mathcal{A}(c')},c')$ in
$\intop^{\mathfrak{C}}\mathcal{A}$ is of the form $(x,\varphi)=(x,id_{c'})\circ(e_{\mathcal{A}(c')},\varphi)$
where $x\in\mathcal{A}(c')$ and $\varphi\in\textrm{Hom}_{\mathfrak{C}}(c,c')$.
We define a left $R\left[\mathcal{A}\right]$-module structure natural
transformation $\lambda_{\mathfrak{s}_{\mathcal{A}}^{*}(F)}:R[\mathcal{A}]\otimes_{R}\mathfrak{s}_{\mathcal{A}}^{*}(F)\rightarrow\mathfrak{s}_{\mathcal{A}}^{*}(F)$
as follows. For each $c\in\textrm{Obj}(\mathfrak{C})$, we define
the morphism $\lambda_{\mathfrak{s}_{\mathcal{A}}^{*}(F)(c)}:R[\mathcal{A}(c)]\otimes_{R}\mathfrak{s}_{\mathcal{A}}^{*}(F)(c)\rightarrow\mathfrak{s}_{\mathcal{A}}^{*}(F)(c)$
by
\[
\lambda_{\mathfrak{s}_{\mathcal{A}}^{*}(F)}(y\otimes_{R}v)=F(y,id_{c})(v)
\]
for all $y\in\mathcal{A}(c)$ and $v\in F(\mathfrak{s}_{\mathcal{A}}(c))$.
We recall from the composition rule in $\intop^{\mathfrak{C}}\mathcal{A}$
that $(e_{\mathcal{A}(c')},\varphi)\circ(y,id_{c})=(\mathcal{A}(\varphi)(y),\varphi)$.
Then, since $F$ is a functor over the category $\intop^{\mathfrak{C}}\mathcal{A}$,
we deduce that for each $\varphi\in\textrm{Hom}_{\mathfrak{C}}(c,c')$:
\[
(\lambda_{\mathfrak{s}_{\mathcal{A}}^{*}(F)(c')}\circ(R[\mathcal{A}(\varphi)]\otimes_{R}\mathfrak{s}_{\mathcal{A}}^{*}(F)(\varphi))(y\underset{R}{\otimes}v)=F((\mathcal{A}(\varphi)(y),\varphi))(v)=(\mathfrak{s}_{\mathcal{A}}^{*}(F)(\varphi)\circ\lambda_{\mathfrak{s}_{\mathcal{A}}^{*}(F)(c)})(y\underset{R}{\otimes}v).
\]
Hence, the morphisms $\lambda_{\mathfrak{s}_{\mathcal{A}}^{*}(F)(c)}$
define a natural transformation and the axioms of an action are easily
checked. The naturality with respect to $F$ follows straightforwardly
from these assignments.

Conversely, let $G$ be a left $R[\mathcal{A}]$-module and we denote
by $\lambda_{G}:R[\mathcal{A}]\otimes_{R}G\rightarrow G$ its associated
left module natural transformation. We extend $G$ to a functor $\widehat{G}$
with $\intop^{\mathfrak{C}}\mathcal{A}$ as source category as follows.
For each $c\in\textrm{Obj}(\mathfrak{C})$, $\widehat{G}$ sends the
object $(\cdot_{\mathcal{A}(c)},c)$ to $G(c)$. Then, for all $(x,\varphi)$
where $x\in\mathcal{A}(c')$ and $\varphi\in\textrm{Hom}_{\mathfrak{C}}(c,c')$,
we define $\widehat{G}((x,\varphi))$ to be the composition $G(c)\to R[A(c')]\otimes_{R}G(c')\to G(c')$
defined by $\lambda_{G(c')}\circ(x\otimes_{R}G(\varphi))$. Considering
another $(x',\varphi')$ where $x'\in\mathcal{A}(c'')$ and $\varphi'\in\textrm{Hom}_{\mathfrak{C}}(c',c'')$,
since $\lambda_{G}$ is a natural transformation, we compute that
$G(\varphi')\circ\lambda_{G(c')}\circ(x\otimes_{R}G(\varphi))=\lambda_{G(c'')}\circ(\mathcal{A}(\varphi)(x)\otimes_{R}G(\varphi'\circ\varphi))$
and deduce that
\[
\widehat{G}((x',\varphi'))\circ\widehat{G}((x,\varphi))=\lambda_{G(c'')}\circ(x'\mathcal{A}(\varphi)(x)\otimes_{R}G(\varphi'\circ\varphi))=\widehat{G}((x',\varphi')\circ(x,\varphi)).
\]
Therefore the assignment for $\widehat{G}$ satisfies the composition
axiom on $\intop^{\mathfrak{C}}\mathcal{A}$ and the identity axiom
is easily checked. Also, the naturality with respect to $G$ follows
from the naturality with respect to $\lambda_{G}$ of a functor between
two $R[\mathcal{A}]$-modules. Then, it follows from these definitions
that $\widehat{-}$ defines a functor from $R[\mathcal{A}]\textrm{-}\mathfrak{Mod}$
to $\mathbf{Fct}(\intop^{\mathfrak{C}}\mathcal{A},R\textrm{-}\mathfrak{Mod})$
and that is the inverse of $\mathfrak{s}_{\mathcal{A}}^{*}$.
\end{proof}
Now, we can introduce the construction:
\begin{defn}
\label{def:twistingbytheaugmentationidealfunctorconstruction} Let
$\chi:\intop^{\mathfrak{C}}\mathcal{A}\rightarrow\mathfrak{D}$ be
a functor where $\mathfrak{D}$ is another small category and we denote
by $\chi^{*}$ the precomposition functor induced by $\chi$. We define
$\mathfrak{T}_{\mathcal{A},\chi}$ to be the composition:
\[
\xymatrix{\mathbf{Fct}(\mathfrak{D},R\textrm{-}\mathfrak{Mod})\ar@{->}[rr]^{\,\,\,\,\,\,\,\,\mathfrak{s}_{\mathcal{A}}^{*}\circ\chi^{*}} &  & R[\mathcal{A}]\textrm{-}\mathfrak{Mod}\ar@{->}[rr]^{\mathcal{I}_{R[\mathcal{A}]}\otimes_{R[\mathcal{A}]}-\,\,\,\,\,\,\,\,\,\,} &  & \mathbf{Fct}(\mathfrak{C},R\textrm{-}\mathfrak{Mod}).}
\]
The functor $\mathfrak{T}_{\mathcal{A},\chi}$ is called \textit{the
tensorial construction} by the functor $\mathcal{A}$ along $\chi$.
\end{defn}

We take this opportunity to spell out how a morphism in $\mathfrak{C}$
acts on $\mathfrak{T}_{\mathcal{A},\chi}$. The following result is
a straightforward consequence of the definition of $\mathfrak{T}_{\mathcal{A},\chi}$:
\begin{lem}
\label{lem:explicit_def_T}We take up the notations of Definition
\ref{def:twistingbytheaugmentationidealfunctorconstruction}. Let
$F$ be an object of $\mathbf{Fct}(\mathfrak{D},R\textrm{-}\mathfrak{Mod})$
and $\varphi\in\textrm{Hom}_{\mathfrak{C}}(c,c')$. Then $\mathfrak{T}_{\mathcal{A},\chi}(F)(\varphi)$
is the morphism from $\mathcal{I}_{R[\mathcal{A}(c)]}\otimes_{R[\mathcal{A}(c)]}F((\chi\circ\mathfrak{s}_{\mathcal{A}})(c))$
to $\mathcal{I}_{R[\mathcal{A}(c')]}\otimes_{R[\mathcal{A}(c')]}F((\chi\circ\mathfrak{s}_{\mathcal{A}})(c'))$
defined by
\[
\mathfrak{T}_{\mathcal{A},\chi}(F)(\varphi)(i\underset{R[\mathcal{A}(c)]}{\varotimes}v)=\mathcal{I}_{R[\mathcal{A}]}(\varphi)(i)\underset{R[\mathcal{A}(c')]}{\varotimes}F((\chi\circ\mathfrak{s}_{\mathcal{A}})(\varphi))(v)
\]
for all $i\in\mathcal{I}_{R[\mathcal{A}(c)]}$ and $v\in F((\chi\circ\mathfrak{s}_{\mathcal{A}})(c))$.

Also, let $G$ be another objects of $\mathbf{Fct}(\mathfrak{D},R\textrm{-}\mathfrak{Mod})$
and $\eta:F\rightarrow G$ be a natural transformation. Then, the
natural transformation $\mathfrak{T}_{\mathcal{A},\chi}(\eta):\mathfrak{T}_{\mathcal{A},\chi}(F)\rightarrow\mathfrak{T}_{\mathcal{A},\chi}(G)$
is given for all $c\in\textrm{Obj}(\mathfrak{C})$ by assigning the
morphism $\mathfrak{T}_{\mathcal{A},\chi}(\eta_{d})_{c}$ to be the
tensor product morphism $id_{\mathcal{I}_{R[\mathcal{A}(c)]}}\otimes_{R[\mathcal{A}(c)]}\eta_{(\chi\circ\mathfrak{s}_{\mathcal{A}})(c)}$.
\end{lem}

Let us give some immediate properties of a tensorial construction.
\begin{prop}
\label{prop:exactnessLM}The tensorial construction $\mathfrak{T}_{\mathcal{A},\chi}$
is additive, right-exact and commutes with all colimits.
\end{prop}

\begin{proof}
Let $0:\mathfrak{E}\rightarrow R\textrm{-}\mathfrak{Mod}$ denote
the null functor for a small category $\mathfrak{E}$. It follows
from the definition that $\mathfrak{T}_{\mathcal{A},\chi}\left(0\right)=0$,
and so additivity follows from the commutation with all colimits.

Since the precomposition functors $\chi^{*}$ and $\mathfrak{s}_{\mathcal{A}}^{*}$
are exact, it is enough to prove the results for the functor $\mathcal{I}_{\mathcal{A}}\otimes_{R[\mathcal{A}]}-$.
As a consequence of the properties of the tensor product of modules,
the functor $\mathcal{I}_{R[\mathcal{A}(c)]}\otimes_{R[\mathcal{A}(c)]}-:R[\mathcal{A}(c)]\textrm{-}\mathfrak{Mod}\rightarrow R\textrm{-}\mathfrak{Mod}$
is additive, right-exact and commutes with all colimits for all $c\in\textrm{Obj}(\mathfrak{C})$.
Then the naturality of each of these properties with respect to the
morphisms of $\mathfrak{C}$ is a formal consequence of the definition
of $\mathcal{I}_{\mathcal{A}}\otimes_{R[\mathcal{A}]}-$.
\end{proof}

\subsection{The Long-Moody functors\label{subsec:Functoriality-of-the}}

Using the tensorial construction of Section \ref{subsec:Framework-of-the},
we introduce here the generalized Long-Moody functors, inspired from
the Long-Moody construction \cite{Long1}. While the original construction
was associated with braid groups, the following framework encompasses
a much broader class of groups; see Section \ref{subsec:Examples}. 

\paragraph{Categorical framework.}

First, we require the following categorical framework to define generalized
Long-Moody functors. Let $(\mathcal{G}',\natural,0_{\mathcal{G}'},b_{-,-}^{\mathcal{G}'})$
be a locally small strict braided monoidal groupoid with no zero divisors
and such that $\textrm{Aut}_{\mathcal{G}'}(0_{\mathcal{G}'})=\{id_{0_{\mathcal{G}'}}\}$.
We recall from Proposition \ref{prop:Quillen'sconstructionprebraided}
that Quillen's bracket construction $(\mathfrak{U}\mathcal{G}',\natural,0_{\mathcal{G}'})$
is a locally small pre-braided monoidal category such that the unit
$0_{\mathcal{G}'}$ is an initial object. Let $\underline{0}$ and
$1$ be two objects of $\mathcal{G}'$, such that $1$ is not isomorphic
to $0_{\mathcal{G}'}$.
\begin{notation}
For all natural numbers $n$, we denote the object $1^{\natural n}\natural\underline{0}$
of $\mathcal{G}'$ by $\underline{n}$ and the object $1^{\natural n}$
of $\mathcal{G}'$ by $n$. Note that $m\natural\underline{n}=\underline{m+n}$
for all natural numbers $m$ and $n$.
\end{notation}

\begin{defn}
\label{def:UGfullsbcatUG'} Let $\mathcal{G}$ be the small full subgroupoid
of $\mathcal{G}'$ on the objects $\{\underline{n}\}_{n\in\mathbb{N}}$.
We denote by $G_{n}$ the automorphism group $\textrm{Aut}_{\mathcal{G}}(\underline{n})$
for all $n$. For convenience, we denote the small full subcategory
of Quillen's bracket construction $\mathfrak{U}\mathcal{G}'$ on the
objects $\{\underline{n}\}_{n\in\mathbb{N}}$ by $\mathfrak{U}\mathcal{G}$.
Let $1\natural-:\mathfrak{U}\mathcal{G}\rightarrow\mathfrak{U}\mathcal{G}$
be the functor defined by $(1\natural-)(\underline{n})=1\natural\underline{n}$
for all $\underline{n}\in\textrm{Obj}(\mathcal{G})$ and $(1\natural-)([n'-n,g])=id_{1}\natural[n'-n,g]$
for all morphism $[n'-n,g]$ of $\mathfrak{U}\mathcal{G}$.
\end{defn}

\textbf{Warning:} the category $\mathfrak{U}\mathcal{G}$ is not in
general Quillen's bracket construction introduced in Section \ref{sec:Recollections-on-Quillen's}
applied to $\mathcal{G}$, but Quillen's bracket construction applied
to the ambient groupoid $\mathcal{G}'$. We choose this convention
to avoid too heavy notations. Note that $\mathfrak{U}\mathcal{G}$
is Quillen's bracket construction on $\mathcal{G}$ if and only if
$0_{\mathcal{G}'}=\underline{0}$. The present framework allows one
to handle families of groups such as mapping class groups of surfaces;
see Sections \ref{subsec:Modifyingpunctures0} and \ref{subsec:Modifyingthegenus}.

\textbf{We fix a braided monoidal groupoid $\mathcal{G}'$, the small
full subgroupoid $\mathcal{G}$ of Definition \ref{def:UGfullsbcatUG'}
and a functor $\mathcal{A}:\mathfrak{U}\mathcal{G}\rightarrow\mathfrak{Gr}$
for the remainder of Section \ref{subsec:Functoriality-of-the}.}

\subsubsection{Definition of the Long-Moody functors\label{subsec:Definition-of-theLM}}

The idea to define a Long-Moody functor is to use the tensorial construction
by the functor $\mathcal{A}$ along some $\chi$, such that the composition
the section $\mathfrak{s}_{\mathcal{A}}$ is the functor $1\natural-$.
We have the choice to consider $\mathcal{A}$ over $\mathcal{G}$
or $\mathfrak{U}\mathcal{G}$. Actually, this choice leads to an extension
problem: the restriction along the canonical functor $\mathfrak{c}_{\mathfrak{U}\mathcal{G}}:\mathcal{G}\hookrightarrow\mathfrak{U}\mathcal{G}$
defines the Grothendieck construction $\intop^{\mathcal{G}}\mathcal{A}$
together with an inclusion functor $\intop^{\mathcal{G}}\mathcal{A}\hookrightarrow\intop^{\mathfrak{U}\mathcal{G}}\mathcal{A}$,
so that the following diagram is commutative
\[
\xymatrix{\mathcal{G}\ar@{^{(}->}[r]\ar@{^{(}->}[d] & \mathfrak{U}\mathcal{G}\ar@{^{(}->}[d]\\
\intop^{\mathcal{G}}\mathcal{A}\ar@{^{(}->}[r] & \intop^{\mathfrak{U}\mathcal{G}}\mathcal{A}.
}
\]
This is not a pushout of categories: given $\mathcal{G}\hookrightarrow\intop^{\mathcal{G}}\mathcal{A}\rightarrow\mathfrak{E}$
a functor for $\mathfrak{E}$ a small category, we deal with the extension
problem
\[
\xymatrix{\mathcal{G}\ar@{^{(}->}[rr]\ar@{^{(}->}[d] &  & \mathfrak{U}\mathcal{G}\ar@{^{(}->}[d]\ar@{->}[ddr]\\
\intop^{\mathcal{G}}\mathcal{A}\ar@{^{(}->}[rr]\ar@{->}[drrr] &  & \intop^{\mathfrak{U}\mathcal{G}}\mathcal{A}\ar@{.>}[dr]\\
 &  &  & \mathfrak{E}.
}
\]
In the present situation, $\mathfrak{E}$ is taken to be $\mathfrak{U}\mathcal{G}$
and the composites $\mathcal{G}\rightarrow\intop^{\mathcal{G}}\mathcal{A}\rightarrow\mathfrak{E}$
and $\mathfrak{U}\mathcal{G}\rightarrow\intop^{\mathfrak{U}\mathcal{G}}\mathcal{A}\rightarrow\mathfrak{E}$
to be $1\natural-$. This motivates the following:
\begin{defn}
\label{def:longmoodysystem} The setting $\{\mathcal{A},\mathcal{G},\mathcal{G}',\chi\}$
is a \textit{Long-Moody system} if $\chi:\intop^{\mathcal{G}}\mathcal{A}\rightarrow\mathcal{G}$
is a functor such that the following diagram is commutative:
\begin{equation}
\xymatrix{\mathcal{G}\,\ar@{^{(}->}[r]^{\mathfrak{s}_{\mathcal{A}}}\ar@{->}[dr]_{1\natural-} & \intop^{\mathcal{G}}\mathcal{A}\ar@{->}[d]^{\chi}\\
 & \mathcal{G}.
}
\label{cond:coherenceconditionsigmanan}
\end{equation}
The functor $\mathcal{A}$ equipped with such functor $\chi$ is said
to define a Long-Moody system, denoted by $\{\mathcal{A},\mathcal{G},\mathcal{G}',\chi\}$.
If $\chi$ extends along the inclusion $\intop^{\mathcal{G}}\mathcal{A}\hookrightarrow\intop^{\mathfrak{U}\mathcal{G}}\mathcal{A}$
to define a functor $\chi:\intop^{\mathfrak{U}\mathcal{G}}\mathcal{A}\rightarrow\mathfrak{U}\mathcal{G}$,
such that the following diagram is commutative:
\begin{equation}
\xymatrix{\mathfrak{U}\mathcal{G}\,\ar@{^{(}->}[r]^{\mathfrak{s}_{\mathcal{A}}}\ar@{->}[dr]_{1\natural-} & \intop^{\mathfrak{U}\mathcal{G}}\mathcal{A}\ar@{->}[d]^{\chi}\\
 & \mathfrak{U}\mathcal{G},
}
\label{cond:coherenceconditionsigmanan2}
\end{equation}
then the Long-Moody system $\{\mathcal{A},\mathcal{G},\mathcal{G}',\chi\}$
is then said to be \textit{coherent}.
\end{defn}

Then, we can introduce the main concept of Section \ref{sec:The-generalized-Long-Moody}:
\begin{defn}
\textit{\label{Thm:LMFunctor}} The Long-Moody functor associated
with the Long-Moody system (respectively coherent Long-Moody system)
$\{\mathcal{A},\mathcal{G},\mathcal{G}',\chi\}$, denoted by $\mathbf{LM}_{\{\mathcal{A},\mathcal{G},\mathcal{G}',\chi\}}$
(respectively $\mathbf{LM}_{\{\mathcal{A},\mathcal{G},\mathcal{G}',\chi\}}^{\mathfrak{U}}$),
is the  tensorial construction $\mathfrak{T}_{\mathcal{A},\chi}$
by the functors $\mathcal{A}:\mathfrak{U}\mathcal{G}\rightarrow\mathfrak{Gr}$
along $\chi:\intop^{\mathcal{G}}\mathcal{A}\rightarrow\mathcal{G}$
(respectively along $\chi:\intop^{\mathfrak{U}\mathcal{G}}\mathcal{A}\rightarrow\mathfrak{U}\mathcal{G}$)
of Definition \ref{def:twistingbytheaugmentationidealfunctorconstruction}.
\end{defn}

\begin{notation}
When there is no ambiguity, once the Long-Moody system $\{\mathcal{A},\mathcal{G},\mathcal{G}',\chi\}$
is fixed, we omit it from the notation. If the Long-Moody system $\{\mathcal{A},\mathcal{G},\mathcal{G}',\chi\}$
is coherent, we omit it the $\mathfrak{U}$ from the notation if there
is no risk of confusion. In this case, we denote both the endofunctor
of $\mathbf{Fct}(\mathcal{G},R\textrm{-}\mathfrak{Mod})$ and the
one of $\mathbf{Fct}(\mathfrak{U}\mathcal{G},R\textrm{-}\mathfrak{Mod})$
by $\mathbf{LM}$ for simplicity. Indeed, since the clear diagram
relating them -- induced by the precomposition functor $\mathfrak{c}_{\mathfrak{U}\mathfrak{G}}^{*}:\mathbf{Fct}(\mathfrak{U}\mathcal{G},R\textrm{-}\mathfrak{Mod})\rightarrow\mathbf{Fct}(\mathcal{G},R\textrm{-}\mathfrak{Mod})$
-- is commutative.
\end{notation}

For sake of clarity, let us spell out the explicit definition of the
Long-Moody functor associated to a coherent system $\{\mathcal{A},\mathcal{G},\mathcal{G}',\chi\}$.
For an object $F$ if $\mathbf{Fct}(\mathfrak{U}\mathcal{G},R\textrm{-}\mathfrak{Mod})$,
for all objects $\underline{n}$ of $\mathcal{G}$ we have
\[
\mathbf{LM}_{\{\mathcal{A},\mathcal{G},\mathcal{G}',\chi\}}(F)(\underline{n})=\mathcal{I}_{R[\mathcal{A}(\underline{n})]}\underset{R[\mathcal{A}(\underline{n})]}{\otimes}F(id_{1}\natural\underline{n}).
\]
Also, the actions of $\mathbf{LM}_{\{\mathcal{A},\mathcal{G},\mathcal{G}',\chi\}}(F)$
on morphisms of $\mathbf{Fct}(\mathfrak{U}\mathcal{G},R\textrm{-}\mathfrak{Mod})$
and of $\mathbf{LM}_{\{\mathcal{A},\mathcal{G},\mathcal{G}',\chi\}}$
on morphisms of $\mathfrak{U}\mathcal{G}$ are induced by Lemma \ref{lem:explicit_def_T}.

Non-trivial coherent Long-Moody systems arise naturally in many situations
in connection with topology; see Section \ref{subsec:Examples}. We
give a first example here:
\begin{example}
Let us fix $(\mathcal{G}',\natural,0_{\mathcal{G}'})=(\mathcal{G},\natural,0_{\mathcal{G}})=(\boldsymbol{\beta},\natural,0)$,
where $\boldsymbol{\beta}$ is the braid groupoid. It has the natural
numbers as its objects the natural numbers and its automorphisms are
the braid groups $\{\mathbf{B}_{n}\}_{n\in\mathbb{N}}$. The strict
monoidal structure $\natural$ is defined by the usual addition for
the objects and laying two braids side by side for the morphisms;
see \cite[Chapter XI, Section 4]{MacLane1} for more details.

Let $\mathbf{F}_{n}$ be the free group of rank $n$. In this case,
the Artin representations $\{\mathcal{A}_{1,n}^{\boldsymbol{\beta}}:\mathbf{B}_{n}\rightarrow\textrm{Aut}(\mathbf{F}_{n})\}_{n\in\mathbb{N}}$,
defined by the action $\mathbf{B}_{n}$ on the fundamental group of
a $n$-punctured disc, assemble to define a functor $\mathcal{A}_{1}^{\boldsymbol{\beta}}:\mathfrak{U}\boldsymbol{\beta}\rightarrow\mathfrak{Gr}$;
see Section \ref{subsec:symplecticfunctor }. Moreover, there exists
a family of non-trivial morphisms $\{\chi_{n,1}:\mathbf{F}_{n}\rightarrow\mathbf{B}_{n+1}\}_{n\in\mathbb{N}}$
(see Definition \ref{def:defsigma1}) such that the morphism given
by the coproduct $\chi_{n,1}*(id_{1}\natural-):\mathbf{F}_{n}*\mathbf{B}_{n}\rightarrow\mathbf{B}_{n+1}$
factors through the canonical surjection to the semidirect product
$\mathbf{F}_{n}\rtimes_{\mathcal{A}_{1,n}^{\boldsymbol{\beta}}}\mathbf{B}_{n}$
and such that the corresponding diagram (\ref{eq:diagramstability})
is commutative; see Section \ref{subsec:Modifyingpunctures}. We thus
define a non-trivial functor $\chi_{1}:\intop^{\mathfrak{U}\boldsymbol{\beta}}\mathcal{A}_{1}^{\boldsymbol{\beta}}\rightarrow\mathfrak{U}\boldsymbol{\beta}$
such that we have a coherent Long-Moody system $\{\mathcal{A}_{1}^{\boldsymbol{\beta}},\boldsymbol{\beta},\boldsymbol{\beta},\chi_{1}\}$.
We refer to Section \ref{subsec:Surface-braid-groups:} for more details.
Hence Definitions \ref{def:longmoodysystem} and \ref{Thm:LMFunctor}
recover their analogues \cite[Definition 2.14 and Theorem 2.19]{soulieLMBilan}.
\end{example}

\subsubsection{Properties of Long-Moody systems\label{subsec:Equivalent-characterisation-of}}

\paragraph{Equivalent characterization of Long-Moody systems.}

We give now an equivalent description of the functor $\chi$ introduced
in Definition \ref{def:longmoodysystem} to define a Long-Moody system.
For all natural numbers $n$, we denote by $\mathcal{A}_{n}:G_{n}\rightarrow\textrm{Aut}_{\mathfrak{Gr}}(\mathcal{A}(\underline{n}))$
the group morphisms induced by the functor $\mathcal{A}$. We note
that the coproduct $\mathcal{A}(\underline{n})*G_{n}$ canonically
surjects onto the semidirect product $\mathcal{A}(\underline{n})\rtimes_{\mathcal{A}_{n}}G_{n}$
and we denote this canonical surjection by $s_{\rtimes}$. Also, considering
the functor $\chi:\intop^{\mathcal{G}}\mathcal{A}\rightarrow\mathcal{G}$
is equivalent to considering a family of group morphisms $\{\chi_{n}^{\rtimes}:\mathcal{A}(\underline{n})\rtimes_{\mathcal{A}_{n}}G_{n}\rightarrow G_{n+1}\}_{n\in\mathbb{N}}$
by Lemma \ref{exa:casegroupcategory}. Hence, we deduce:
\begin{lem}
\label{lem:equivcond:coherenceconditionsigmanan}Given a family of
group morphisms $\{\chi_{n}:\mathcal{A}(\underline{n})\rightarrow G_{n+1}\}_{n\in\mathbb{N}}$
such that the coproduct morphism $\chi_{n}*(id_{1}\natural-):\mathcal{A}(\underline{n})*G_{n}\rightarrow G_{n+1}$
factors across the canonical surjection $s_{\rtimes}$, ie there exist
group morphisms $\{\chi_{n}^{\rtimes}:\mathcal{A}(\underline{n})\rtimes_{\mathcal{A}_{n}}G_{n}\rightarrow G_{n+1}\}_{n\in\mathbb{N}}$
such that the following diagram is commutative
\[
\xymatrix{\mathcal{A}(\underline{n})\ar@{^{(}->}[r]\ar@{->}[dr]_{\chi_{n}} & \mathcal{A}(\underline{n})\rtimes_{\mathcal{A}_{n}}G_{n}\ar@{->}[d]^{\chi_{n}^{\rtimes}} & G_{n}\ar@{_{(}->}[l]\ar@{->}[dl]^{id_{1}\natural-}\\
 & G_{1+n},
}
\]
then there is a functor $\chi:\intop^{\mathcal{G}}\mathcal{A}\rightarrow\mathcal{G}$
such that the diagram (\ref{cond:coherenceconditionsigmanan}) is
commutative. In particular, these conditions hold taking the composition
$\mathcal{A}(\underline{n})\hookrightarrow\mathcal{A}(\underline{n})\rtimes G_{n}\overset{\chi_{n}^{\rtimes}}{\to}G_{n+1}$
for $\chi_{n}$.
\end{lem}

\begin{proof}
The commutation of the diagram (\ref{cond:coherenceconditionsigmanan})
is equivalent to the equality in $G_{n+1}$
\begin{equation}
(id_{1}\natural g)\circ\chi_{n}(h)=\chi_{n}(\mathcal{A}_{n}(g)(h))\circ(id_{1}\natural g)\label{eq::equivcond:coherenceconditionsigmanan}
\end{equation}
for all $g\in G_{n}$ and $h\in\mathcal{A}(\underline{n})$. This
is exactly the definition of the fact that $\chi_{n}*(id_{1}\natural-)$
factors across the semidirect product $\mathcal{A}(\underline{n})\rtimes_{\mathcal{A}_{n}}G_{n}$.
\end{proof}
Furthermore, the following result highlights the underlying subtleties
when extending a Long-Moody system to a coherent one.
\begin{prop}
\label{cond:conditionstability-1}We consider a functor $\chi:\intop^{\mathcal{G}}\mathcal{A}\rightarrow\mathcal{G}$.
If the functor $\chi$ extends to the Grothendieck construction $\intop^{\mathfrak{U}\mathcal{G}}\mathcal{A}$
such that the diagram (\ref{cond:coherenceconditionsigmanan2}) is
commutative, then the extension is unique.

Furthermore, given a family of group morphisms $\{\chi_{n}:\mathcal{A}(\underline{n})\rightarrow G_{n+1}\}_{n\in\mathbb{N}}$
satisfying the conditions of Lemma \ref{lem:equivcond:coherenceconditionsigmanan}
and such that the following diagram is commutative in the category
$\mathfrak{U}\mathcal{G}$
\begin{equation}
\xymatrix{1\natural\underline{n}\ar@{->}[rrr]^{\chi_{n}(h)}\ar@{->}[d]_{id_{1}\natural[n'-n,id_{\underline{n'}}]} &  &  & 1\natural\underline{n}\ar@{->}[d]^{id_{1}\natural[n'-n,id_{\underline{n'}}]}\\
1\natural\underline{n'}\ar@{->}[rrr]_{\chi_{n'}(\mathcal{A}([n'-n,id_{\underline{n'}}])(h))} &  &  & 1\natural\underline{n'},
}
\label{eq:diagramstability}
\end{equation}
for all elements $h\in\mathcal{A}(\underline{n})$, for all natural
numbers $n$ and $n'$ such that $n'\geq n$, then there is a functor
$\chi:\intop^{\mathfrak{U}\mathcal{G}}\mathcal{A}\rightarrow\mathcal{G}$
such that the diagram (\ref{cond:coherenceconditionsigmanan2}) is
commutative. In particular, if the Long-Moody system $\{\mathcal{A},\mathcal{G},\mathcal{G}',\chi\}$
is \textbf{coherent}, these conditions hold taking the composition
$\mathcal{A}(\underline{n})\hookrightarrow\mathcal{A}(\underline{n})\rtimes G_{n}\overset{\chi_{n}^{\rtimes}}{\to}G_{n+1}$
for $\chi_{n}$.
\end{prop}

\begin{proof}
We recall that $e_{\mathcal{A}(\underline{n'})}$ denotes the unit
elements of the groups $\mathcal{A}(\underline{n'})$. It follows
from the definition of the Grothendieck construction $\intop^{\mathfrak{U}\mathcal{G}}\mathcal{A}$
that an extension of $\chi:\intop^{\mathcal{G}}\mathcal{A}\rightarrow\mathcal{G}$
is defined by $\chi((\cdot_{\underline{n}},\underline{n}))=1\natural\underline{n}$
for all $\underline{n}\in\textrm{Obj}(\mathcal{G})$ for objects,
and for morphisms by $\chi((h,id_{\underline{n}}))=\chi_{n}(h)$ for
all $h\in\mathcal{A}(\underline{n})$ and
\[
\chi((e_{\mathcal{A}(\underline{n'})},[n'-n,\varphi]))=id_{1}\natural[n'-n,\varphi]
\]
for all $[n'-n,\varphi]\in\textrm{Hom}_{\mathfrak{U}\mathcal{G}}(\underline{n},\underline{n'})$.
In particular, the uniqueness of the extension follows from these
assignments. Actually, the functor $\chi$ extends to define a coherent
Long-Moody system if and only if these assignments on morphisms satisfy
the composition axiom for a functor. We note that the additional composition
axiom which has to be checked for extending $\chi$ to $\intop^{\mathfrak{U}\mathcal{G}}\mathcal{A}$
is for the morphisms of type $(e_{\mathcal{A}(\underline{n}')},[n'-n,id_{\underline{n'}}])\circ(h,id_{\underline{n}})$.
Namely, $\chi$ extends to $\intop^{\mathfrak{U}\mathcal{G}}\mathcal{A}$
if and only if
\[
\chi((e_{\mathcal{A}(\underline{n'})},[n'-n,id_{\underline{n'}}]))\circ\chi((h,id_{\underline{n}}))=\chi((e_{\mathcal{A}(\underline{n'})},[n'-n,id_{\underline{n'}}])\circ(h,id_{\underline{n}}))
\]
for all natural numbers $n'\geq n$ and $h\in\mathcal{A}(\underline{n})$.
The second statement is then a direct consequence of the assignments
for $\chi$ and of the composition rule in $\intop^{\mathfrak{U}\mathcal{G}}\mathcal{A}$.
\end{proof}
Alternatively, the following proposition gives a \textit{sufficient
criterion} for a Long-Moody system to be coherent. This result is
in practice more convenient to handle and easier to check that the
technical equivalent condition of Proposition \ref{cond:conditionstability-1};
see Section \ref{subsec:Examples}. We denote by $\sigma_{1}$ the
braiding $b_{1,1}^{\mathcal{G}'}\natural id_{\underline{n}}$ of $\mathcal{G}'$
for simplicity; this is consistent with the usual notation for the
Artin generators of braid groups.
\begin{prop}
\label{cond:conditionstability}Given a family of group morphisms
$\{\chi_{n}:\mathcal{A}(\underline{n})\rightarrow G_{n+1}\}_{n\in\mathbb{N}}$
satisfying the conditions of Lemma \ref{lem:equivcond:coherenceconditionsigmanan}
and such that following equality holds in $G_{n+2}$
\begin{eqnarray}
id_{1}\natural\chi_{n}(h) & = & \sigma_{1}\circ\chi_{n+1}(\mathcal{A}([1,id_{\underline{n+1}}])(h))\circ\sigma_{1}^{-1},\label{eq:equiva}
\end{eqnarray}
for all elements $h\in\mathcal{A}(\underline{n})$ and for all $n$,
then there is a functor $\chi:\intop^{\mathfrak{U}\mathcal{G}}\mathcal{A}\rightarrow\mathcal{G}$
defining a \textbf{coherent} Long-Moody system $\{\mathcal{A},\mathcal{G},\mathcal{G}',\chi\}$.

Conversely, if $\textrm{Aut}_{\mathcal{G}'}(1)=\{id_{1}\}$ and considering
a \textbf{coherent} Long-Moody system $\{\mathcal{A},\mathcal{G},\mathcal{G}',\chi\}$,
then the equality (\ref{eq:equiva}) holds for each $n$ taking the
composition $\mathcal{A}(\underline{n})\hookrightarrow\mathcal{A}(\underline{n})\rtimes G_{n}\overset{\chi_{n}^{\rtimes}}{\to}G_{n+1}$
for $\chi_{n}$.
\end{prop}

\begin{proof}
First, note that by definition of the braiding $b_{-,-}^{\mathcal{G}'}$,
we have:
\[
(b_{1,1}^{\mathcal{G}'})^{-1}\natural id_{n}=((b_{1,1}^{\mathcal{G}'})^{-1}\natural id_{n'-1})\circ(id_{1}\natural(b_{1,\left(n'-n\right)-1}^{\mathcal{G}'})^{-1}\natural id_{n}).
\]
Hence, a straightforward recursion proves that the commutation of
the diagram (\ref{eq:diagramstability}) is equivalent to assuming
that the morphisms $\{\chi_{n}\}_{n\in\mathbb{N}}$ satisfy the following
equality, as morphisms in the category $\mathfrak{U}\mathcal{G}$,
for all elements $h\in\mathcal{A}(\underline{n})$:
\[
[1,((b_{1,1}^{\mathcal{G}'})^{-1}\natural id_{\underline{n}})\circ(id_{1}\natural\chi_{n}(h))]=[1,\chi_{n+1}(\mathcal{A}([1,id_{\underline{n+1}}])(h))\circ((b_{1,1}^{\mathcal{G}'})^{-1}\natural id_{\underline{n}})].
\]
Hence, by Proposition \ref{cond:conditionstability-1}, the statements
follow from the equivalence relation (\ref{eq:equivalence relation})
for morphisms in $\mathfrak{U}\mathcal{G}$.
\end{proof}
\begin{rem}
In Section \ref{sec:Behaviour-of-the}, we will have to assume that
the stronger equality (\ref{eq:equiva}) holds; see Assumption \ref{assu:decomposeAfreeproduct}.
\end{rem}

\paragraph{Connection with (twisted) first homology.}

Let \textbf{$\{\mathcal{A},\mathcal{G},\mathcal{G}',\chi\}$} be a
(possibly coherent) Long-Moody system and let $F$ be an object of
$\mathbf{Fct}(\mathcal{G},R\textrm{-}\mathfrak{Mod})$. We fix a natural
number $n$. We recall that the augmentation ideal defines the short
exact sequence of $R[\mathcal{A}(\underline{n})]$-modules $0\to\mathcal{I}_{R[\mathcal{A}(\underline{n})]}\to R[\mathcal{A}(\underline{n})]\to R\to0$.
We consider the long exact sequence obtained by applying the functor
$\textrm{Tor}_{*}^{R[\mathcal{A}(\underline{n})]}(-,F(\underline{n+1}))$
to that short exact sequence. We note that $\textrm{Tor}_{i}^{R[\mathcal{A}(\underline{n})]}(R[\mathcal{A}(\underline{n})],F(\underline{n+1}))=0$
for all $i>0$ since $R[\mathcal{A}(\underline{n})]$ is a free $R[\mathcal{A}(\underline{n})]$-module
and recall that the twisted homology group $H_{i}(\mathcal{A}(\underline{n});F(\underline{n+1}))$
is equal to $\textrm{Tor}_{i}^{R[\mathcal{A}(\underline{n})]}(R,F(\underline{n+1}))$.
Therefore, we obtain the following $4$-term exact sequence of $R$-modules:
\begin{equation}
0\xymatrix{\ar@{->}[r] & H_{1}(\mathcal{A}(\underline{n});F(\underline{n+1}))\ar@{->}[r] & \mathbf{LM}(F)(\underline{n})\ar@{->}[r] & F(\underline{n+1})\ar@{->}[r] & H_{0}(\mathcal{A}(\underline{n});F(\underline{n+1}))\ar@{->}[r] & 0.}
\label{eq:4_term_sequence_LM}
\end{equation}
If $F(\underline{n+1})$ is not a trivial $G_{n}$-module, the first
homology group of the group $\mathcal{A}(\underline{n})$ with twisted
coefficients thus defines a \textit{proper} subrepresentation of the
one encoded by the Long-Moody functor.

If the group $\mathcal{A}(\underline{n})$ is free, the augmentation
ideal $\mathcal{I}_{R[\mathcal{A}(\underline{n})]}$ has the significant
advantage to automatically be a free module equipped with a natural
basis (see \cite[Proposition 6.2.6]{Weibel1} for instance), whereas
the first twisted homology of a free group can be less convenient
to handle. Moreover, the action of the group $G_{n}$ on the Long-Moody
representation has an explicit formula from the definition of the
Long-Moody construction and can therefore be directly fully computed.
On the contrary, understanding the $G_{n}$-action on the first twisted
homology requires much more work generally speaking.

\paragraph{Iteration and triviality.}

Considering the categories $\mathcal{G}'$, $\mathcal{G}$ and a functor
$\mathcal{A}:\mathfrak{U}\mathcal{G}\rightarrow\mathfrak{Gr}$, there
always exists at least one functor $\chi:\intop^{\mathfrak{U}\mathcal{G}}\mathcal{A}\rightarrow\mathfrak{U}\mathcal{G}$
such that the diagram (\ref{cond:coherenceconditionsigmanan2}) is
commutative: the functor $\chi_{tr}:\intop^{\mathcal{G}}\mathcal{A}\rightarrow\mathcal{G}$
induced by the family of morphisms $\{\chi_{n,tr}:\mathcal{A}(\underline{n})\rightarrow0_{\mathfrak{Gr}}\rightarrow G_{n+1}\}_{n\in\mathbb{N}}$
factoring across the trivial group $0_{\mathfrak{Gr}}$ (considered
as a category with one object) always works. This functor $\chi_{tr}$
trivially extends to $\intop^{\mathfrak{U}\mathcal{G}}\mathcal{A}$,
a fortiori defining a coherent Long-Moody system $\{\mathcal{A},\mathcal{G},\mathcal{G}',\chi_{tr}\}$.

We denote by $R:\mathfrak{U}\mathcal{G}\rightarrow R\textrm{-}\mathfrak{Mod}$
the constant functor at $R$. We also recall that the first homology
group $H_{1}(-;R)$ defines a functor from the category of groups
$\mathfrak{Gr}$ to the category $R\textrm{-}\mathfrak{Mod}$ (see
\cite[Section 8]{brown} for example), and we denote the composition
$H_{1}(-;R)\circ\mathcal{A}$ by $H_{1}(\mathcal{A};R)$. The following
result provides a general description of a Long-Moody functor when
the input has a trivial $\mathcal{A}(\underline{n})$-module structure
for all $n$. In particular, this description fully determines the
Long-Moody functor associated to the coherent system $\{\mathcal{A},\mathcal{G},\mathcal{G}',\chi_{tr}\}$.
\begin{prop}
\label{prop:casesigmatrivial} Let \textbf{$\{\mathcal{A},\mathcal{G},\mathcal{G}',\chi\}$}
be a (possibly coherent) Long-Moody system and $F$ be an object of
$\mathbf{Fct}(\mathcal{G},R\textrm{-}\mathfrak{Mod})$ (or possibly
$\mathbf{Fct}(\mathfrak{U}\mathcal{G},R\textrm{-}\mathfrak{Mod})$
if the Long-Moody system is coherent).

If $F(\underline{n+1})$ is a trivial $R[\mathcal{A}(\underline{n})]$-module
for all natural numbers $n$, then $\mathbf{LM}(F)\cong H_{1}(\mathcal{A};R)\otimes_{R}F(1\natural-).$
\end{prop}

\begin{proof}
Let $n$ be a natural number. Since $F(1\natural\underline{n})$ a
trivial $R[\mathcal{A}(\underline{n})]$-module, we obtain that $H_{0}(\mathcal{A}(\underline{n});F(1\natural\underline{n}))\cong F(1\natural\underline{n})$.
Then, we deduce from the exact sequence (\ref{eq:4_term_sequence_LM})
and from the universal coefficient theorem for group homology that
there is an $R$-module isomorphism $\mathcal{I}_{R[\mathcal{A}(\underline{n})]}\otimes_{R[\mathcal{A}(\underline{n})]}F(1\natural\underline{n})\cong H_{1}(\mathcal{A};R)(\underline{n})\otimes_{R}F(1\natural\underline{n})$.
It is straightforward to check that this isomorphism is natural with
respect to the morphisms of $\mathfrak{U}\mathcal{G}$.
\end{proof}

\paragraph{Using free products of groups for $\mathcal{A}$.}

We obtain more properties for the associated Long-Moody functors assuming
that there exist two \textit{free} groups $H_{0}$ and $H$ such that
$\mathcal{A}(\underline{n})=H^{*n}*H_{0}$ for all natural numbers
$n$. First, we recall the following classical result:
\begin{lem}
\label{lem:Swan}Let $G$ be a torsion-free group. The augmentation
ideal $\mathcal{I}_{R[G]}$ is a projective $R[G]$-module if and
only if $G$ is a free group.
\end{lem}

\begin{proof}
Let us assume that $\mathcal{I}_{R[G]}$ is a projective $R[G]$-module.
The short exact sequence $0\rightarrow\mathcal{I}_{R[G]}\rightarrow R[G]\rightarrow R\rightarrow0$
is thus a projective resolution of $R$ as an $R[G]$-module. Hence
the cohomological dimension of $G$ is one. Then it follows from \cite[Theorem A]{swan1969groups}
that $G$ is a free group. The converse is a classical result of homological
algebra; see \cite[Corollary 6.2.7]{Weibel1}.
\end{proof}
\begin{cor}
\label{cor:exactnessLM} Let \textbf{$\{\mathcal{A},\mathcal{G},\mathcal{G}',\chi\}$}
be a (possibly coherent) Long-Moody system such that $\mathcal{A}(\underline{n})=H^{*n}*H_{0}$
for all natural numbers $n$, where the groups $H$ and $H_{0}$ are
free. Then the associated functor $\mathbf{LM}$ is exact and commutes
with all finite limits.
\end{cor}

\begin{proof}
Let $n$ be a natural number. Since the augmentation ideal $\mathcal{I}_{R[\mathcal{A}(\underline{n})]}$
is a projective $R[\mathcal{A}(\underline{n})]$-module by Lemma \ref{lem:Swan},
it is a flat $R[\mathcal{A}(\underline{n})]$-module: hence the functor
$\mathcal{I}_{R[\mathcal{A}(\underline{n})]}\otimes_{R[\mathcal{A}(\underline{n})]}-:R\textrm{-}\mathfrak{Mod}\rightarrow R\textrm{-}\mathfrak{Mod}$
is exact. Then the naturality of this exactness property with respect
to the morphisms of $\mathfrak{U\mathcal{G}}$ is a formal consequence
of the definition of the Long-Moody functor; see Theorem \ref{Thm:LMFunctor}.
The commutation result for finite limits is a general property of
exact functors; see \cite[Chapter 8, section 3]{MacLane1} for example.
\end{proof}

\section{Examples\label{subsec:Examples}}

The groups $\left\{ G_{n}\right\} _{n\in\mathbb{N}}$ for which it
is natural to define some generalized Long-Moody functors are mapping
class groups of surfaces. In this section, we present various coherent
Long-Moody systems which are defined for several families of mapping
class groups and surface braid groups. In particular, we recover the
Long-Moody functors for braid groups of \cite{soulieLMBilan} in Section
\ref{subsec:Surface-braid-groups:}, and we introduce new ones thanks
to the more general framework of Section \ref{sec:The-generalized-Long-Moody}.

\subsection{The monoidal groupoid associated with surfaces\label{subsec:The-monoidal-groupoid}}

Let us first introduce a suitable category for our work, inspired
by \cite[Section 5.6]{WahlRandal-Williams}. Namely, the decorated
surfaces groupoid $\mathcal{M}_{2}$ is the locally small groupoid
defined by:\\
\\
$\star$ \textbf{Objects:} decorated surfaces $(S,I,\mathcal{\mathscr{P}})$,
where:
\begin{itemize}
\item $S$ is a smooth connected compact surface with one boundary component
denoted by $\partial_{0}S$;
\item $I:[-1,1]\hookrightarrow\partial_{0}S$ is a parametrized interval
in the boundary and $p=0\in I$ a basepoint; when there is no ambiguity,
we omit the parametrized interval $I$ from the notation;
\item $\mathcal{\mathscr{P}}$ is a finite set of points removed from the
interior of $S$ (in other words with \textit{punctures}); we usually
omit $\mathcal{\mathscr{P}}$ from the notation for convenience or
if it is empty.
\end{itemize}
$\star$ \textbf{Morphisms:} the isotopy classes of homeomorphisms
of $S$ restricting to the identity on a neighborhood of the parametrized
interval $I$ and fixing the punctures setwise, denoted by $\pi_{0}\textrm{Homeo}^{I}(S,\mathscr{P})$.
\begin{rem}
\label{rem:alternative_morph_M_2}We recall that $\pi_{0}\textrm{Homeo}^{I}(S,\mathscr{P})$
is isomorphic to the isotopy classes of homeomorphisms of $S$ restricting
to the identity on the boundary $\partial_{0}S$ and fixing the punctures
setwise, which is the traditional way to define the mapping class
group of $S$. Also, if the surface $S$ is orientable, the orientation
on $S$ is induced by the orientation of $I$: then the elements of
$\pi_{0}\textrm{Homeo}^{I}(S,\mathscr{P})$ automatically preserve
that orientation.

Alternatively, the morphisms of the groupoid $\mathcal{M}_{2}$ may
be described as follows. We denote by $\widehat{S}$ the surface obtained
from $S\in\textrm{Obj}(\mathcal{M}_{2})$ removing an open disc on
a neighbourhood of each punctures, so that we obtain boundary components
instead of punctures; see \cite[Section 5.6.1]{WahlRandal-Williams}.
Let $\textrm{Diff}{}^{\partial_{0}}(\widehat{S})$ be the group of
diffeomorphisms of $\hat{S}$ fixing pointwise the boundary component
$\partial_{0}S$ and moving freely the other boundary components.
We recall that $\pi_{0}\textrm{Homeo}^{I}(S,\mathscr{P})$ is isomorphic
to the group of isotopy classes $\pi_{0}\textrm{Diff}{}^{\partial_{0}}(\widehat{S})$;
see \cite[Section 1.4.2]{farb2011primer} for instance.
\end{rem}

The groupoid $\mathcal{M}_{2}$ has a monoidal structure induced by
gluing; for completeness, the definition is outlined below; we refer
to \cite[Section 5.6.1]{WahlRandal-Williams} for technical details.
Let $(S_{1},I_{1})$ and $(S_{2},I_{2})$ be for two decorated surfaces.
For a parametrized interval $I$, the left half-interval $[-1,0]\hookrightarrow\partial_{0}S$
of $I$ is denoted by $I^{-}$ and the right half-interval $[0,1]\hookrightarrow\partial_{0}S$
of $I$ is denoted by $I^{+}$, defining $I=I^{-}\cup I^{+}$. We
define the boundary connected sum $(S_{1},I_{1})\natural(S_{2},I_{2})$
as $(S_{1}\natural S_{2},I_{1}\natural I_{2})$, where $S_{1}\natural S_{2}$
is the surface obtained from gluing $S_{1}$ and $S_{2}$ along the
half-interval $I_{1}^{+}$ and the half-interval $I_{2}^{-}$, and
$I_{1}\natural I_{2}=I_{1}^{-}\cup I_{2}^{+}$. The homeomorphisms
being the identity on a neighbourhood of the parametrized intervals
$I_{1}$ and $I_{2}$, we canonically extend the homeomorphisms of
$S_{1}$ and $S_{2}$ to $S_{1}\natural S_{2}$. The braiding of the
monoidal structure $b_{(S_{1},I_{1}),(S_{2},I_{2})}^{\mathcal{M}_{2}}$
is the morphism $(S_{1},I_{1})\natural(S_{2},I_{2})\rightarrow(S_{2},I_{2})\natural(S_{1},I_{1})$
given by doing half a Dehn twist in a pair of pants neighbourhood
of $\partial_{0}S_{1}$ and $\partial_{0}S_{2}$; see Figure \ref{fig:tressageM2}.
\begin{notation}
We consider the unit $2$-disc $\mathbb{D}^{2}$. We denote by $D$
the disc $\mathbb{D}^{2}$ with one puncture, we fix a torus with
one boundary component denoted by $T$ and a Möbius band denoted by
$M$. By the classification of surfaces, for $S$ an object of the
groupoid $\mathcal{M}_{2}$, there exist $g,s,h\in\mathbb{N}$ such
that there is an homeomorphism $S\simeq(\natural_{s}D)\natural(\natural_{g}T)\natural(\natural_{h}M)$.
Moreover, if $h=0$, then $g$ and $s$ are unique.
\end{notation}

\begin{figure}
~~~~~~~~~~~~~~~~~~~~~~~\includegraphics[clip,scale=0.4]{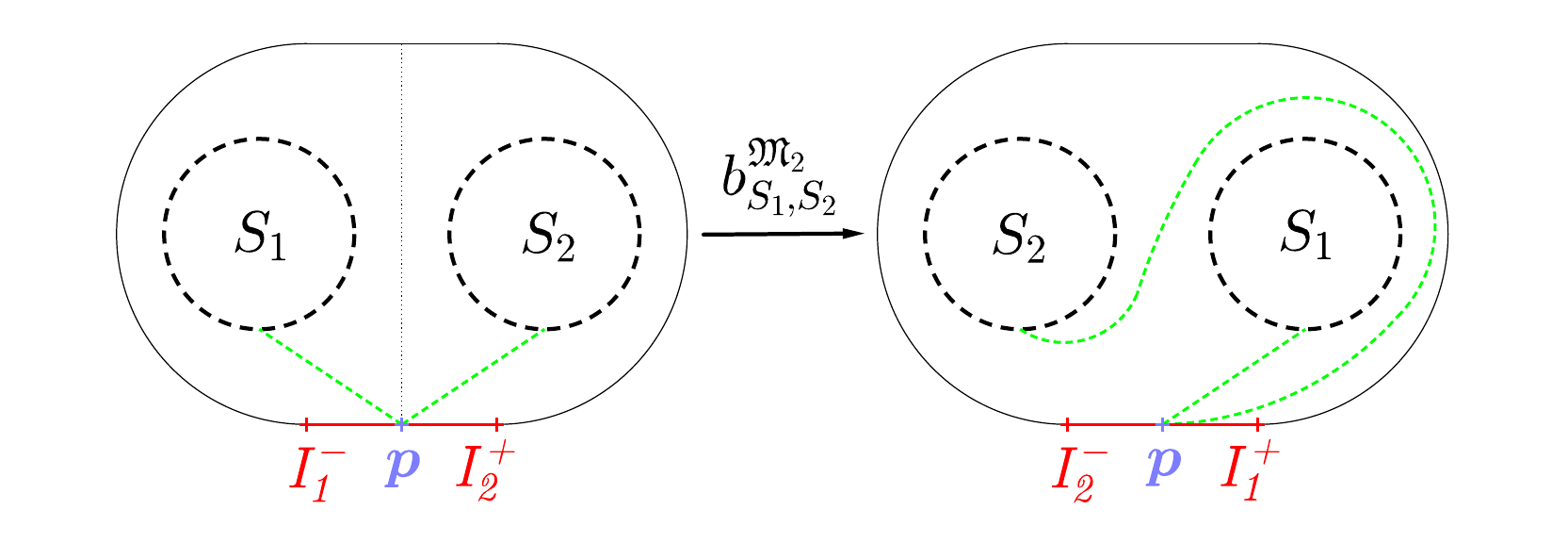}

\caption{\label{fig:tressageM2}Braiding}
\end{figure}

By \cite[Proposition 5.18]{WahlRandal-Williams}, the boundary connected
sum $\natural$ induces a braided monoidal structure on $\mathcal{M}_{2}$.
There are no zero divisors in the category $\mathcal{M}_{2}$ and
$\textrm{Aut}_{\mathcal{M}_{2}}(\mathbb{D}^{2})=\{id_{\mathbb{D}^{2}}\}$
by Alexander's trick. For consistency with respect to the framework
of Section \ref{sec:The-generalized-Long-Moody}, we consider the
following subgroupoid of $\mathcal{M}_{2}$ for the monoidal structure
$\natural$ to become strict: let $\mathfrak{M}_{2}$ be the full
subgroupoid of $\mathcal{M}_{2}$ on the same objects modulo the identification
that $\mathbb{D}^{2}\natural S=S\natural\mathbb{D}^{2}=S$ for any
surface $S$. Hence the groupoid $(\mathfrak{M}_{2},\natural,\mathbb{D}^{2},b_{-,-}^{\mathfrak{M}_{2}})$
is locally small, strict braided monoidal with no zero divisors and
such that $\textrm{Aut}_{\mathfrak{M}_{2}}(\mathbb{D}^{2})=\{id_{\mathbb{D}^{2}}\}$.
We denote by $\mathfrak{U}\mathfrak{M}_{2}$ Quillen's bracket construction
on the groupoid $(\mathfrak{M}_{2},\natural,\mathbb{D}^{2})$; by
Proposition \ref{prop:Quillen'sconstructionprebraided}, we obtain
a strict pre-braided monoidal category $(\mathfrak{U}\mathfrak{M}_{2},\natural,\mathbb{D}^{2})$
where the unit $\mathbb{D}^{2}$ is an initial object. Following the
framework of Section \ref{subsec:Definition-of-theLM}, the groupoid
$\mathfrak{M}_{2}$ plays the role of the category $\mathcal{G}'$.

\subsection{Fundamental group functor\label{subsec:symplecticfunctor }}

As detailed in Definition \ref{def:longmoodysystem}, we have to fix
a functor $\mathcal{A}:\mathfrak{U}\mathcal{G}\rightarrow\mathfrak{Gr}$
in order to define a Long-Moody functor. In this section, we choose
the functor which will play the role of that parameter introduced
in the next section.

We define the functor $\pi_{1}(-,p):(\mathfrak{M}_{2},\natural,\mathbb{D}^{2})\rightarrow\mathfrak{Gr}$
by assigning the fundamental group $\pi_{1}(S,p)$ to each object
$S$ of $\mathfrak{M}_{2}$ and the corresponding natural action,
denoted by $\mathcal{A}_{S}$,  of the mapping class group on $\pi_{1}(S,p)$
for the morphisms of $\mathfrak{M}_{2}$.

Let $\mathfrak{gr}$ be a skeleton of the full subcategory of $\mathfrak{Gr}$
of finitely-generated free groups. The free product $*:\mathfrak{gr}\times\mathfrak{gr}\rightarrow\mathfrak{gr}$
defines a monoidal structure on $\mathfrak{gr}$, with $0$ the unit,
denoted by $(\mathfrak{gr},*,0)$. We recall that the group $\pi_{1}(S,p)$
is a finitely generated free group for any object $S$ of $\mathfrak{M}_{2}$.
We use the category $\mathfrak{gr}$ for the functor $\pi_{1}(-,p)$
to be strict monoidal:
\begin{lem}
\label{lem:symplecticstrongmon}The functor $\pi_{1}(-,p):(\mathfrak{M}_{2},\natural,\mathbb{D}^{2})\rightarrow(\mathfrak{gr},*,0_{\mathfrak{Gr}})$
is strict monoidal.
\end{lem}

\begin{proof}
Since the category $\mathfrak{gr}$ is skeletal, it follows from Van
Kampen's theorem that $\pi_{1}(S'\natural S,p)=\pi_{1}(S',p)*\pi_{1}(S,p)$
for any $S,S'\in\textrm{Obj}(\mathfrak{M}_{2})$. Let $\phi$ and
$\phi'$ be elements of $\pi_{0}\textrm{Homeo\ensuremath{^{I}}}(S,\mathscr{P})$
and $\pi_{0}\textrm{Homeo\ensuremath{^{I}}}(S',\mathscr{P})$ respectively.
The isomorphism $id_{S'}\natural\phi$ (respectively $\phi'\natural id_{S}$)
corresponds to extending $\phi$ (resp. $\phi'$) to $S'\natural S$
by the identity on $S'$ (resp. $S$). It is clear that  $id_{S'}\natural\phi$
acts trivially on $\pi_{1}(S',p)$ in $\pi_{1}(S',p)*\pi_{1}(S,p)$
and has the same action as $\phi$ on $\pi_{1}(S,p)$ in $\pi_{1}(S',p)*\pi_{1}(S,p)$.
Similarly, $\phi'\natural id_{S}$ acts trivially on $\pi_{1}(S,p)$
in $\pi_{1}(S',p)*\pi_{1}(S,p)$ and has the same action as $\phi'$
on $\pi_{1}(S',p)$ in $\pi_{1}(S',p)*\pi_{1}(S,p)$. Hence the action
of $\phi'\natural\phi$ on $\pi_{1}(S'\natural S,p)=\pi_{1}(S',p)*\pi_{1}(S,p)$
is the same as $\mathcal{A}_{S'}(\phi')*\mathcal{A}_{S}(\phi)$, which
ends the proof.
\end{proof}
The object $0_{\mathfrak{Gr}}$ is null in the category of groups
$\mathfrak{Gr}$, and we let $\iota_{G}:0_{\mathfrak{Gr}}\rightarrow G$
denote the unique morphism from $0_{\mathfrak{Gr}}$ to the group
$G$. The next proposition tells us that the functor $\pi_{1}(-,p)$
naturally extends to the category $\mathfrak{U}\mathfrak{M}_{2}$.
\begin{prop}
\label{prop:definesymplecticfunctor}The functor $\pi_{1}(-,p)$ of
Lemma \ref{lem:symplecticstrongmon} extends to a functor $\pi_{1}(-,p):\mathfrak{U}\mathfrak{M}_{2}\rightarrow\mathfrak{gr}$
by assigning $\pi_{1}(-,p)([S',id_{S'\natural S}])=\iota_{\pi_{1}(S',p)}*id_{\pi_{1}(S,p)}$
to all $S,S'\in\textrm{Obj}(\mathfrak{M}_{2})$.
\end{prop}

\begin{proof}
Let $S$ and $S'$ be objects of $\mathfrak{M}_{2}$. It follows from
the definitions that relation (\ref{eq:criterion}) of Lemma \ref{lem:criterionfamilymorphismsfunctor}
is automatically satisfied. Let $\varphi\in\pi_{0}\textrm{Homeo\ensuremath{^{I}}}(S,\mathscr{P})$
and $\varphi'\in\pi_{0}\textrm{Homeo\ensuremath{^{I}}}(S',\mathscr{P}')$.
By Lemma \ref{lem:symplecticstrongmon}, we have $\pi_{1}(-,p)(\varphi'\natural\varphi)\circ\pi_{1}(-,p)([S',id_{S'\natural S}])=(\mathcal{A}_{S'}(\varphi')*\mathcal{A}_{S}(\varphi))\circ\pi_{1}(-,p)([S',id_{S'\natural S}])$.
Hence, by definition of the morphism $\iota_{\pi_{1}(S',p)}$, we
obtain that $\pi_{1}(-,p)(\varphi'\natural\varphi)\circ\pi_{1}(-,p)([S',id_{S'\natural S}])$
is equal to $\pi_{1}(-,p)([S',id_{S'\natural S}])\circ\mathcal{A}_{S}(\varphi)$.
Relation (\ref{eq:criterion'}) of Lemma \ref{lem:criterionfamilymorphismsfunctor}
is thus satisfied, which implies the desired result.
\end{proof}
In the constructions of the following sections, we take restrictions
of the functor $\pi_{1}(-,p)$ to some small subcategories of $\mathfrak{U}\mathfrak{M}_{2}$
for the parameter $\mathcal{A}$ to define the Long-Moody systems.
Note that we could make other choices for such parameter: for instance,
in the case of braid groups, Wada introduces a families of morphisms
$\mathbf{B}_{n}\rightarrow\textrm{Aut}(\mathbf{F}_{n})$ which are
not Artin's representation \cite{wada1992group}. We may use such
other choices to construct alternative Long-Moody functors.

\subsection{Modifying the number of punctures\label{subsec:Modifyingpunctures0}}

This section is devoted to the construction of Long-Moody functors
for the mapping class groups of surfaces when the number of punctures
varies (see Section \ref{subsec:Modifyingpunctures}) and surface
braid groups (see Section \ref{subsec:Surface-braid-groups:}). We
will detail the case of orientable surfaces for simplicity, and then
we will see by the end of this section that the non-orientable case
follows mutatis mutandis. Finally, we study the applications of these
Long-Moody functors in Section \ref{subsec:Applications}.

\subsubsection{Mapping class groups\label{subsec:Modifyingpunctures}}

We fix a natural number $g\geq0$. For all natural numbers $n$, we
denote by $\varSigma_{g,1}^{n}$ the surface $D^{\natural n}\natural T^{\natural g}$
(and $\varSigma_{0,1}^{0}:=\mathbb{D}^{2}$ for consistency), and
we denote the corresponding mapping class group by $\mathbf{\Gamma}_{g,1}^{n}$.
Let $\mathfrak{M}_{2}^{+,g}$ be the small full subgroupoid of $\mathfrak{M}_{2}$
on the objects $\{\varSigma_{g,1}^{n}\}_{n\in\mathbb{N}}$. In particular,
in the notations of Section \ref{subsec:Definition-of-theLM}, the
groupoid $\mathfrak{M}_{2}$ plays the role of $\mathcal{G}'$ and
$\mathfrak{M}_{2}^{+,g}$ corresponds to $\mathcal{G}$ where $\underline{0}:=\varSigma_{g,1}^{0}$
and $1:=\varSigma_{0,1}^{1}$.

Let $H$ be the free group $\pi_{1}(\varSigma_{0,1}^{1},p)\cong\mathbb{Z}$,
let $H_{0}$ be the free group $\pi_{1}(\varSigma_{g,1}^{0},p)\cong\mathbf{F}_{2g}$
and thus $H_{n}=H^{*n}*H_{0}=\pi_{1}(\varSigma_{g,1}^{n},p)$ for
all natural numbers $n$. We denote by $\pi_{1}(\varSigma_{g,1}^{-},p):\mathfrak{U}\mathfrak{M}_{2}^{+,g}\rightarrow\mathfrak{Gr}$
the restriction of the functor of Proposition \ref{prop:definesymplecticfunctor}
along the inclusion functor $\mathfrak{U}\mathfrak{M}_{2}^{+,g}\hookrightarrow\mathfrak{U}\mathfrak{M}_{2}$.

Let $\{a_{i},b_{i}\mid i\in\{1,\ldots,g\}\}$ be a direct system of
meridians and longitudes of the surface $\varSigma_{g,1}^{0}$ with
orientations as pictured in Figure \ref{fig:Generators-and-paths}.
For $n$ a natural number and $j\in\{1,\ldots,n\}$, let $p_{j}$
be an oriented closed curve encircling the puncture in the $j$-th
surface $\varSigma_{0,1}^{1}$ in $\varSigma_{0,1}^{1}\natural\cdots\natural\varSigma_{0,1}^{1}\natural\varSigma_{g,1}^{0}$
as in Figure \ref{fig:Generators-and-paths}; we abuse the notation
and also denote by $p_{j}$ the puncture encircled by the curve $p_{j}$.

The generator of the $j$-th copy of $H$ in $H_{n}$ is the homotopy
class of a simple closed curve $\alpha_{p_{j}}$ of $\varSigma_{0,1}^{1}$
in $\varSigma_{g,1}^{n}$ based at $p$ and encircling the curve $p_{j}$.
A generator $h'$ of $H_{0}$ in $H_{n}$ is the homotopy class of
a simple closed curve $\alpha_{h'}$ of $\varSigma_{g,1}^{0}$ in
$\varSigma_{g,1}^{n}$ based at $p$ and encircling on the curves
$\{a_{i},b_{i}\mid i\in\{1,\ldots,g\}\}$. From now on, we fix a choice
of such simple closed curves $\alpha_{h}$ for each standard generator
$h$ of $H_{n}$.\\

\begin{figure}
~~~~~~~~~~~~~~~~~~~~\includegraphics[clip,scale=0.4]{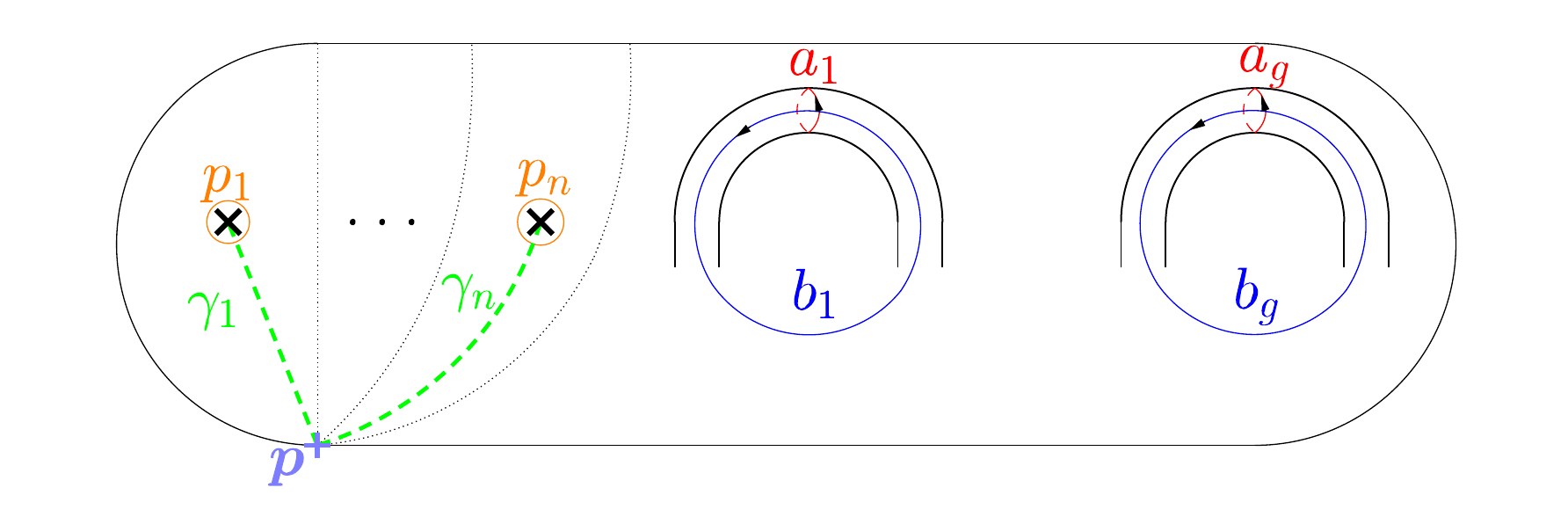}

\caption{\label{fig:Generators-and-paths}Generators and paths}
\end{figure}

To introduce the group morphisms $\{\chi_{n}:\pi_{1}(\varSigma_{g,1}^{n},p)\rightarrow\mathbf{\Gamma}_{g,1}^{1+n}\}_{n\in\mathbb{N}}$
in order to define a Long-Moody functor in this section, we first
need to introduce additional tools and recall some classical facts
about mapping class groups of surfaces.

For the unit disc with one puncture $\varSigma_{0,1}^{1}$, we consider
$x_{1}$ a marked point filling in the puncture $p_{1}$ and denote
by $\varSigma_{0,1}^{[x_{1}]}$ the obtained surface. Let $\gamma_{1}$
be a path in $\varSigma_{0,1}^{[x_{1}]}$ connecting the point $p\in I$
to $x_{1}$. For $n$ a natural number and $j\in\{1,\ldots,n\}$,
we denote by $x_{j}$ the corresponding filling point of the $j$-th
copy of $\varSigma_{0,1}^{1}$ in $\varSigma_{0,1}^{1}\natural\cdots\natural\varSigma_{0,1}^{1}\natural\varSigma_{g,1}^{0}$.
Also, we denote by $\gamma_{j}$ the corresponding path $\gamma_{1}$
of the $j$-th copy of $\varSigma_{0,1}^{1}$  in $\varSigma_{0,1}^{1}\natural\cdots\natural\varSigma_{0,1}^{1}\natural\varSigma_{g,1}^{0}$;
see Figure \ref{fig:Generators-and-paths}. For all natural numbers
$n$, we denote the surface $\varSigma_{0,1}^{[x_{1}]}\natural\varSigma_{g,1}^{n}$
by $\varSigma_{g,1}^{[x_{1}],n}$.

Now, we fix a canonical isomorphism between the group $H_{n}$ and
the fundamental group of $\varSigma_{g,1}^{[x_{1}],n}$ with $x_{1}$
as basepoint.

\begin{figure}
~~~~\includegraphics[clip,scale=0.3]{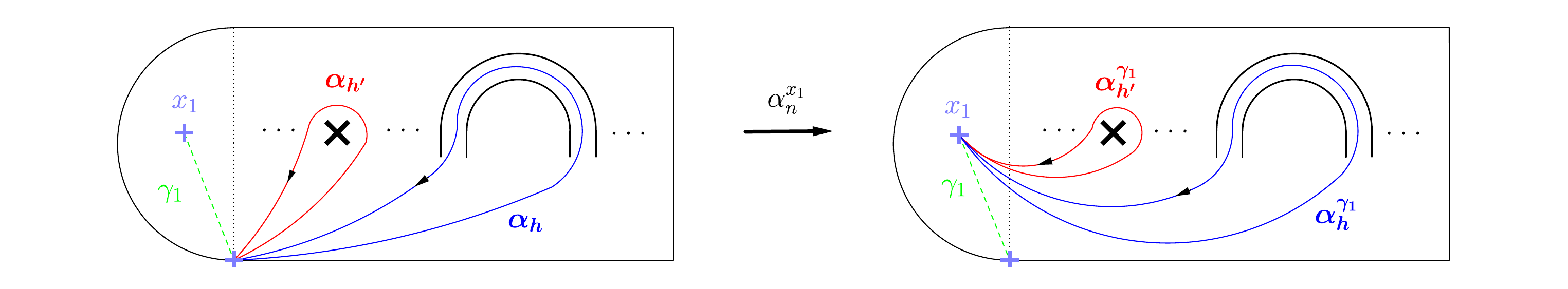}

\caption{\label{fig:alpha^=00007Bgamma=00007D}The simple closed curve $\alpha^{\gamma_{1}}$}
\end{figure}

\begin{defn}
\label{def:defisocano} For a generator $h$ of $H_{n}$ and its representative
simple closed curve $\alpha_{h}$, we consider the path $\gamma_{1}\cdot\alpha_{h}\cdot\gamma_{1}^{-1}$
obtained by changing the basepoint of the curve $\alpha_{h}$ from
$p$ to $x_{1}$ via the path $\gamma_{1}$. This procedure for each
generator $h$ of $H_{n}$ induces an isomorphism $\alpha_{n}^{x_{1}}:H_{n}\overset{\sim}{\rightarrow}\pi_{1}(\varSigma_{g,1}^{[x_{1}],n},x_{1})$,
which is independent of the choice of path $\gamma_{1}$ up to homotopy
of $\varSigma_{0,1}^{[x_{1}]}$. For a generator $h$ of $H_{n}$,
we denote by $\alpha_{h}^{\gamma_{1}}$ a simple closed curve of $\varSigma_{g,1}^{[x_{1}],n}$
based at $x_{1}$, homotopic to $\gamma_{1}\cdot\alpha_{h}\cdot\gamma_{1}^{-1}$
and thus representing the element $\alpha_{n}^{x_{1}}(h)$; see Figure
\ref{fig:alpha^=00007Bgamma=00007D} for an illustration.
\end{defn}

For all natural numbers $n$, let $\mathbf{\Gamma}_{g,1}^{[1],n}$
be the subgroup of the mapping class group $\mathbf{\Gamma}_{g,1}^{1+n}$
where the first puncture $p_{1}$ is sent to itself. We denote by
$\mathscr{E}_{n}:\mathbf{\Gamma}_{g,1}^{[1],n}\hookrightarrow\mathbf{\Gamma}_{g,1}^{1+n}$
the associated canonical embedding. Note that the group $\mathbf{\Gamma}_{g,1}^{[1],n}$
is also clearly isomorphic to the isotopy classes of homeomorphisms
of the surface $\varSigma_{g,1}^{[x_{1}],n}$ restricting to the identity
on the boundary component, fixing the punctures setwise and sending
the marked point $x_{1}$ to itself; see \cite[Section 1.1.1]{farb2011primer}.\\

The central ingredient to define the morphism $\chi_{n}:\pi_{1}(\varSigma_{g,1}^{n},p)\rightarrow\mathbf{\Gamma}_{g,1}^{1+n}$
is the use of the well-known \textit{point pushing map} $\textrm{Push}_{p_{1}}:\pi_{1}(\varSigma_{g,1}^{[x_{1}],n},x_{1})\hookrightarrow\mathbf{\Gamma}_{g,1}^{[1],n}$:
it is basically a reinterpretation of the morphism $\textrm{Push}_{p_{1}}$
in our context, namely replacing the source with $\pi_{1}(\varSigma_{g,1}^{n},p)$
via the isomorphism $\alpha_{n}^{x_{1}}$ of Definition \ref{def:defisocano}.
The formula (\ref{eq:push_formula}) below describes the morphism
$\textrm{Push}_{p_{1}}$ and we refer the reader to \cite[Section 4.2.1]{farb2011primer}
for further details. For a simple loop $\varepsilon$, we denote by
$\tau_{\varepsilon}$ the Dehn twist along it. We consider a generator
$h$ of $H_{n}$ and the associated simple closed curve $\alpha_{h}^{\gamma_{1}}$
of Definition \ref{def:defisocano}. Let $(\alpha_{h}^{\gamma_{1}})^{-}$
and $(\alpha_{h}^{\gamma_{1}})^{+}$ be the isotopy classes of the
simple closed curves in $\varSigma_{g,1}^{[x_{1}],n}$ obtained by
pushing the simple closed curve $\alpha_{h}^{\gamma_{1}}$ off itself
to the left and right respectively. It follows from \cite[Fact 4.7]{farb2011primer}
that:
\begin{equation}
\textrm{Push}_{p_{1}}(\alpha_{n}^{x_{1}}(h))=\tau_{(\alpha_{h}^{\gamma_{1}})^{-}}\circ\tau_{(\alpha_{h}^{\gamma_{1}})^{+}}^{-1}.\label{eq:push_formula}
\end{equation}
Figure \ref{fig:push_map_first} allows us to visualize the effect
of the point-pushing map; see also \cite[Section 4.2.1]{farb2011primer}.
In this figure, $f_{1}$ denotes the generator of $\pi_{1}(\varSigma_{0,1}^{1},p)$
corresponding to the first copy of $\varSigma_{0,1}^{1}$ in $H_{n}$;
note that this $\varSigma_{0,1}^{1}$ is the second copy of $\varSigma_{0,1}^{1}$
in $\varSigma_{g,1}^{1+n}=\varSigma_{0,1}^{1}\natural\varSigma_{0,1}^{1}\natural\varSigma_{g,1}^{n-1}$.
We now can define each morphism $\chi_{n}$ from the morphisms $\alpha_{n}^{x_{1}}$,
$\textrm{Push}_{p_{1}}$ and $\mathscr{E}_{n}$:
\begin{defn}
\label{def:defsigma1}Let $n$ be a natural number. We define the
morphism $\chi_{n,1}:H_{n}\rightarrow\mathbf{\Gamma}_{g,1}^{1+n}$
to be the composition:
\[
\mathscr{E}_{n}\circ\textrm{Push}_{p_{1}}\circ\alpha_{n}^{x_{1}}:H_{n}\overset{\sim}{\rightarrow}\pi_{1}(\varSigma_{g,1}^{[x_{1}],n},x_{1})\hookrightarrow\mathbf{\Gamma}_{g,1}^{[1],n}\hookrightarrow\mathbf{\Gamma}_{g,1}^{1+n}.
\]
In particular, when $g=n=0$, the morphism $\chi_{0,1}:\pi_{1}(\varSigma_{0,1}^{0},p)\cong0_{\mathfrak{Gr}}\rightarrow\mathbf{\Gamma}_{0,1}^{1}\cong0_{\mathfrak{Gr}}$
is trivial. Also, using the notations of Definition \ref{def:defisocano}
and following (\ref{eq:push_formula}), if $2g+n\geq1$ then the morphisms
$\chi_{n,1}$ sends $h$ to $\tau_{(\alpha_{h}^{\gamma_{1}})^{-}}\circ\tau_{(\alpha_{h}^{\gamma_{1}})^{+}}^{-1}$
for each $n$ and each generator $h$ either of $H_{0}$ or of one
of the copies of $H$ in $H_{n}$. Each element of $H_{n}$ is thus
identified with a $1$-strand braid on the surface $\varSigma_{g,1}^{n}$.
For instance, we denote by $\mathbf{PB}_{2}$ the pure braid group
$\mathbf{\Gamma}_{0,1}^{[1],1}$ and by $\sigma_{1}$ the Artin generator
of the braid group on two strands $\mathbf{B}_{2}$ which is identified
with the braiding $b_{\varSigma_{0,1}^{1},\varSigma_{0,1}^{1}}^{\mathcal{M}_{2}}$
as pictured in Figure \ref{fig:tressageM2}. Then, with the orientation
convention of Figure \ref{fig:push_map_first}, the morphism $\chi_{1,1}:\pi_{1}(\varSigma_{0,1}^{1},p)\rightarrow\mathbf{PB}_{2}$
sends the generator $f_{1}$ of $\pi_{1}(\varSigma_{0,1}^{1},p)$
to $\sigma_{1}^{2}$.
\end{defn}

\begin{figure}
~~~~\includegraphics[clip,scale=0.3]{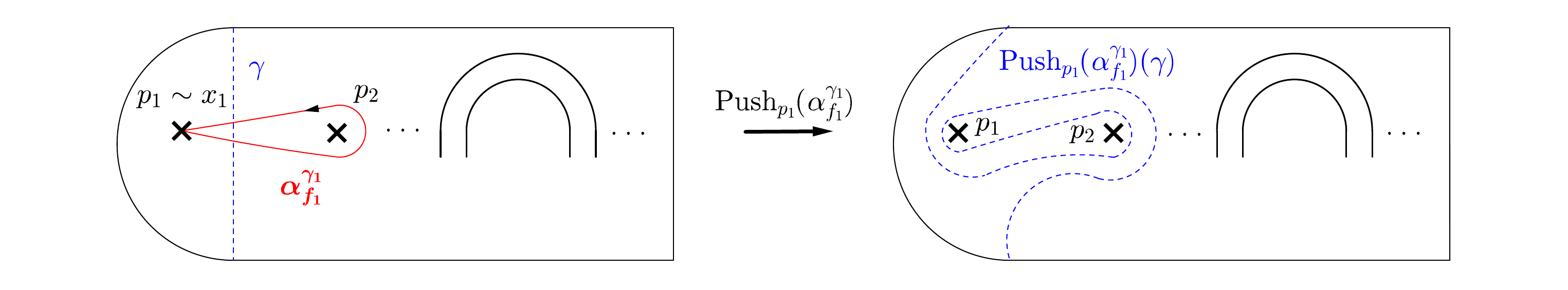}

\caption{\label{fig:push_map_first}Illustration of the point-pushing map $\textrm{Push}_{p_{1}}$}
\end{figure}

In addition, we recall from \cite[Theorem 4.6]{farb2011primer} that
the point pushing map $\textrm{Push}_{p_{1}}$ is part of the following
\textit{Birman short exact sequence} (\ref{eq:Birman}): for all natural
numbers $n$, there is a short exact sequence
\begin{equation}
\xymatrix{1\ar@{->}[r] & \pi_{1}(\varSigma_{g,1}^{[x_{1}],n},x_{1})\ar@{->}[r]^{\,\,\,\,\,\,\,\,\,\,\textrm{Push}_{p_{1}}} & \mathbf{\Gamma}_{g,1}^{[1],n}\ar@{->}[r]^{\textrm{Forget}} & \mathbf{\Gamma}_{g,1}^{n}\ar@{->}[r] & 1.}
\label{eq:Birman}
\end{equation}
Namely, viewing the surface $\varSigma_{g,1}^{[x_{1}],n}$ as the
surface obtained from $\varSigma_{g,1}^{n}=\varSigma_{0,1}^{0}\natural\varSigma_{g,1}^{n}$
by marking a point $x_{1}$ in the interior of $\varSigma_{0,1}^{0}$
and denoting by $\mathscr{P}$ the set $\{p_{2},\ldots,\text{\ensuremath{p_{n+1}}\}}$
of punctures of $\varSigma_{g,1}^{n}$, we recall that (\ref{eq:Birman})
is the long exact sequence in homotopy of the locally trivial fibration
$\textrm{Homeo}^{I}(\varSigma_{g,1}^{n},\mathscr{P})\twoheadrightarrow\varSigma_{g,1}^{n}$
defined by evaluating homeomorphisms at $x_{1}$ in $\varSigma_{g,1}^{n}$;
see \cite[Section 4.2.3]{farb2011primer} for further details. In
particular, the morphism $\textrm{Forget}$ is the map induced in
homotopy by the fiber inclusion defined by forgetting that the point
$x_{1}$ is marked. This short exact sequence has the following stronger
property, which is probably known to the experts.
\begin{lem}
\label{lem:isosemidirectproduct} Let $n$ be a natural number. The
Birman short exact sequence (\ref{eq:Birman}) splits, inducing an
isomorphism $\mathscr{B}_{n}:H_{n}\rtimes_{\mathcal{A}_{\varSigma_{g,1}^{n}}}\mathbf{\Gamma}_{g,1}^{n}\overset{\sim}{\rightarrow}\mathbf{\Gamma}_{g,1}^{[1],n}$.
\end{lem}

\begin{proof}
Identifying $\varSigma_{g,1}^{[x_{1}],n}$ with $\varSigma_{0,1}^{[x_{1}]}\natural\varSigma_{g,1}^{n}$,
the surface $\varSigma_{g,1}^{n}$ can be viewed as a subsurface of
$\varSigma_{g,1}^{[x_{1}],n}$ as the complement of the disc $\varSigma_{0,1}^{[x_{1}]}$
with the marked point $x_{1}$. Extending the elements of $\textrm{Homeo}^{I}(\varSigma_{g,1}^{n},\mathscr{P})$
by the identity on $\varSigma_{0,1}^{[x_{1}]}$ induces in homotopy
the morphism $s_{n}:\mathbf{\Gamma}_{g,1}^{n}\rightarrow\mathbf{\Gamma}_{g,1}^{[1],n}$.
This is known as the stabilization map for the mapping class groups
with respect to the punctures; see \cite{boedigtilman2} for instance.
It is a routine to check from the definitions that $s_{n}$ is a section
of the map $\textrm{Forget}$ (which implies in particular that $s_{n}$
is injective), thus providing a splitting of the exact sequence (\ref{eq:Birman}).
Hence we have an isomorphism $\pi_{1}(\varSigma_{g,1}^{[x_{1}],n},x_{1})\rtimes_{\mathcal{A}_{x_{1}}}\mathbf{\Gamma}_{g,1}^{n}\cong\mathbf{\Gamma}_{g,1}^{[1],n}$,
where $\mathcal{A}_{x_{1}}$ denotes the natural action of $\mathbf{\Gamma}_{g,1}^{n}$
on $\pi_{1}(\varSigma_{g,1}^{[x_{1}],n},x_{1})$. Then, since the
splitting is induced by considering $\varSigma_{g,1}^{n}$ as a subsurface
of $\varSigma_{g,1}^{[x_{1}],n}$, the precomposition by the canonical
isomorphism $\alpha_{n}^{x_{1}}$ of Definition \ref{def:defisocano}
on $\pi_{1}(\varSigma_{g,1}^{[x_{1}],n},x_{1})$ induces the required
isomorphism.
\end{proof}
Therefore all the above ingredients define an appropriate Long-Moody
functor:
\begin{lem}
\label{lem:sigman1satisfiesfirtcond} The setting $\{\pi_{1}(\varSigma_{g,1}^{-},p),\mathfrak{M}_{2}^{+,g},\mathfrak{M}_{2},\chi_{1}\}$
is a Long-Moody system.
\end{lem}

\begin{proof}
It is clear from the definitions that the composition $\mathscr{E}_{n}\circ s_{n}:\mathbf{\Gamma}_{g,1}^{n}\hookrightarrow\mathbf{\Gamma}_{g,1}^{[1],n}\hookrightarrow\mathbf{\Gamma}_{g,1}^{1+n}$
is equal to the injection $id_{\varSigma_{0,1}^{1}}\natural-$. Also,
it follows from Lemma \ref{lem:isosemidirectproduct} that the restriction
of the isomorphism $\mathscr{B}_{n}:H_{n}\rtimes_{\mathcal{A}_{\varSigma_{g,1}^{n}}}\mathbf{\Gamma}_{g,1}^{n}\overset{\sim}{\rightarrow}\mathbf{\Gamma}_{g,1}^{[1],n}$
to $H_{n}$ is equal to the composition $\textrm{Push}_{p_{1}}\circ\alpha_{n}^{x_{1}}$.
Hence, the following diagram is commutative:
\[
\xymatrix{H_{n}\ar@{^{(}->}[r]\ar@{->}[dr]_{\chi_{n,1}} & H_{n}\rtimes_{\mathcal{A}_{\varSigma_{g,1}^{n}}}\mathbf{\Gamma}_{g,1}^{n}\ar@{->}[d]^{\mathscr{E}_{n}\circ\mathscr{B}_{n}} & \mathbf{\Gamma}_{g,1}^{n}\ar@{_{(}->}[l]\ar@{->}[dl]^{id_{\varSigma_{0,1}^{1}}\natural-}\\
 & \mathbf{\Gamma}_{g,1}^{1+n}.
}
\]
Hence it follows Lemma \ref{lem:equivcond:coherenceconditionsigmanan}
that the diagram (\ref{cond:coherenceconditionsigmanan}) of Definition
\ref{def:longmoodysystem} is commutative, which ends the proof.
\end{proof}
Moreover, we have the stronger property:
\begin{prop}
\label{lem:sigma1satisfiesCondition}\label{cor:long-moodysystemmarkedpoints}
The setting $\{\pi_{1}(\varSigma_{g,1}^{-},p),\mathfrak{M}_{2}^{+,g},\mathfrak{M}_{2},\chi_{1}\}$
is a coherent Long-Moody system.
\end{prop}

\begin{proof}
By Proposition \ref{cond:conditionstability} and Lemma \ref{lem:sigman1satisfiesfirtcond},
it is enough to prove that the morphism $\chi_{n,1}$ satisfies the
equality (\ref{eq:equiva}) in the present context and for any natural
number $n$. Let $f$ be a generator either of $H_{0}$ or of one
of the copies of $H$ in $H_{n}$. We recall that the braiding of
$\mathfrak{M}_{2}$ is defined by doing half a Dehn twist in a pair
of pants neighbourhood of the boundaries of two copies of $\varSigma_{0,1}^{1}$
in $\varSigma_{g,1}^{2+n}$. Hence the morphism $b_{\varSigma_{0,1}^{1},\varSigma_{0,1}^{1}}^{\mathfrak{M}_{2}}\natural id_{n}$
is the element $\sigma_{1}\in\mathbf{B}_{2}\leq\mathbf{\Gamma}_{g,1}^{2+n}$
which exchanges the punctures $p_{1}$ and $p_{2}$ in $\varSigma_{g,1}^{2+n}$;
see Figure \ref{fig:tressageM2}. Note that $\pi_{1}(-,p)([\varSigma_{0,1}^{1},id_{\varSigma_{0,1}^{1+n}}])(f)=id_{\pi_{1}(\varSigma_{0,1}^{1},p)}*f$
in $H_{n+1}$, and we denote this element by $f'$ for simplicity.
Hence, following the equality (\ref{eq:equiva}), it is enough to
prove that as elements of $\mathbf{\Gamma}_{g,1}^{2+n}$:
\begin{eqnarray}
\sigma_{1}\circ\chi_{n+1,1}(f')\circ\sigma_{1}^{-1} & = & id_{\varSigma_{0,1}^{1}}\natural\chi_{n,1}(f).\label{eq:equalitylemma2.54}
\end{eqnarray}
If $g=n=0$, then the equality (\ref{eq:equalitylemma2.54}) is trivially
checked since $f$ and $f'$ are trivial. We thus assume that $2g+n\geq1$.
Let $x_{1}$ and $x_{2}$ be marked points filling in the punctures
$p_{1}$ and $p_{2}$ respectively of the two first copies of $\varSigma_{0,1}^{1}$
in $\varSigma_{g,1}^{2+n}=\varSigma_{0,1}^{1}\natural\varSigma_{0,1}^{1}\natural\varSigma_{g,1}^{n}$,
and we denote by $\varSigma_{0,1}^{[x_{1}]}\natural\varSigma_{0,1}^{[x_{2}]}\natural\varSigma_{g,1}^{n}$
the obtained surface. Let $\gamma_{1}$ and $\gamma_{2}$ respectively
be paths in $\varSigma_{0,1}^{[x_{1}]}$ and $\varSigma_{0,1}^{[x_{2}]}$
connecting the point $p\in I$ to $x_{1}$ and $x_{2}$; see Figure
\ref{fig:Generators-and-paths}. Similarly to Definition \ref{def:defisocano},
for each generator $h$ of $H_{n}$ and its representative curve $\alpha_{h}$,
we denote by $\alpha_{h}^{\gamma_{1}}$ (respectively $\alpha_{h}^{\gamma_{2}}$)
a simple closed curve of $\varSigma_{0,1}^{[x_{1}]}\natural\varSigma_{0,1}^{[x_{2}]}\natural\varSigma_{g,1}^{n}$
based at $x_{1}$ (respectively $x_{2}$) obtained after changing
the basepoint of the curve $\alpha_{h}$ from $p$ to $x_{1}$ (respectively
$x_{2}$) via the path $\gamma_{1}$ (respectively $\gamma_{2}$).
Then it follows from (\ref{eq:push_formula}), Definition \ref{def:defsigma1}
and \cite[Fact 3.7]{farb2011primer} that
\[
\sigma_{1}\circ\chi_{n+1,1}(f')\circ\sigma_{1}^{-1}=(\sigma_{1}\circ\tau_{(\alpha_{f'}^{\gamma_{1}})^{-}}\circ\sigma_{1}^{-1})\circ(\sigma_{1}\circ\tau_{(\alpha_{f'}^{\gamma_{1}})^{+}}^{-1}\circ\sigma_{1}^{-1})=\tau_{\sigma_{1}((\alpha_{f'}^{\gamma_{1}})^{-})}\circ\tau_{\sigma_{1}((\alpha_{f'}^{\gamma_{1}})^{+})}^{-1},
\]
and using the shift in the enumeration given by adding a surface $\varSigma_{0,1}^{1}$
on the left that
\[
id_{\varSigma_{0,1}^{1}}\natural\chi_{n,1}(f)=id_{\varSigma_{0,1}^{1}}\natural(\tau_{(\alpha_{f}^{\gamma_{1}})^{-}}\circ\tau_{(\alpha_{f}^{\gamma_{1}})^{+}}^{-1})=\tau_{(\alpha_{f}^{\gamma_{2}})^{-}}\circ\tau_{(\alpha_{f}^{\gamma_{2}})^{+}}^{-1}.
\]
The braiding $\sigma_{1}$ exchanges the two first punctures $p_{1}$
and $p_{2}$ of $\varSigma_{g,1}^{2+n}$ as in Figure \ref{fig:tressageM2}.
In particular, the image of the curve $\alpha_{f'}^{\gamma_{1}}$
by $\sigma_{1}$ is isotopic to $\alpha_{f}^{\gamma_{2}}$. Therefore
$(\alpha_{f}^{\gamma_{2}})^{-}=\sigma_{1}((\alpha_{f'}^{\gamma_{1}})^{-})$
and $(\alpha_{f}^{\gamma_{2}})^{+}=\sigma_{1}((\alpha_{f'}^{\gamma_{1}})^{+})$
as isotopy classes of simple closed curves. Since $\tau_{\gamma}=\tau_{\theta}$
for $\gamma$ and $\theta$ two isotopy classes of simple closed curves
which are equal (see \cite[Fact 3.6]{farb2011primer} for instance),
we deduce that the equality (\ref{eq:equalitylemma2.54}) is satisfied,
which ends the proof.
\end{proof}

\paragraph{Non-orientable surfaces.}

Instead of considering orientable surfaces, we can fix a natural number
$h\geq2$ and consider the non-orientable surfaces $D^{\natural n}\natural M^{\natural h}$
for all natural numbers $n$, that we denote by $\mathscr{N\varSigma}_{h,1}^{n}$.
We denote the corresponding mapping class group by $\boldsymbol{\mathcal{N}}_{h,1}^{n}$.
Let $\mathfrak{M}_{2}^{-,h}$ be the small full subgroupoid of $\mathfrak{M}_{2}$
on these objects. In this case, $H$ is the group $\pi_{1}(\mathscr{N\varSigma}_{0,1}^{1},p)\cong\mathbb{Z}$,
$H_{0}$ is the free group $\pi_{1}(\mathscr{N\varSigma}_{h,1}^{0},p)\cong\mathbf{F}_{h}$
and $H_{n}=H^{*n}*H_{0}=\pi_{1}(\mathscr{N\varSigma}_{h,1}^{n},p)$
for all natural numbers $n$.

Then the above work adapts mutatis mutandis to this situation. Let
$\boldsymbol{\mathcal{N}}_{h,1}^{[1],n}$ be the subgroup of $\boldsymbol{\mathcal{N}}_{h,1}^{1+n}$
of isotopy classes of homeomorphisms of the surface $\mathscr{N\varSigma}_{h,1}^{1+n}$
restricting to the identity on the boundary component, fixing the
first puncture and fixing the other ones setwise. The non-orientable
version of the Birman short exact sequence (\ref{eq:Birman}) is obtained
by the long exact sequence on homotopy associated with the locally
trivial fibration $\textrm{Homeo}^{I}(\mathscr{N\varSigma}_{h,1}^{n},\mathscr{P})\rightarrow\mathscr{N\varSigma}_{h,1}^{n}$,
defined by evaluating homeomorphisms at $x_{1}$ in $\mathscr{N\varSigma}_{h,1}^{n}$.
The associated fiber inclusion induces the map $\boldsymbol{\mathcal{N}}_{h,1}^{[1],n}\twoheadrightarrow\boldsymbol{\mathcal{N}}_{h,1}^{n}$
which forgets that the point $x_{1}$ is marked. Moreover, identifying
$\mathscr{N\varSigma}_{h,1}^{[x_{1}],n}$ with $\varSigma_{0,1}^{[x_{1}]}\natural\mathscr{N\varSigma}_{h,1}^{n}$,
extending the homeomorphisms of $\mathscr{N\varSigma}_{h,1}^{n}$
to $\mathscr{N\varSigma}_{h,1}^{[x_{1}],n}$ by the identity on $\varSigma_{0,1}^{[x_{1}]}$
defines a section $\boldsymbol{\mathcal{N}}_{h,1}^{n}\rightarrow\boldsymbol{\mathcal{N}}_{h,1}^{[1],n}$.
We thus define $\chi_{n,2}:H_{n}\hookrightarrow H_{n}\rtimes\boldsymbol{\mathcal{N}}_{h,1}^{n}\cong\boldsymbol{\mathcal{N}}_{h,1}^{[1],n}\hookrightarrow\boldsymbol{\mathcal{N}}_{h,1}^{1+n}$
for all $n$. The proofs of Lemma \ref{lem:sigman1satisfiesfirtcond}
and Proposition \ref{lem:sigma1satisfiesCondition} repeat mutatis
mutandis and we deduce that:
\begin{prop}
\label{prop:Analoguefornon-orient}The setting $\{\pi_{1}(\mathscr{N\varSigma}_{h,1}^{-},p),\mathfrak{M}_{2}^{-,h},\mathfrak{M}_{2},\chi_{2}\}$
is a coherent Long-Moody system.
\end{prop}

\subsubsection{Surface braid groups\label{subsec:Surface-braid-groups:}}

Drawing on the work of Section \ref{subsec:Modifyingpunctures}, we
can construct Long-Moody functors for the surface braid groups. We
review some basic results about these groups and refer the reader
to \cite{birmanbraids} or \cite{GuaschiJuan-Pineda} for further
details.

Let $(S,I)$ be an object of $\mathfrak{M}_{2}$ where $S$ is a surface
without punctures and $n$ be a natural number. We denote by $\mathfrak{S}_{n}$
the symmetric group on a set of $n$ elements. For a partition $(\lambda_{1},\lambda_{2})$
of $n$ we denote by $\mathfrak{S}_{\lambda_{1},\lambda_{2}}$ the
product $\mathfrak{S}_{\lambda_{1}}\times\mathfrak{S}_{\lambda_{2}}$;
for the partitions $(n,0)$ or $(0,n)$ we simply denote it by $\mathfrak{S}_{n}$.
Let $F_{n}(S)=\{(x_{1},\ldots,x_{n})\in\mathring{S}^{\times n}\mid x_{i}\neq x_{j}\textrm{ if }i\neq j\}$
be the \textit{ordered} configuration spaces of $n$ distinct points
in the interior of $S$. For the partition $(\lambda_{1},\lambda_{2})$
of $n$ , we denote by $C_{(\lambda_{1},\lambda_{2})}(S)$ the quotient
$F_{n}(S)/\mathfrak{S}_{\lambda_{1},\lambda_{2}}$ and its fundamental
group by $\mathbf{B}_{\lambda_{1},\lambda_{2}}(S)$. For the partitions
$(n,0)$ or $(0,n)$, the set $C_{n}(S)$ corresponds to the classical
\textit{unordered} configuration space of $n$ distinct points in
the interior of $S$. The surface braid group $\mathbf{B}_{n}(S)$
is the fundamental group of $\ensuremath{C_{n}(S)}$.

Let $\mathcal{\mathscr{P}}$ be a set of $n$ points removed from
the interior of $S$ and $(\lambda_{1},\lambda_{2})$ be a partition
of $n$. Identifying $S$ with the boundary connected sum $\mathbb{D}^{2}\natural S$,
the surface $S\setminus\mathcal{\mathscr{P}}$ can be viewed as a
subsurface of $S$ obtained by removing $\mathcal{\mathscr{P}}$ in
the interior of $\mathbb{D}^{2}$. We denote by $\pi_{0}\textrm{Homeo}^{I}(S,(\lambda_{1},\lambda_{2}))$
the isotopy classes of homeomorphisms of $S$ restricting to the identity
on a neighborhood of the parametrized interval $I$ and fixing the
punctures of $\lambda_{1}$ and $\lambda_{2}$ setwise.
\begin{prop}
For the partition $(\lambda_{1},\lambda_{2})$ of a natural number
$n$, we have a \textup{short} exact sequence:
\end{prop}

\begin{equation}
\xymatrix{1\ar@{->}[r] & \mathbf{B}_{\lambda_{1},\lambda_{2}}(S)\ar@{->}[r]^{i_{\lambda}^{\mathfrak{b}}\,\,\,\,\,\,\,\,\,\,\,\,\,\,\,\,\,\,\,} & \pi_{0}\mathrm{\textrm{Homeo\ensuremath{^{I}}}}(S,(\lambda_{1},\lambda_{2}))\ar@{->}[r]^{\,\,\,\,\,\,\,\,\,\,\,\,p_{\lambda}^{\mathfrak{b}}} & \pi_{0}\mathrm{\textrm{Homeo\ensuremath{^{I}}}}(S)\ar@{->}[r] & 1}
\label{eq:Birman_2}
\end{equation}

\begin{proof}
We consider the map $\epsilon:\textrm{Homeo\ensuremath{^{I}}}(S)\rightarrow F_{n}(S)/\mathfrak{S}_{\lambda_{1},\lambda_{2}}$
defined by evaluating homeomorphisms at a base configuration in $S\setminus\mathcal{\mathscr{P}}$.
Let $\epsilon_{F}:\textrm{Homeo\ensuremath{^{I}}}(S)\rightarrow F_{n}(S)$
be the map defined by evaluating homeomorphisms at a base configuration
in $S\setminus\mathcal{\mathscr{P}}$. We recall from \cite[Lemma 1.2]{BirmanMCGBraid}
that $\epsilon_{F}$ is a locally trivial fibration, whose fiber identifies
with the group of homeomorphisms of $S$ restricting to the identity
on a neighborhood of the parametrized interval $I$ and fixing each
puncture pointwise denoted by $\textrm{Homeo\ensuremath{^{I}}}(S,[n])$.
The natural action of $\mathfrak{S}_{\lambda_{1},\lambda_{2}}$ on
$F_{n}(S)$ defined by permutation of coordinates is free. Therefore
the associated canonical projection $p:F_{n}(S)\rightarrow F_{n}(S)/\mathfrak{S}_{\lambda_{1},\lambda_{2}}$
is a regular covering space map with deck transformation group $\mathfrak{S}_{\lambda_{1},\lambda_{2}}$;
see \cite[Proposition 1.40]{hatcher2002algebraic} for instance. In
particular, the map $p$ defines a locally trivial fibration with
discrete fiber $\mathfrak{S}_{\lambda_{1},\lambda_{2}}$. Hence the
evaluation map $\epsilon$ is equal to the composition $p\circ\epsilon_{F}$
and it is therefore a locally trivial fibration. Its fiber identifies
with $\textrm{Homeo\ensuremath{^{I}}}(S,(\lambda_{1},\lambda_{2}))$.
We denote by $\widehat{\omega}_{\lambda}:\textrm{Homeo\ensuremath{^{I}}}(S,(\lambda_{1},\lambda_{2}))\rightarrow\textrm{Homeo\ensuremath{^{I}}}(S)$
the associated inclusion. We obtain the exact sequence (\ref{eq:Birman_2})
from the long exact sequence of homotopy groups of this locally trivial
fibration, using the fact that $\pi_{1}(\textrm{Homeo\ensuremath{^{I}}}(S))=0$
by \cite{Hamstrom}.
\end{proof}
In particular, we consider the surface braid group $\mathbf{B}_{n}(S)$
as a normal subgroup of the mapping class group $\pi_{0}\textrm{Homeo\ensuremath{^{I}}}(S,n)$
via (\ref{eq:Birman_2}) for the partition $(0,n)$. We define $\mathfrak{B}_{2}$
to be the subgroupoid of $\mathfrak{M}_{2}$ with the same objects
and with morphisms for each surface $S$ given by the surface braid
group $\mathbf{B}_{n}(S)$. The monoidal structure $(\mathfrak{M}_{2},\natural,0)$
clearly restricts to a braided monoidal structure on the subgroupoid
$\mathfrak{B}_{2}$, denoted in the same way $(\mathfrak{B}_{2},\natural,0)$.
We thus have a canonical inclusion $\mathfrak{U}\mathfrak{B}_{2}\hookrightarrow\mathfrak{U}\mathfrak{M}_{2}$.
We denote by $\pi_{1}(-,p)^{\mathfrak{b}}$ the restriction of the
fundamental group functor $\pi_{1}(-,p):\mathfrak{U}\mathfrak{M}_{2}\rightarrow\mathfrak{gr}$
introduced in Section \ref{subsec:symplecticfunctor } along this
inclusion. In particular, it follows from Lemma \ref{lem:symplecticstrongmon}
that the functor $\pi_{1}(-,p)^{\mathfrak{b}}:(\mathfrak{B}_{2},\natural,\mathbb{D}^{2})\rightarrow(\mathfrak{gr},*,0_{\mathfrak{Gr}})$
is strict monoidal.

We fix a natural number $g\geq0$ throughout Section \ref{subsec:Surface-braid-groups:}
and consider the surface $\varSigma_{g,1}$. Let $\mathfrak{B}_{2}^{g}$
be the small full subgroupoid of $\mathfrak{B}_{2}$ on the objects
$\underline{n}=1^{\natural n}\natural\underline{0}$ where $\underline{0}:=\varSigma_{g,1}^{0}$
and $1:=\varSigma_{0,1}^{1}$. In the notations of Section \ref{subsec:Definition-of-theLM},
the groupoid $\mathfrak{B}_{2}$ plays the role of $\mathcal{G}'$
and $\mathfrak{B}_{2}^{g}$ corresponds to $\mathcal{G}$.

As in Section \ref{subsec:Modifyingpunctures}, let $H$ be the free
group $\pi_{1}(\varSigma_{0,1}^{1},p)\cong\mathbf{F}_{1}$, $H_{0}$
be the free group $\pi_{1}(\varSigma_{g,1}^{0},p)\cong\mathbf{F}_{2g}$
and thus $H_{n}=H^{*n}*H_{0}=\pi_{1}(\varSigma_{g,1}^{n},p)$ for
all natural numbers $n$. Precomposing by $\mathfrak{U}\mathfrak{B}_{2}^{g}\rightarrow\mathfrak{U}\mathfrak{B}_{2}$,
the restriction of the functor $\pi_{1}(-,p)^{\mathfrak{b}}$ to $\mathfrak{U}\mathfrak{B}_{2}^{g}$
sends $\varSigma_{g,1}^{n}$ to $H_{n}$. We denote by $\mathcal{A}_{\varSigma_{g,1}^{n}}^{\mathfrak{b}}:\mathbf{B}_{n}(\varSigma_{g,1})\rightarrow\textrm{Aut}(\pi_{1}(\varSigma_{g,1}^{n},p))$
the associated natural action for natural number $n$: it corresponds
to the restriction of the natural action $\mathcal{A}_{\varSigma_{g,1}^{n}}:\mathbf{\Gamma}_{g,1}^{n}\rightarrow\textrm{Aut}(\pi_{1}(\varSigma_{g,1}^{n},p))$
to the normal subgroup $\mathbf{B}_{n}(\varSigma_{g,1})\hookrightarrow\mathbf{\Gamma}_{g,1}^{n}$.

For sake of coherence with the notation of Section \ref{subsec:Modifyingpunctures},
we henceforth denote the group $\mathbf{B}_{1,n}(\varSigma_{g,1})$
by $\mathbf{B}_{[1],n}(\varSigma_{g,1})$ since it is the kernel of
the morphism $\mathbf{\Gamma}_{g,1}^{[1],n}\twoheadrightarrow\mathbf{\Gamma}_{g,1}$
by (\ref{eq:Birman_2}). The following lemma is the key to define
a Long-Moody system in the present situation.
\begin{lem}
\label{lem:splitting_surface_braids}For all $n\in\mathbb{N}$, there
is an isomorphism $\mathscr{B}_{n}^{\mathfrak{b}}:H_{n}\rtimes_{\mathcal{A}_{\varSigma_{g,1}^{n}}^{\mathfrak{b}}}\mathbf{B}_{n}(\varSigma_{g,1})\overset{\sim}{\rightarrow}\mathbf{B}_{[1],n}(\varSigma_{g,1})$.
In particular, there is a canonical injection $\textrm{Push}_{p_{1}}^{\mathfrak{b}}:H_{n}\hookrightarrow\mathbf{B}_{[1],n}(\varSigma_{g,1})$
such that $i_{(1,n)}^{\mathfrak{b}}\circ\textrm{Push}_{p_{1}}^{\mathfrak{b}}=\textrm{Push}_{p_{1}}\circ\alpha_{n}^{x_{1}}$.
\end{lem}

\begin{proof}
It follows from the definitions of the morphism $\textrm{Forget}$
of the exact sequence (\ref{eq:Birman}), of its splitting $s_{n}:\mathbf{\Gamma}_{g,1}^{n}\hookrightarrow\mathbf{\Gamma}_{g,1}^{[1],n}$
introduced in Lemma \ref{lem:isosemidirectproduct} and of the morphisms
$p_{(1,n)}^{\mathfrak{b}}$ and $p_{n}^{\mathfrak{b}}$ of (\ref{eq:Birman_2})
that the following diagrams are commutative:

\[
\xymatrix{\mathbf{\Gamma}_{g,1}^{[1],n}\ar@{->>}[dr]_{p_{(1,n)}^{\mathfrak{b}}}\ar@{->>}[rr]^{\textrm{Forget}} &  & \mathbf{\Gamma}_{g,1}^{n}\ar@{->>}[dl]^{p_{n}^{\mathfrak{b}}}\\
 & \mathbf{\Gamma}_{g,1},
}
\textrm{ \,\,\,\, \,\,\,\, }\xymatrix{\mathbf{\Gamma}_{g,1}^{[1],n}\ar@{->>}[dr]_{p_{(1,n)}^{\mathfrak{b}}} &  & \mathbf{\Gamma}_{g,1}^{n}\ar@{->>}[dl]^{p_{n}^{\mathfrak{b}}}\ar@{_{(}->}[ll]_{s_{n}}\\
 & \mathbf{\Gamma}_{g,1}.
}
\]
We deduce from the universal property of the kernel that there exist
morphisms $\textrm{Forget}^{\mathfrak{b}}:\mathbf{B}_{[1],n}(\varSigma_{g,1})\rightarrow\mathbf{B}_{n}(\varSigma_{g,1})$
and $s_{n}^{\mathfrak{b}}:\mathbf{B}_{n}(\varSigma_{g,1})\rightarrow\mathbf{B}_{[1],n}(\varSigma_{g,1})$
such that the following diagrams are commutative:

\[
\xymatrix{\mathbf{B}_{[1],n}(\varSigma_{g,1})\ar@{^{(}->}[d]_{i_{(1,n)}^{\mathfrak{b}}}\ar@{->}[rr]^{\textrm{Forget}^{\mathfrak{b}}} &  & \mathbf{B}_{n}(\varSigma_{g,1})\ar@{^{(}->}[d]^{i_{n}^{\mathfrak{b}}}\\
\mathbf{\Gamma}_{g,1}^{[1],n}\ar@{->>}[rr]^{\textrm{Forget}} &  & \mathbf{\Gamma}_{g,1}^{n},
}
\textrm{ \,\,\,\, \,\,\,\, }\xymatrix{\mathbf{B}_{[1],n}(\varSigma_{g,1})\ar@{^{(}->}[d]_{i_{(1,n)}^{\mathfrak{b}}} &  & \mathbf{B}_{n}(\varSigma_{g,1})\ar@{^{(}->}[d]^{i_{n}^{\mathfrak{b}}}\ar@{->}[ll]_{s_{n}^{\mathfrak{b}}}\\
\mathbf{\Gamma}_{g,1}^{[1],n} &  & \mathbf{\Gamma}_{g,1}^{n}\ar@{_{(}->}[ll]_{s_{n}}.
}
\]
The commutativity of the second square implies that the morphism $s_{n}^{\mathfrak{b}}$
is a monomorphism. Also it is a formal fact that $s_{n}^{\mathfrak{b}}$
defines a section of $\textrm{Forget}^{\mathfrak{b}}$, and the latter
is therefore an epimorphism. Denoting by $K$ the kernel of $\textrm{Forget}^{\mathfrak{b}}$,
we thus obtain that $\mathbf{B}_{[1],n}(\varSigma_{g,1})\cong K\rtimes\mathbf{B}_{n}(\varSigma_{g,1})$.

We deduce from the universal property of $\pi_{1}(\varSigma_{g,1}^{n},p)$
as the kernel of $\textrm{Forget}$ that there exists a unique monomorphism
$K\hookrightarrow\pi_{1}(\varSigma_{g,1}^{n},p)$. Since $p_{(1,n)}^{\mathfrak{b}}\circ(\textrm{Push}_{p_{1}}\circ\alpha_{n}^{x_{1}})=(p_{n}^{\mathfrak{b}}\circ\textrm{Forget})\circ(\textrm{Push}_{p_{1}}\circ\alpha_{n}^{x_{1}})=0$,
the universal property of $\mathbf{B}_{[1],n}(\varSigma_{g,1})$ as
the kernel of $p_{(1,n)}^{\mathfrak{b}}$ provides a unique monomorphism
$\theta:\pi_{1}(\varSigma_{g,1}^{n},p)\hookrightarrow\mathbf{B}_{[1],n}(\varSigma_{g,1})$
such that $\textrm{Push}_{p_{1}}\circ\alpha_{n}^{x_{1}}=i_{(1,n)}^{\mathfrak{b}}\circ\theta$.
Then, the commutativity of the above first square and the universal
property of $K$ provide a unique monomorphism $\pi_{1}(\varSigma_{g,1}^{n},p)\hookrightarrow K$
which is a section of $K\hookrightarrow\pi_{1}(\varSigma_{g,1}^{n},p)$:
hence $K$ is isomorphic to $\pi_{1}(\varSigma_{g,1}^{n},p)$. The
action of $\mathbf{B}_{n}(\varSigma_{g,1})$ on $\pi_{1}(\varSigma_{g,1}^{n},p)$
is formally induced by the restriction $\mathcal{A}_{\varSigma_{g,1}^{n}}^{\mathfrak{b}}$
of $\mathcal{A}_{\varSigma_{g,1}^{n}}$ to the surface braid group,
which ends the proof.
\end{proof}
Furthermore, we consider the short exact sequences of type (\ref{eq:Birman_2})
associated with the morphisms $p_{(1,n)}^{\mathfrak{b}}:\mathbf{\Gamma}_{g,1}^{[1],n}\twoheadrightarrow\mathbf{\Gamma}_{g,1}$
and $p_{n+1}^{\mathfrak{b}}:\mathbf{\Gamma}_{g,1}^{1+n}\twoheadrightarrow\mathbf{\Gamma}_{g,1}$.
We recall that we have a canonical injection $\mathscr{E}_{n}:\mathbf{\Gamma}_{g,1}^{[1],n}\hookrightarrow\mathbf{\Gamma}_{g,1}^{1+n}$
induced by viewing $\mathbf{\Gamma}_{g,1}^{[1],n}$ as a subgroup
of $\mathbf{\Gamma}_{g,1}^{1+n}$ where the first puncture is fixed.
It is clear from the definitions that $p_{n+1}^{\mathfrak{b}}\circ\mathscr{E}_{n}=p_{(1,n)}^{\mathfrak{b}}$.
Therefore the universal property of $\mathbf{B}_{[1],n}(\varSigma_{g,1})$
as the kernel of $p_{(1,n)}^{\mathfrak{b}}$ provides a canonical
embedding $\mathscr{E}_{n}^{\mathfrak{b}}:\mathbf{B}_{[1],n}(\varSigma_{g,1})\hookrightarrow\mathbf{B}_{1+n}(\varSigma_{g,1})$.
Hence we now can define the appropriate morphisms $\{\chi_{n}:\pi_{1}(\varSigma_{g,1}^{n},p)\rightarrow\mathbf{B}_{1+n}(\varSigma_{g,1})\}_{n\in\mathbb{N}}$
to define a Long-Moody functor for surface braid groups:
\begin{defn}
\label{def:defsigmabraid1}Let $n$ be a natural number. We define
the morphism $\chi_{n,1}^{\mathfrak{b}}:H_{n}\rightarrow\mathbf{B}_{1+n}(\varSigma_{g,1})$
to be the composition $\mathscr{E}_{n}^{\mathfrak{b}}\circ\mathscr{B}_{n}^{\mathfrak{b}}\circ\textrm{Push}_{p_{1}}^{\mathfrak{b}}:H_{n}\hookrightarrow\mathbf{B}_{[1],n}(\varSigma_{g,1})\hookrightarrow\mathbf{B}_{1+n}(\varSigma_{g,1})$.

If $g=0$, we note that $\chi_{0,1}^{\mathfrak{b}}:\pi_{1}(\varSigma_{0,1}^{0},p)\rightarrow0_{\mathfrak{Gr}}$
is the trivial morphism and $\chi_{1,1}^{\mathfrak{b}}:\pi_{1}(\varSigma_{0,1}^{1},p)\rightarrow\mathbf{PB}_{2}$
is the morphism sending the generator $f_{1}$ of $\pi_{1}(\varSigma_{0,1}^{1},p)$
to $\sigma_{1}^{2}$ (where $\sigma_{1}$ denotes the Artin generator
of the braid group on two strands $\mathbf{B}_{2}$).
\end{defn}

We now have all the required ingredients to define an appropriate
Long-Moody functor:
\begin{prop}
\label{prop:long-moodysystemsurfacebraidgroups}The setting $\{\pi_{1}(-,p)^{\mathfrak{b}},\mathfrak{B}_{2}^{g},\mathfrak{B}_{2},\chi_{1}^{\mathfrak{b}}\}$
is a coherent Long-Moody system.
\end{prop}

\begin{proof}
First, we note from the definitions that the composition of the canonical
injection of $\mathbf{B}_{n}(\varSigma_{g,1})$ into $H_{n}\rtimes_{\mathcal{A}_{\varSigma_{g,1}^{n}}^{\mathfrak{b}}}\mathbf{B}_{n}(\varSigma_{g,1})$
with the composition $\mathscr{E}_{n}^{\mathfrak{b}}\circ\mathscr{B}_{n}^{\mathfrak{b}}:H_{n}\rtimes_{\mathcal{A}_{\varSigma_{g,1}^{n}}^{\mathfrak{b}}}\mathbf{B}_{n}(\varSigma_{g,1})\hookrightarrow\mathbf{B}_{1+n}(\varSigma_{g,1})$
is the canonical morphism $id_{\varSigma_{0,1}^{1}}\natural-:\mathbf{B}_{n}(\varSigma_{g,1})\hookrightarrow\mathbf{B}_{1+n}(\varSigma_{g,1})$
induced by the monoidal structure. Then, following mutatis mutandis
the proof of Lemma \ref{lem:sigman1satisfiesfirtcond}, the setting
$\{\pi_{1}(\varSigma_{g,1}^{-},p)^{\mathfrak{b}},\mathfrak{B}_{2}^{g},\mathfrak{B}_{2},\chi_{1}^{\mathfrak{b}}\}$
is a Long-Moody system.

By Proposition \ref{cond:conditionstability}, to prove that this
system is coherent, it is enough to prove that the morphisms $\chi_{n,1}^{\mathfrak{b}}$
satisfy the equality (\ref{eq:equiva}) for all natural numbers $n$.
It follows from Definitions \ref{def:defsigma1} and \ref{def:defsigmabraid1}
that the composition $i_{n}^{\mathfrak{b}}\circ\chi_{n,1}^{\mathfrak{b}}:H_{n}\hookrightarrow\mathbf{B}_{1+n}(\varSigma_{g,1})\hookrightarrow\mathbf{\Gamma}_{g,1}^{1+n}$
is actually equal to $\chi_{n,1}$: therefore the morphisms $\chi_{n,1}^{\mathfrak{b}}$
satisfy the equality (\ref{eq:equiva}) because the morphisms $\chi_{n,1}$
do; see the proof of Proposition \ref{lem:sigma1satisfiesCondition}.
\end{proof}

\paragraph{Connection with previous work.}

Assuming that $g=0$, we recover the results of an earlier paper of
the author \cite{soulieLMBilan}. Indeed, in this case we consider
the category $\mathfrak{U}\mathfrak{B}_{2}^{0}=\mathfrak{U}\boldsymbol{\beta}$,
which is Quillen's bracket construction on the braid groupoid $\boldsymbol{\beta}$.
The choice $\chi_{n,1}=\chi_{n,1}^{\mathfrak{b}}:\mathbf{F}_{n}\rightarrow\mathbf{B}_{n+1}$
of Definitions \ref{def:defsigma1} and \ref{def:defsigmabraid1}
corresponds to the morphism introduced in \cite[Example 2.7]{soulieLMBilan}.
The actions $\mathcal{A}_{\varSigma_{0,1}^{n}}:\mathbf{B}_{n}\rightarrow\textrm{Aut}_{\mathfrak{Gr}}(\pi_{1}(\varSigma_{0,1}^{n},p))$,
which correspond to Artin's representations for all natural numbers
$n$. We deduce that the Long-Moody functor associated with the coherent
system $\{\pi_{1}(-,p)^{\mathfrak{b}},\mathfrak{B}_{2}^{0},\mathfrak{B}_{2},\chi_{1}^{\mathfrak{b}}\}$
is isomorphic the Long-Moody functor of \cite[Section 2.3.1]{soulieLMBilan}
denoted by $\mathbf{LM}_{1}$. We could also have chosen other actions
$\mathcal{A}_{n}:\mathbf{B}_{n}\rightarrow\textrm{Aut}(\mathbf{F}_{n})$
and morphisms $\chi_{n}:\mathbf{F}_{n}\rightarrow\mathbf{B}_{n+1}$
so that the framework of Section \ref{sec:The-generalized-Long-Moody}
is satisfied. Hence, we recover all the Long-Moody functors introduced
in \cite{soulieLMBilan}.

In addition, the following example shows that the new framework developed
in the present paper recovers even more families of representations
of braid groups that the work of \cite{soulieLMBilan} could not obtain.
\begin{example}
\label{exa:recoverCallegaro} For each natural number $n$, there
is a classical geometric embedding $\mathscr{W}_{n}:\mathbf{B}_{2n+1}\hookrightarrow\mathbf{\Gamma}_{n,1}$
that sends the standard generators of the braid group to Dehn twists
around a fixed system of meridians and longitudes on the surface $\varSigma_{n,1}$;
we refer the reader to \cite[Section 1]{boedigtilman1} for more details
about this embedding. Let $\mathcal{W}$ be the subgroupoid of $\mathfrak{M}_{2}^{+,0}$
defined by the embeddings $\{\mathscr{W}_{n}\}_{n\in\mathbb{N}}$.
We assign $H$ to be the group $\pi_{1}(\varSigma_{n,1},p)$ and $H_{0}$
to be the trivial group. Hence, the restriction of the functor $\pi_{1}(\varSigma_{-,1},p)$
provides $\pi_{1}(\varSigma_{-,1},p)^{\mathfrak{b},2}:\mathfrak{U}\mathcal{W}\rightarrow\mathfrak{U}\mathfrak{M}_{2}^{+,0}\rightarrow\mathfrak{Gr}$.
Using the trivial functor $\chi_{tr}$, we have a coherent Long-Moody
system $\{\pi_{1}(-,p)^{\mathfrak{b},2},\mathcal{W},\mathfrak{M}_{2}^{+,0},\chi_{tr}\}$.
By Proposition \ref{prop:casesigmatrivial} the associated Long-Moody
functor applied to the constant functor $R$ is isomorphic to the
restriction of the functor $H_{1}(\varSigma_{-,1},R)$ to the category
$\mathfrak{U}\mathcal{W}$, denoted by $H_{1}(\varSigma_{-,1},\mathbb{Z})_{\mathfrak{U}\mathcal{W}}$.
The encoded representations are considered by Callegaro and Salvetti
in \cite{callegarosalvetti2017}.

In particular, considering a coherent Long-Moody system $\{\pi_{1}(-,p)^{\mathfrak{b}},\mathfrak{B}_{2},\mathfrak{B}_{2},\chi\}$
fitting in the previous framework of \cite{soulieLMBilan}, the group
$H_{n}$ is the free group on $n$ generators $\mathbf{F}_{n}$. Then
$\mathbf{LM}(M)(n)\cong M^{\oplus n}$ for all objects $M$ of $\mathbf{Fct}(\mathfrak{U}\boldsymbol{\beta},R\textrm{-}\mathfrak{Mod})$
and all natural numbers $n$. Since $H_{1}(\varSigma_{-,1},\mathbb{Z})_{\mathfrak{U}\mathcal{W}}(n)\cong\mathbb{Z}^{\oplus2n}$,
it is impossible to directly recover the functor $H_{1}(\varSigma_{-,1},\mathbb{Z})_{\mathfrak{U}\mathcal{W}}$
applying a Long-Moody functor with this setting.
\end{example}

\paragraph{For non-orientable surfaces.}

As in Section \ref{subsec:Modifyingpunctures}, the above work adapts
verbatim to the case of the braid group on the non-orientable surface
$\mathscr{N\varSigma}_{h,1}^{n}$, with a natural number $h\geq2$.
Let $\mathfrak{B}_{2}^{-,h}$ be the small full subgroupoid of $\mathfrak{B}_{2}$
on the objects $\underline{n}=1^{\natural n}\natural\underline{0}$
where $\underline{0}:=\mathscr{N\varSigma}_{h,1}$ and $1:=D$. The
morphisms $\{\chi_{n,2}\}_{n\in\mathbb{N}}$ induce in the same way
a family of morphisms $\{\chi_{n,2}^{\mathfrak{b}}:\pi_{1}(\mathscr{N\varSigma}_{h,1}^{n},p)\rightarrow\mathbf{B}_{1+n}(\mathscr{N\varSigma}_{h,1})\}_{n\in\mathbb{N}}$.
Then we can analogously prove that:
\begin{prop}
The setting $\{\pi_{1}(\mathscr{N\varSigma}_{h,1}^{n},p)^{\mathfrak{b}},\mathfrak{B}_{2}^{-,h},\mathfrak{B}_{2},\chi_{2}^{\mathfrak{b}}\}$
is a coherent Long-Moody system.
\end{prop}

\subsubsection{Applications\label{subsec:Applications}}

Finally, we briefly state some applications of the Long-Moody functors
defined by Propositions \ref{lem:sigma1satisfiesCondition} and \ref{prop:long-moodysystemsurfacebraidgroups}
to construct new families of representations for mapping class groups
and surface braid groups respectively. We do not detail the proofs
showing that the constructed representations are not certain known
ones: these follow from some standard (but quite lengthy) computations,
which are left to the reader. 

\paragraph{Mapping class groups.}

We fix $g\geq1$ and take the commutative ring $R$ to be $\mathbb{Q}$
for this paragraph. We consider the coherent Long-Moody system of
Proposition \ref{lem:sigma1satisfiesCondition} and denote its associated
functor by $\mathbf{LM}_{\mathfrak{M}_{2}^{+,g}}$ for simplicity.
The surface $\varSigma_{g,1}^{m}$ being a classifying space of $\pi_{1}(\varSigma_{g,1}^{m},p)$,
we denote by $H_{1}(\varSigma_{g,1}^{-},R)$ the composition functor
$H_{1}(-,R)\circ\pi_{1}(\varSigma_{g,1}^{-},p)$. It follows from
Proposition \ref{prop:casesigmatrivial} that $H_{1}(\varSigma_{g,1}^{-},\mathbb{Q})$
is isomorphic to $\mathbf{LM}_{\mathfrak{M}_{2}^{+,g}}(\mathbb{Q})$.
However, some straightforward (but lengthy) matrix computations show
that the functor $\mathbf{LM}_{\mathfrak{M}_{2}^{+,g}}(H_{1}(\varSigma_{g,1}^{-},\mathbb{Q}))$
is not isomorphic to $\mathbf{LM}_{\mathfrak{M}_{2}^{+,g}}(\mathbb{Q})\otimes_{\mathbb{Q}}H_{1}(\varSigma_{g,1}^{-},\mathbb{Q})(1\natural-)$.
The computation of the Long-Moody functor associated with $\{\pi_{1}(\varSigma_{g,1}^{-},p),\mathfrak{M}_{2}^{+,g},\mathfrak{M}_{2},\chi_{1}\}$
on an object $F$ of $\mathbf{Fct}(\mathfrak{U}\mathfrak{M}_{2}^{+,g},R\textrm{-}\mathfrak{Mod})$
is therefore not generally speaking given by Proposition \ref{prop:casesigmatrivial}.
By applying the associated Long-Moody functor, we thus provide new
families of linear representations for the mapping class groups $\{\mathbf{\Gamma}_{g,1}^{n}\}_{n\in\mathbb{N}}$,
which have very few examples of linear representations in the literature.

\paragraph{Surface braid groups.}

We fix $g\geq0$ and consider the coherent Long-Moody system of Proposition
\ref{lem:sigma1satisfiesCondition} with $R=\mathbb{Z}$ as commutative
ring. We denote its associated functor by $\mathbf{LM}_{\mathfrak{B}_{2}^{g}}$
for simplicity.

Instead of applying the Long-Moody functors on the first homology
of the surface as above for the mapping class groups, there is a more
natural and interesting alternative for surface braid groups. For
a group $G$, we denote by $\{\varGamma_{l}(G),l\geq0\}$ its \textit{lower
central series}. When there is no ambiguity, we omit $G$ from the
notations. For $g=0$, we recall that the abelianization $\mathbf{B}_{n}/\Gamma_{2}\mathbf{B}_{n}$
is isomorphic to $\mathbb{Z}$ for $n\geq2$. Let $\mathbb{Z}[\mathbb{Z}]:\boldsymbol{\beta}\rightarrow\textrm{\ensuremath{\mathbb{Z}}-}\mathfrak{Mod}$
be the functor defined by sending a natural number $n$ to the group
ring $\mathbb{Z}[\mathbb{Z}]$ if $n\geq2$, on which the action of
$\mathbf{B}_{n}$ is induced by the left multiplication on the quotient
$\mathbf{B}_{n}/\Gamma_{2}\mathbf{B}_{n}$, or to the trivial group
if $n\leq1$. Then, the iterates of $\mathbf{LM}_{\boldsymbol{\beta}}$
on $\mathbb{Z}[\mathbb{Z}]$ provide new families of representations;
see \cite[Section 2.3]{soulieLMBilan}.

Now we fix $g\geq1$. A direct computation from \cite[Theorem 1.1]{bellingeri2004presentations}
shows that $\mathbf{B}_{n}(\varSigma_{g,1})/\varGamma_{2}(\mathbf{B}_{n}(\varSigma_{g,1}))$
is isomorphic to $\mathbb{Z}/2\mathbb{Z}\times\mathbb{Z}^{2g}$, where
the braid generators are sent to $\mathbb{Z}/2\mathbb{Z}$. A fortiori
the generators $\sigma_{i}^{2}$ act trivially on that quotient: we
thus go one step further in the lower central series to get a module
which has a non-trivial action via $\chi_{n-1,1}^{\mathfrak{b}}:H_{n-1}\rightarrow\mathbf{B}_{n}(\varSigma_{g,1})$.
The third lower central quotient $\mathbf{B}_{n}(\varSigma_{g,1})/\varGamma_{3}(\mathbf{B}_{n}(\varSigma_{g,1}))$
is isomorphic to the semidirect product $(\mathbb{Z}\times\mathbb{Z}^{g})\rtimes\mathbb{Z}^{g}$
for $n\geq3$ by \cite[Corollary 3.14]{bellingeriguaschigodelle}.
Let $\mathbb{Z}[\mathfrak{B}_{2}^{g}/\varGamma_{3}]:\mathfrak{B}_{2}^{g}\rightarrow\mathbb{Z}\textrm{-}\mathfrak{Mod}$
be the functor defined by assigning the group ring $\mathbb{Z}[(\mathbb{Z}\times\mathbb{Z}^{g})\rtimes\mathbb{Z}^{g}]$
to each natural number $n$, on which the action of $\mathbf{B}_{n}(\varSigma_{g,1})$
is induced by the left multiplication on the quotient $\mathbf{B}_{n}(\varSigma_{g,1})/\Gamma_{3}(\mathbf{B}_{n}(\varSigma_{g,1}))$
if $n\geq3$ and the trivial action if $n\leq2$. Hence we deduce
from some direct elementary (but lengthy) computations on the obtained
matrices that the functor $\mathbf{LM}_{\mathfrak{B}_{2}^{g}}^{\circ d+1}(\mathbb{Z}[\mathfrak{B}_{2}^{g}/\varGamma_{3}])$
is not isomorphic to the tensor product functor $\mathbf{LM}_{\mathfrak{B}_{2}^{g}}(\mathbb{Z})\otimes_{\mathbb{Z}}(\mathbf{LM}_{\mathfrak{B}_{2}^{g}}^{\circ d}(\mathbb{Z}[\mathfrak{B}_{2}^{g}/\varGamma_{3}])(1\natural-)$
for $d\in\{0,1\}$. Hence, the iterates of $\mathbf{LM}_{\mathfrak{B}_{2}^{g}}$
are not generally speaking determined by Proposition \ref{prop:casesigmatrivial}
and thus define new representations for surface braid groups. As far
as the author knows, there are very few explicit examples of linear
representations of surface braid groups for $g\geq1$.

\subsection{Modifying the genus\label{subsec:Modifyingthegenus}\label{par:Orientable-surfaces}}

In this section, we construct Long-Moody functors for the mapping
class groups of surfaces when the orientable genus varies. For all
natural numbers $n$, we denote by $\varSigma_{n,1}$ the surface
$T^{\natural n}\natural\mathbb{D}^{2}$. For consistency if $n=0$,
we assign $\varSigma_{0,1}=\mathbb{D}^{2}$. Let $\mathfrak{M}_{2}^{+}$
be the small full subgroupoid of $\mathfrak{M}_{2}$ on the objects
$\{\varSigma_{n,1}\}_{n\in\mathbb{N}}$. The monoidal structure $(\mathfrak{M}_{2},\natural,0)$
clearly restricts to a braided monoidal structure on the subgroupoid
$\mathfrak{M}_{2}^{+}$, denoted in the same way $(\mathfrak{M}_{2}^{+},\natural,0)$.
In particular, in the notations of Section \ref{subsec:Definition-of-theLM},
the groupoid $\mathfrak{M}_{2}^{+}$ plays the role of both $\mathcal{G}'$
and $\mathcal{G}$ where $\underline{0}:=\varSigma_{0,1}$ and $1:=\varSigma_{1,1}$.
We denote the mapping class group $\pi_{0}\textrm{Homeo\ensuremath{^{I}}}(\varSigma_{n,1})$
by $\mathbf{\Gamma}_{n,1}$, for all $n\in\mathbb{N}$.

The key point to define a non-trivial Long-Moody system in the present
situation is the use of \textit{pure surface framed braids}: for completeness,
the definition is recalled below and we refer the reader to \cite{BellingeriGervais}
for further details. We fix a natural number $n\geq0$ and choose
a marked point $x$ in the interior of $\varSigma_{n,1}$. There is
no loss of generality in endowing the surface $\varSigma_{n,1}$ with
an arbitrary smooth structure and a Riemannian metric. We denote by
$U\varSigma_{n,1}$ the total space of the unit tangent bundle and
fix a unit tangent vector $v$ of $\varSigma_{n,1}$ at a point $x$
in the interior of $\varSigma_{n,1}$. Using the long exact sequence
in homotopy for the fibration $\mathbb{S}^{1}\rightarrow U\varSigma_{n,1}\rightarrow\varSigma_{n,1}$
(where $\mathbb{S}^{1}$ is the $1$-dimensional unit sphere) and
the contractibility results of \cite[Théorème 5]{gramain1973type},
we deduce that the fundamental group $\pi_{1}(U\varSigma_{n,1},(x,v))$
is isomorphic to $\mathbb{Z}\times\pi_{1}(\varSigma_{n,1},x)$ as
a group if $n\geq1$. As $\varSigma_{n,1}$ is a path-connected space,
the group $\pi_{1}(U\varSigma_{n,1},(x,v))$ is independent of the
base $(x,v)$ up to an isomorphism (in particular, if $x$ is fixed
then the isomorphism is unique). The \textit{pure framed braid group}
$FP_{1}(\varSigma_{n,1})$ is defined as the fundamental group of
$U\varSigma_{n,1}$ (see \cite[Definition 1]{BellingeriGervais})
and is isomorphic to $\mathbb{Z}\times\pi_{1}(\varSigma_{n,1},x)$
(see \cite[Theorem 5, 1)]{BellingeriGervais}) if $n\geq1$. The natural
action of the diffeomorphisms of $\varSigma_{n,1}$ on the unit tangent
bundle induce a canonical action of $\mathbf{\Gamma}_{n,1}$ on $FP_{1}(\varSigma_{n,1})$,
denoted by $\mathcal{A}_{U\varSigma_{n,1}}$. Then we define a functor
\[
FP_{1}:(\mathfrak{M}_{2}^{+},\natural,\mathbb{D}^{2})\rightarrow\mathfrak{Gr}
\]
assigning the pure framed braid group $FP_{1}(\varSigma_{n,1})$ to
each object $\varSigma_{n,1}$ and the corresponding natural action
of the mapping class group on morphisms. The functor $FP_{1}$ plays
the role of the parameter $\mathcal{A}:\mathcal{G}\rightarrow\mathfrak{Gr}$
for the Long-Moody functor of this section.

The remaining ingredient to be introduced to define an appropriate
Long-Moody system is the set of suitable morphisms $\{\chi_{n}:FP_{1}(\varSigma_{n,1})\rightarrow\mathbf{\Gamma}_{1+n,1}\}_{n\in\mathbb{N}}$.
Let $\varSigma_{n,2}$ be a smooth connected compact surface with
two boundary components: one of the two is marked by the parametrized
interval $I$ and the other one is denoted by $\partial'\varSigma_{n,2}$.
Denoting by $\textrm{Diff}^{+,\partial}(\varSigma_{n,2})$ the group
of orientation preserving diffeomorphisms of $\varSigma_{n,2}$ which
fix the boundary pointwise, we denote by $\mathbf{\Gamma}_{n,2}$
the mapping class group $\pi_{0}\textrm{Diff}^{+,\partial}(\varSigma_{n,2})$.
We recall from \cite[Corollary 3]{BellingeriGervais} that we have
a short exact sequence for each $n\geq1$:
\begin{equation}
\xymatrix{1\ar@{->}[r] & FP_{1}(\varSigma_{n,1})\ar@{->}[r]^{\,\,\,\,\,\,\,\,\overrightarrow{\textrm{Push}}_{n}} & \mathbf{\Gamma}_{n,2}\ar@{->}[r]^{\textrm{Cap}} & \mathbf{\Gamma}_{n,1}\ar@{->}[r] & 1.}
\label{eq:sesframe}
\end{equation}
Namely, the short exact sequence (\ref{eq:sesframe}) is the long
exact sequence in homotopy of a locally trivial fibration $\textrm{Diff}^{\textrm{\ensuremath{\partial_{0}}}}(\varSigma_{n,1})\rightarrow U\varSigma_{n,1}$
which fiber identifies with the group $\textrm{Diff}^{+,\partial}(\varSigma_{n,2})$;
see \cite[Proposition 2]{BellingeriGervais} for further details.
In particular, the morphism $\textrm{Cap}$ is the map induced in
homotopy by the associated fiber inclusion $\varpi_{n}$, corresponding
to the capping of the boundary component $\partial'\varSigma_{n,2}$
with a closed disc. Moreover, we prove that:
\begin{lem}
The sequence (\ref{eq:sesframe}) splits, inducing an isomorphism
$\mathscr{C}_{n}:\mathbf{\Gamma}_{n,2}\cong FP_{1}(\varSigma_{n,1})\rtimes_{\mathcal{A}_{U\varSigma_{n,1}}}\mathbf{\Gamma}_{n,1}$.
\end{lem}

\begin{proof}
Identifying $\varSigma_{n,2}$ with the boundary connected sum $\varSigma_{0,2}\natural\varSigma_{n,1}$
along the marked intervals, the surface $\varSigma_{n,1}$ can be
viewed as a subsurface of $\varSigma_{n,2}$ as the complement of
the cylinder $\varSigma_{0,2}$ with the extra boundary component
$\partial'\varSigma_{n,2}$. Let $\nu_{n}:\textrm{Diff}^{\textrm{\ensuremath{\partial_{0}}}}(\varSigma_{n,1})\to\textrm{Diff}^{+,\partial}(\varSigma_{n,2})$
be the inclusion map defined by sending each $\varphi\in\textrm{Diff}^{+,\partial}(\varSigma_{n,1})$
to $id_{\varSigma_{0,2}}\natural\varphi$, thus inducing a morphism
$\pi_{0}(\nu_{n}):\mathbf{\Gamma}_{n,1}\rightarrow\mathbf{\Gamma}_{n,2}$
in homotopy. The composition $\varpi_{n}\circ\nu_{n}(\varphi)$ is
clearly isotopic to $\varphi$ and a fortiori $\pi_{0}(\nu_{n})$
is a section of $\textrm{Cap}$. Hence the short exact sequence (\ref{eq:sesframe})
splits, thus providing the required isomorphism.
\end{proof}
On another note, we define a morphism $id_{\mathring{\mathbb{T}}}\sharp-:\mathbf{\Gamma}_{n,2}\rightarrow\mathbf{\Gamma}_{1+n,1}$
by gluing a torus with a boundary component $\mathring{\mathbb{T}}$
along the boundary component $\partial'\varSigma_{n,2}$, and by extending
the diffeomorphisms of $\varSigma_{n,2}$ by the identity on $\mathring{\mathbb{T}}$.
We now are able to define the appropriate morphisms $\{\chi_{n}:FP_{1}(\varSigma_{n,1})\rightarrow\mathbf{\Gamma}_{1+n,1}\}_{n\in\mathbb{N}}$
in the present context:
\begin{defn}
We assign $\chi_{0,fr}:FP_{1}(\varSigma_{0,1})\rightarrow\mathbf{\Gamma}_{1,1}$
to be the trivial morphism. For $n\geq1$ a natural number, we define
the morphism $\chi_{n,fr}:FP_{1}(\varSigma_{n,1})\rightarrow\mathbf{\Gamma}_{1+n,1}$
to be the composite:
\[
\xymatrix{FP_{1}(\varSigma_{n,1})\ar@{^{(}->}[r] & FP_{1}(\varSigma_{n,1})\rtimes_{\mathcal{A}_{U\varSigma_{n,1}}}\mathbf{\Gamma}_{n,1}\overset{\mathscr{C}_{n}}{\cong}\mathbf{\Gamma}_{n,2}\ar@{->}[r]^{\,\,\,\,\,\,\,\,\,\,\,\,\,\,\,\,\,\,\,\,\,\,\,\,\,\,\,\,\,\,\,\,\,\,\,\,id_{\mathring{\mathbb{T}}}\sharp-} & \mathbf{\Gamma}_{1+n,1}.}
\]
\end{defn}

We now have all the required ingredients to define a Long-Moody functor:
\begin{prop}
\label{prop:LM_mcg_genus}The setting $\{FP_{1},\mathfrak{M}_{2}^{+},\mathfrak{M}_{2},\chi_{fr}\}$
is a Long-Moody system.
\end{prop}

\begin{proof}
If $n\geq1$, note that the injection $\pi_{0}(\nu_{n}):\mathbf{\Gamma}_{n,1}\hookrightarrow FP_{1}(\varSigma_{n,1})\rtimes\mathbf{\Gamma}_{n,1}\cong\mathbf{\Gamma}_{n,2}$
is equal to $id_{\varSigma_{0,2}}\natural-$. Let $id_{\varSigma_{1,1}}\natural-:\mathbf{\Gamma}_{n,1}\rightarrow\mathbf{\Gamma}_{1+n,1}$
be the morphism induced by extending the diffeomorphisms of $\varSigma_{n,1}$
by the identity on $\varSigma_{1+n,1}$. It is straightforward from
the above definitions that the composition $(id_{\mathring{\mathbb{T}}}\sharp-)\circ(id_{\varSigma_{0,2}}\natural-)$
is equal to $id_{\varSigma_{1,1}}\natural-$. Hence, the following
diagram is commutative:
\[
\xymatrix{FP_{1}(\varSigma_{n,1})\ar@{^{(}->}[r]\ar@{->}[dr]_{\chi_{n,fr}} & FP_{1}(\varSigma_{n,1})\rtimes_{\mathcal{A}_{U\varSigma_{n,1}}}\mathbf{\Gamma}_{n,1}\cong\mathbf{\Gamma}_{n,2}\ar@{->}[d]^{(id_{\mathring{\mathbb{T}}}\sharp-)} & \mathbf{\Gamma}_{n,1}\ar@{_{(}->}[l]\ar@{->}[dl]^{id_{\varSigma_{1,1}}\natural-}\\
 & \mathbf{\Gamma}_{1+n,1}.
}
\]
For $n=0$, the commutativity of the analogous diagram is trivially
checked. Hence it follows Lemma \ref{lem:equivcond:coherenceconditionsigmanan}
that the diagram (\ref{cond:coherenceconditionsigmanan}) of Definition
\ref{def:longmoodysystem} is commutative, which ends the proof.
\end{proof}
For simplicity, we denote by $\mathbf{LM}_{\mathfrak{M}_{2}^{+}}$
the functor associated to the Long-Moody system of Proposition \ref{prop:LM_mcg_genus}
and take the commutative ring $R$ to be $\mathbb{Z}$. We recall
that the framed pure braid group $FP_{1}(\varSigma_{n,1})$ corresponds
to the fundamental group of the unit tangent bundle $U\varSigma_{n,1}$.
Recalling that the framed pure braid groups define a functor $FP_{1}:(\mathfrak{M}_{2}^{+},\natural,\mathbb{D}^{2})\rightarrow\mathfrak{Gr}$,
we denote by $H_{1}(U\varSigma_{-,1};\mathbb{Z})$ the composition
functor $H_{1}(-,\mathbb{Z})\circ FP_{1}$. It follows from Proposition
\ref{prop:casesigmatrivial} that $\mathbf{LM}_{\mathfrak{M}_{2}^{+}}(\mathbb{Z})(\underline{n})$
is isomorphic to $H_{1}(U\Sigma_{n,1};\mathbb{Z})$ as an $\mathbb{Z}$-module
for each $n\in\mathbb{N}$. The $\Gamma_{n,1}$-module structure of
$H_{1}(U\Sigma_{n,1};\mathbb{Z})$ has been described in details by
Trapp in \cite{Trapp}. The following result shows that the evaluation
of the functor $\mathbf{LM}_{\mathfrak{M}_{2}^{+}}$ is not determined
by Proposition \ref{prop:casesigmatrivial} generally speaking. Identifying
$\varSigma_{1,1}\natural\varSigma_{n,1}$ with $\varSigma_{n+1,1}$,
we fix a point $x$ in the interior of $\varSigma_{1,1}$ in $\varSigma_{n+1,1}$.
\begin{prop}
\label{prop:iteration_Usigma}Let $M$ be an object of $\mathbf{Fct}(\mathfrak{M}_{2}^{+},\mathbb{Z}\textrm{-}\mathfrak{Mod})$
such that $M(1\natural\underline{n})$ is a free abelian group for
some $n\geq1$. Using the group isomorphism $FP_{1}(\varSigma_{n,1})\cong\mathbb{Z}\times\pi_{1}(\varSigma_{n,1},x)$,
we consider $\pi_{1}(\varSigma_{n,1},x)$ as a subgroup of $FP_{1}(\varSigma_{n,1})$.
We assume that the $\pi_{1}(\varSigma_{n,1},x)$-action induced by
$\overrightarrow{\textrm{Push}}_{n}:FP_{1}(\varSigma_{n,1})\hookrightarrow\mathbf{\Gamma}_{n+1,1}$
on $M(1\natural\underline{n})$ is not trivial. Then the functor $\mathbf{LM}_{\mathfrak{M}_{2}^{+,g}}(M)$
is not isomorphic to $\mathbf{LM}_{\mathfrak{M}_{2}^{+}}(\mathbb{Z})\otimes_{\mathbb{Z}}M(1\natural-)$.
\end{prop}

\begin{proof}
Since $H_{1}(U\Sigma_{n,1};\mathbb{Z})$ is a free abelian group of
rank $2n+1$, the module $\mathbf{LM}_{\mathfrak{M}_{2}^{+}}(\mathbb{Z})(\underline{n})\otimes_{\mathbb{Z}}M(1\natural\underline{n})$
is isomorphic to $M(1\natural\underline{n})^{\oplus(2n+1)}$. Also,
we have a surjection $\mathbb{Z}[FP_{1}(\varSigma_{n,1})]\twoheadrightarrow\mathbb{Z}[\pi_{1}(\varSigma_{n,1},x)]$
(induced from $\pi_{1}(U\varSigma_{n,1},(x,v))\twoheadrightarrow\pi_{1}(\varSigma_{n,1},x)$)
which defines a right $\mathbb{Z}[FP_{1}(\varSigma_{n,1})]$-module
surjection 
\[
p_{FP_{1}(\varSigma_{n,1})}:\mathcal{I}_{\mathbb{Z}[FP_{1}(\varSigma_{n,1})]}\twoheadrightarrow\mathcal{I}_{\mathbb{Z}[\pi_{1}(\varSigma_{n,1},x)]}\otimes_{\mathbb{Z}[\pi_{1}(\varSigma_{n,1},x)]}\mathbb{Z}[FP_{1}(\varSigma_{n,1})].
\]
We recall from Lemma \ref{lem:Swan} that $\mathcal{I}_{\mathbb{Z}[\pi_{1}(\varSigma_{n,1},x)]}$
is a free $\mathbb{Z}[\pi_{1}(\varSigma_{n,1},x)]$-module of rank
$2n$. Therefore, the tensor product morphism $p_{FP_{1}(\varSigma_{n,1})}\otimes_{\mathbb{Z}[FP_{1}(\varSigma_{n,1})]}id_{M(1\natural\underline{n})}$
provides a surjection of abelian groups from $\mathbf{LM}_{\mathfrak{M}_{2}^{+,g}}(M)(\underline{n})$
to $M(1\natural\underline{n})^{\oplus2n}$. We denote by $K$ the
kernel of this surjection. Also, since the target of this surjection
is a projective abelian group, this surjection has a section which
associated quotient is isomorphic to $K$.

Now suppose that $\mathbf{LM}_{\mathfrak{M}_{2}^{+,g}}(M)(\underline{n})$
is isomorphic to $H_{1}(U\varSigma_{n,1},\mathbb{Z})\otimes_{\mathbb{Z}}M(1\natural\underline{n})$
as an abelian group. Then the kernel $K$ is isomorphic to $M(1\natural\underline{n})$.
Using the decomposition $\mathbb{Z}\times\pi_{1}(\varSigma_{n,1},x)$
of $FP_{1}(\varSigma_{n,1})$, an element of $\mathbf{LM}_{\mathfrak{M}_{2}^{+,g}}(H_{1}(U\varSigma_{-,1},\mathbb{Z}))(\underline{n})$
is of the form $(\sigma\alpha-1)\otimes v$ where $\sigma\in\mathbb{Z}$,
$\alpha\in\pi_{1}(\varSigma_{n,1},x)$ and $v\in M(1\natural\underline{n})$.
Since $\sigma\alpha=\alpha\sigma$ in $FP_{1}(\varSigma_{n,1})$,
we compute that $\sigma\alpha-1=(\alpha-1)+(\sigma-1)\alpha=(\alpha-1)\sigma+(\sigma-1).$
Hence we have $(\sigma-1)\alpha\otimes v=(\sigma-1)\otimes v$ in
$K$ (seen as a quotient of $\mathbf{LM}_{\mathfrak{M}_{2}^{+,g}}(M)(\underline{n})$).
However, since the abelian group $M(1\natural\underline{n})$ is a
non-trivial $\pi_{1}(\varSigma_{n,1},x)$-module, we deduce that there
cannot exist an injective abelian group morphism from $M(1\natural\underline{n})$
to $K$: this contradicts the fact that $K$ is isomorphic to $M(1\natural\underline{n})$
and ends the proof.
\end{proof}
\begin{example}
The kernel of the $\Gamma_{n,1}$-module $H_{1}(U\Sigma_{n,1};\mathbb{Z})$
is called the \textit{Chillingworth subgroup} and denoted by $\mathfrak{C}_{n,1}$;
see \cite[Definition 2.6]{Trapp}. We recall that equipping the surface
$\Sigma_{n,1}$ with a continuous non-vanishing vector field $X$,
the winding number $\omega_{X}(c)$ with respect to $X$ of an oriented
regular curve $c$ is the number of times its tangent rotates with
respect to the framing induced by $X$; see \cite{Chillingworth,Chillingworth2}.
Then $\mathfrak{C}_{n,1}$ is the group of all elements $f$ in the
\textit{Torelli subgroup} such that $\omega_{X}(f(c))=\omega_{X}(c)$
for all oriented regular curve $c$.

We identify $\varSigma_{1+n,1}$ with $\varSigma_{1,1}\natural\varSigma_{n,1}$
and we consider the standard basis $\{\alpha_{i},\beta_{i}\mid i\in\{1,\ldots,n\}\}$
of generators of $\pi_{1}(\varSigma_{n,1},p)$ associated with the
system of meridians and longitudes $\{a_{1},b_{1},\ldots,a_{n},b_{n}\}$
fixed in Section \ref{subsec:Modifyingpunctures}; see Figure \ref{fig:Generators-and-paths}.
The action of $\overrightarrow{\textrm{Push}}_{n}(\beta_{1})$ on
$a_{1}$ is analogous to the one pictured in Figure \ref{fig:push_map_first}:
the difference is that the puncture $p_{1}$ is replaced by a glued
handle $\varSigma_{1,1}$ and that we act on the meridian curve. In
particular, the winding number of $\overrightarrow{\textrm{Push}}_{n}(\beta_{1})(a_{1})$
is clearly different from the one of $a_{1}$, and so $\overrightarrow{\textrm{Push}}_{n}(\beta_{1})$
does not belong to the Chillingworth subgroup $\mathfrak{C}_{n+1,1}$.
Hence $H_{1}(U\Sigma_{n+1,1};\mathbb{Z})$ is not a trivial $\pi_{1}(\varSigma_{n,1},x)$-module
and thus $H_{1}(U\varSigma_{-,1};\mathbb{Z})$ satisfies the assumptions
of Proposition \ref{prop:iteration_Usigma}. The functor $\mathbf{LM}_{\mathfrak{M}_{2}^{+,g}}(H_{1}(U\varSigma_{-,1};\mathbb{Z}))$
thus provides a family of representations of the mapping class groups
which (as far as the author knows) does not appear in the literature. 
\end{example}

\section{Strong and weak polynomial functors\label{sec:Strong-and-weak}}

This section introduces the notions of (very) strong and weak polynomial
functors with respect to the framework of this paper. Namely, the
first subsection presents strong and very strong polynomial functors
and their basic properties. In the second subsection, we introduce
weak polynomial functors for some subcategories of pre-braided monoidal
categories with an initial object, generalizing the previous notion
of \cite[Section 1]{DV3}. We also detail some first properties of
these functors and present their use to organize families of representations.

\subsection{Strong and very strong polynomial functors\label{subsec:Prerequisite-on-strong}}

\textbf{For the remainder of Section \ref{subsec:Prerequisite-on-strong},
$(\mathfrak{M}',\natural,0)$ is a pre-braided strict monoidal category
where the unit $0$ is an initial object. We consider $\mathfrak{M}$
a small full subcategory of $(\mathfrak{M}',\natural,0)$. Finally,
we fix $\mathscr{A}$ an abelian category.}

In this section, we introduce the notions of strong and very strong
polynomiality for objects in the functor category $\mathbf{Fct}(\mathfrak{M},\mathscr{A})$.
In \cite[Section 3]{soulieLMBilan}, a framework is given for defining
these notions in the category $\mathbf{Fct}(M,\mathscr{A})$, where
$M$ is a small pre-braided monoidal category where the unit is an
initial object. It generalizes the previous work of Djament and Vespa
in \cite[Section 1]{DV3}. We also refer to \cite{palmer2017comparison}
for a comparison of the various instances of the notions of twisted
coefficient system and polynomial functor. This section thus extends
the definitions and properties of \cite[Section 3]{soulieLMBilan}
to the present larger framework, the various proofs being direct generalizations
of this previous work.
\begin{notation}
\label{nota:goodsetobj}We denote by $\textrm{Obj}(\mathfrak{M})_{\natural}$
the set of objects $m$ of $\mathfrak{M}'$ such that $m\natural n$
is an object of $\mathfrak{M}$ for all objects $n$ of $\mathfrak{M}$.
\end{notation}

Let $m$ be an element of $\textrm{Obj}(\mathfrak{M})_{\natural}$.
We denote by $\tau_{m}:\mathbf{Fct}(\mathfrak{M},\mathscr{A})\rightarrow\mathbf{Fct}(\mathfrak{M},\mathscr{A})$
the \textit{translation} functor defined by $\tau_{m}(F)=F(m\natural-)$,
$i_{m}:Id\rightarrow\tau_{m}$ the natural transformation of $\mathbf{Fct}(\mathfrak{M},\mathscr{A})$
induced by the unique morphism $\iota_{m}:0\rightarrow m$. We define
$\delta_{m}=\textrm{coker}(i_{m})$ the \textit{difference} functor
and $\kappa_{m}=\textrm{ker}(i_{m})$ the \textit{evanescence} functor.
The following basic properties are direct generalizations of \cite[Propositions 3.2 and 3.5]{soulieLMBilan}:
\begin{prop}
\label{prop:lemmecaract} Let $m,m'\in\textrm{Obj}(\mathfrak{M})_{\natural}$.
Then the translation functor $\tau_{m}$ is exact and we have the
following exact sequence of endofunctors of $\mathbf{Fct}(\mathfrak{M},\mathscr{A})$:
\begin{equation}
0\longrightarrow\kappa_{m}\overset{\Omega_{m}}{\longrightarrow}Id\overset{i_{m}}{\longrightarrow}\tau_{m}\overset{\varDelta_{m}}{\longrightarrow}\delta_{m}\longrightarrow0.\label{eq:ESCaract-1}
\end{equation}
Moreover, for a short exact sequence $0\longrightarrow F\longrightarrow G\longrightarrow H\longrightarrow0$
in the category $\mathbf{Fct}(\mathfrak{M},\mathscr{A})$, there is
a natural exact sequence in the category $\mathbf{Fct}(\mathfrak{M},\mathscr{A})$:
\begin{equation}
0\longrightarrow\kappa_{m}(F)\longrightarrow\kappa_{m}(G)\longrightarrow\kappa_{m}(H)\longrightarrow\delta_{m}(F)\longrightarrow\delta_{m}(G)\longrightarrow\delta_{m}(H)\longrightarrow0.\label{eq:LESkappadelta}
\end{equation}
In addition, the functors $\tau_{m}$ and $\tau_{m'}$ commute up
to natural isomorphism and they commute with limits and colimits;
the difference functors $\delta_{m}$ and $\delta_{m'}$ commute up
to natural isomorphism and they commute with colimits; the functors
$\kappa_{m}$ and $\kappa_{m'}$ commute up to natural isomorphism
and they commute with limits; the functor $\tau_{m}$ commute with
the functors $\delta_{m'}$ and $\kappa_{m'}$ up to natural isomorphism.
\end{prop}

We can define the notions of strong and very strong polynomial functors
using Proposition \ref{prop:lemmecaract}. Namely:
\begin{defn}
\label{def:strongandverystrong}We recursively define on $d\in\mathbb{N}$
the categories $\mathcal{P}ol_{d}^{strong}(\mathfrak{M},\mathscr{A})$
and $\mathcal{VP}ol_{d}(\mathfrak{M},\mathscr{A})$ of \textit{strong
}and\textit{ very strong polynomial} functors of degree less than
or equal to $d$ to be the full subcategories of $\mathbf{Fct}(\mathfrak{M},\mathscr{A})$
as follows:

\begin{enumerate}
\item If $d<0$, $\mathcal{P}ol_{d}^{strong}(\mathfrak{M},\mathscr{A})=\mathcal{VP}ol_{d}(\mathfrak{M},\mathscr{A})=\{0\}$;
\item if $d\geq0$, the objects of $\mathcal{P}ol_{d}^{strong}(\mathfrak{M},\mathscr{A})$
are the functors $F$ such that for all $m\in\textrm{Obj}(\mathfrak{M})_{\natural}$,
the functor $\delta_{m}(F)$ is an object of $\mathcal{P}ol_{d-1}^{strong}(\mathfrak{M},\mathscr{A})$;
the objects of $\mathcal{VP}ol_{d}(\mathfrak{M},\mathscr{A})$ are
the objects $F$ of $\mathcal{P}ol_{d}(\mathfrak{M},\mathscr{A})$
such that $\kappa_{m}(F)=0$ and the functor $\delta_{m}(F)$ is an
object of $\mathcal{\mathcal{VP}}ol_{d-1}(\mathfrak{M},\mathscr{A})$
for all $m\in\textrm{Obj}(\mathfrak{M})_{\natural}$.
\end{enumerate}
For an object $F$ of $\mathbf{Fct}(\mathfrak{M},\mathscr{A})$ which
is strong polynomial of degree less than or equal to $n\in\mathbb{N}$,
the \textit{smallest} natural number $d\leq n$ for which $F$ is
an object of $\mathcal{P}ol_{d}^{strong}(\mathfrak{M},\mathscr{A})$
is called the \textit{strong degree} of $F$. If in addition $F$
is very strong polynomial, its strong degree is also the smallest
natural number $d\leq n$ for which $F$ is an object of $\mathcal{VP}ol_{d}(\mathfrak{M},\mathscr{A})$
and is then also called the \textit{very strong degree} of $F$.
\end{defn}

Finally, we recall useful properties of the categories associated
with strong and very strong polynomial functors. Beforehand, for a
strict monoidal category $(\mathfrak{C},\natural,0)$, we say that
a full subcategory $\mathfrak{D}$ of $\mathfrak{C}$ is \textit{finitely
generated by the monoidal structure} if there exists a finite set
$E$ of objects of the category $\mathfrak{C}$ such that all objects
$d$ of $\mathfrak{D}$ is isomorphic to a finite monoidal product
of objects of $E$. The following properties are direct generalizations
of \cite[Propositions 3.9 and 3.19]{soulieLMBilan}, the proofs carry
over mutatis mutandis to the present framework.
\begin{prop}
\label{prop:proppoln} We assume that the category $\mathfrak{M}$
is finitely generated by the monoidal structure in $(\mathfrak{M}',\natural,0)$.
We denote by $E$ a finite generating set of $\mathfrak{M}$. Let
$d$ be a natural number. Then:

\begin{enumerate}
\item The category $\mathcal{P}ol_{d}^{strong}(\mathfrak{M},\mathscr{A})$
is closed under the translation functor, under quotient, under extension
and under colimits. The category $\mathcal{VP}ol_{d}(\mathfrak{M},\mathscr{A})$
is closed under the translation functor, under normal subobjects and
under extensions.
\item An object $F$ of $\mathbf{Fct}(\mathfrak{M},\mathscr{A})$ belongs
to $\mathcal{P}ol_{d}^{strong}(\mathfrak{M},\mathscr{A})$ or to $\mathcal{VP}ol_{d}(\mathfrak{M},\mathscr{A})$
if and only if the conditions on the evanescence and difference functors
of Definition \ref{def:strongandverystrong} are satisfied (only)
by the objects of $E\cap\textrm{Obj}(\mathfrak{M})_{\natural}$.
\item An object $F$ of $\mathbf{Fct}(\mathfrak{M},\mathscr{A})$ belongs
to $\mathcal{P}ol_{0}^{strong}(\mathfrak{M},\mathscr{A})$ if and
only if it is the quotient of a constant object of $\mathbf{Fct}(\mathfrak{M},\mathscr{A})$.
\end{enumerate}
\end{prop}

\subsection{Weak polynomial functors\label{subsec:Weak-polynomiality}}

We deal here with the concept of \textit{weak} polynomial functor.
It is introduced by Djament and Vespa in \cite[Section 2]{DV3} in
the category $\mathbf{Fct}(S,A)$ where $S$ is a symmetric monoidal
category where the unit is an initial object, and $A$ is a Grothendieck
category. Weak polynomial functors form a \textit{localizing} subcategory
of $\mathbf{Fct}(S,A)$; see Proposition \ref{prop:thickweak}. In
particular, this notion happens to be more appropriate to study the
stable behaviour for objects of the category $\mathbf{Fct}(S,A)$
(see for example \cite[Section 6]{DV3}, \cite{djamentcongruence}
and Section \ref{subsec:Weak-polynomial-functors-1}) and provide
a new organizing tool for polynomial functors and a fortiori representations
of families of groups; see Section \ref{subsec:Weak-polynomial-functors-1}.

We introduce the definition and properties of weak polynomial functors
to the present larger setting. In particular, considering a reliable
Long-Moody system\textit{ }$\{\mathcal{A},\mathcal{G},\mathcal{G}',\chi\}$,
the notion of weak polynomial functor is well-defined for the category
$\mathbf{Fct}(\mathfrak{U}\mathcal{G},R\textrm{-}\mathfrak{Mod})$
where $\mathfrak{U}\mathcal{G}$ is the full subcategory of Quillen's
bracket construction $\mathfrak{U}\mathcal{G}'$ on the objects of
$\mathcal{G}$; see Definition \ref{def:UGfullsbcatUG'}. We refer
the reader to \cite[Chapitres II et III]{gabriel} for general notions
on abelian categories and quotient abelian category which are required
for this section. We recall that a \textit{Grothendieck category}
is a cocomplete abelian category, which admits a generator, and in
which filtered colimits of exact sequences are exact.

\textbf{For the remainder of Section \ref{subsec:Weak-polynomiality},
we assume that the abelian category $\mathscr{A}$ is a Grothendieck
category. We recall that we consider $(\mathfrak{M}',\natural,0)$
a strict pre-braided monoidal small category where the unit $0$ is
an initial object and $\mathfrak{M}$ a small full subcategory of
$(\mathfrak{M}',\natural,0)$ finitely generated by the monoidal structure.
}Therefore, the functor category $\mathbf{Fct}(\mathfrak{M},\mathscr{A})$
is a Grothendieck category; see \cite{gabriel}. We recall that we
defined a particular set of objects of $\mathfrak{M}'$ denoted $\textrm{Obj}(\mathfrak{M})_{\natural}$
in Notation \ref{nota:goodsetobj}.

Let $F$ be an object of $\mathbf{Fct}(\mathfrak{M},\mathscr{A})$.
We denote the subfunctor $\sum_{m\in\textrm{Obj}(\mathfrak{M})_{\natural}}\kappa_{m}F$
of $F$ by $\kappa(F)$. Let $K(\mathfrak{M},\mathscr{A})$ be the
full subcategory of $\mathbf{Fct}(\mathfrak{M},\mathscr{A})$ of the
objects $F$ such that $\kappa(F)=F$. We have the following basic
properties:
\begin{lem}
\label{rem:fourre-toutfilteredcolim} The functor $\kappa$ is left
exact. Moreover, the functor $\kappa(F)$ is an object of $K(\mathfrak{M},\mathscr{A})$
for all objects $F$ of $\mathbf{Fct}(\mathfrak{M},\mathscr{A})$.
\end{lem}

\begin{proof}
A filtration on the evanescence functors $\{\kappa_{m}\}_{m\in\textrm{Obj}(\mathfrak{M})_{\natural}}$
is given by the canonical inclusions $\kappa_{n'}\hookrightarrow\kappa_{n'\natural n}$
and $\kappa_{n}\hookrightarrow\kappa_{n'\natural n}$ induced by the
morphisms $n\to n'\natural n$ and $n'\to n'\natural n$. Hence, since
filtered colimits are exact in $\mathbf{Fct}(\mathfrak{M},\mathscr{A})$
as it is a Grothendieck category, the functor $\kappa$ is left exact
as the filtered colimit of the left exact functors $\{\kappa_{m}\}_{m\in\textrm{Obj}(\mathfrak{M})_{\natural}}$.
Also, we note that the functor $\kappa$ commutes with filtered colimits
since it is a filtered colimit of finite limits; see \cite[Chapter IX, section 2]{MacLane1}.

Let $x$ be an object of $\textrm{Obj}(\mathfrak{M})_{\natural}$.
By definition, there is a canonical inclusion $\kappa_{m}(\kappa_{x}F)\hookrightarrow\kappa_{x}F$
for each $m\in\textrm{Obj}(\mathfrak{M})_{\natural}$. Since $\kappa_{x}$
is the kernel of a natural transformation between the identity functor
and a left exact functor, we note that the inclusion $\kappa_{x}(\kappa_{x}F)\hookrightarrow\kappa_{x}F$
is an isomorphism. Hence, as the functor $\kappa(\kappa_{x}F)$ is
the filtered colimit of the functors $\{\kappa_{m}(\kappa_{x}F)\}_{m\in\textrm{Obj}(\mathfrak{M})_{\natural}}$,
we deduce from the universal property of a colimit that $\kappa(\kappa_{x}F)=\kappa_{x}F$.
Therefore, viewing the functor $\kappa(F)$ as the filtered colimit
of the functors $\{\kappa_{x}F\}_{x\in\textrm{Obj}(\mathfrak{M})_{\natural}}$,
the second result follows from the commutation of $\kappa$ with filtered
colimits.
\end{proof}
The following proposition is the key property to define weak polynomial
functors. It extends the result \cite[Corollary 2.15]{DV3}, although
its proof is quite different. We recall that we follow the terminology
of \cite{gabriel}: a subcategory is \textit{thick }if it closed under
subobjects, quotients and extensions; this notion is also known as
a \textit{Serre subcategory}.
\begin{prop}
\label{prop:Snthick} The category $K(\mathfrak{M},\mathscr{A})$
is a thick subcategory of $\mathbf{Fct}(\mathfrak{M},\mathscr{A})$
and it is closed under colimits.
\end{prop}

\begin{proof}
Let us prove that $K(\mathfrak{M},\mathscr{A})$ is a thick subcategory
of $\mathbf{Fct}(\mathfrak{M},\mathscr{A})$. First, let $B$ be a
subfunctor of an object $F$ of $K(\mathfrak{M},\mathscr{A})$. As
$\mathbf{Fct}(\mathfrak{M},\mathscr{A})$ is a Grothendieck category,
we denote by $F/B$ the quotient. Since $\kappa$ is left exact, the
following diagram, where the rows are exact and the vertical arrows
are the inclusions, is commutative:

\[
\xymatrix{0\ar[r] & \kappa(B)\ar[r]\ar@{^{(}->}[d] & \kappa(F)\ar[r]\ar@{^{(}->}[d] & \kappa(F/B)\ar@{^{(}->}[d]\\
0\ar[r] & B\ar[r] & F\ar[r] & F/B\ar[r] & 0.
}
\]
Since $\kappa(F)=F$, it follows from the $4$-lemma that the inclusion
$\kappa\left(B\right)\hookrightarrow B$ is an equality: $K(\mathfrak{M},\mathscr{A})$
is thus closed under subobject.

Now, let $f:F\rightarrow Q\rightarrow0$ be an epimorphism of $\mathbf{Fct}(\mathfrak{M},\mathscr{A})$
such that $F$ is an object of $K(\mathfrak{M},\mathscr{A})$. We
consider the following commutative diagram where the vertical arrows
are the inclusions:
\[
\xymatrix{\kappa(F)\ar[r]^{\kappa(f)}\ar@{^{(}->}[d] & \kappa(Q)\ar@{^{(}->}[d]\\
F\ar@{->}[r]^{f} & Q\ar[r] & 0.
}
\]
Thus, since $\kappa(F)=F$, then the arrow $\kappa(Q)\hookrightarrow Q$
is also an epimorphism and a fortiori an equality. Hence, $K\left(\mathfrak{M},\mathscr{A}\right)$
is closed under quotient.

Finally, let $0\to B\rightarrow F\to Q\to0$ be a short exact sequence
of $\mathbf{Fct}(\mathfrak{M},\mathscr{A})$ such that $B$ and $Q$
are objects $K(\mathfrak{M},\mathscr{A})$. Let $m$ be an object
of $\textrm{Obj}(\mathfrak{M})_{\natural}$. Let $F_{m}$ be the pullback
of the morphisms $F\twoheadrightarrow Q$ and $\kappa_{m}(Q)\hookrightarrow Q$:
the functor $F$ is thus the filtered colimit (with respect to the
inclusions) of the pullbacks $\{F_{m}\}_{m\in\textrm{Obj}(\mathfrak{M})_{\natural}}$.
We recall that the functor $\kappa$ is the filtered colimit with
respect to the inclusions of the evanescence functors $\{\kappa_{x}\}_{x\in\textrm{Obj}(\mathfrak{M})_{\natural}}$
and that filtered colimits in $\mathbf{Fct}(\mathfrak{M},\mathscr{A})$
are exact since it is a Grothendieck category. Hence, the functor
$\kappa$ commutes with filtered colimits. Therefore, it is enough
to prove that each $F_{m}$ is in $K(\mathfrak{M},\mathscr{A})$ for
all $m\in\textrm{Obj}(\mathfrak{M})_{\natural}$ to show that $K(\mathfrak{M},\mathscr{A})$
is closed under extension.

Let $B_{m}$ be the kernel of $F_{m}\twoheadrightarrow\kappa_{m}(Q)$.
It follows from the universal property of a kernel and the four lemma
that the canonical inclusions $\kappa_{m}(Q)\hookrightarrow Q$ and
$F_{m}\hookrightarrow F$ induce an inclusion $B_{m}\hookrightarrow B$.
We deduce from the closure $K(\mathfrak{M},\mathscr{A})$ under subobjects
(proved above) that $B_{m}$ is an object of $K(\mathfrak{M},\mathscr{A})$.
Furthermore, we recall that, since $\kappa_{m}$ is the kernel of
a natural transformation between the identity functor and a left exact
functor, the composition $\kappa_{m}\circ\kappa_{m}$ is isomorphic
to $\kappa_{m}$ and therefore $i_{m}(\kappa_{m}(Q))=0$. By the universal
property of the kernel, there exists a unique morphism $\varphi_{m}$
such that the following diagram is commutative:
\[
\xymatrix{0\ar[r] & B_{m}\ar[d]_{i_{m}(B_{m})}\ar[r]^{\alpha} & F_{m}\ar[d]^{i_{m}(F_{m})}\ar[r]\ar@{-->}[dl]^{\varphi_{m}} & \kappa_{m}(Q)\ar[d]^{i_{m}(\kappa_{m}(Q))=0}\ar[r] & 0\\
0\ar[r] & \tau_{m}(B_{m})\ar[r] & \tau_{m}(F_{m})\ar[r] & \tau_{m}(\kappa_{m}(Q))\ar[r] & 0.
}
\]
For all $n\in\textrm{Obj}(\mathfrak{M})_{\natural}$, let $\varphi_{m}^{-1}(\tau_{m}(\kappa_{n}(B_{m})))$
be the pullback of the morphisms $\varphi_{m}:F_{m}\rightarrow\tau_{m}(B_{m})$
and $\tau_{m}(\kappa_{n}(B_{m}))\hookrightarrow\tau_{m}(B_{m})$.
As a pullback commutes with a filtered colimit in an abelian category
and since $\tau_{m}$ commutes with filtered colimits and since $\kappa(B_{m})=B_{m}$,
we deduce that
\[
\underset{n\in\textrm{Obj}(\mathfrak{M})_{\natural}}{\textrm{Colim}}\left(\varphi_{m}^{-1}(\tau_{m}(\kappa_{n}(B_{m})))\right)=F_{m}.
\]
In addition, since $\mathfrak{M}'$ is pre-braided monoidal, we deduce
from the relation (\ref{eq:defbraid}) that $(b_{m,n}^{\mathfrak{M}'})^{-1}\circ(\iota_{n}\natural id_{m})=id_{m}\natural\iota_{n}$.
Hence the precomposition by $(b_{m,n}^{\mathfrak{M}'})^{-1}\natural id_{\mathfrak{M}}$
defines a natural isomorphism $((b_{m,n}^{\mathfrak{M}'})^{-1}\natural id_{\mathfrak{M}})^{*}:\tau_{m}\circ\tau_{n}\overset{\sim}{\rightarrow}\tau_{n}\circ\tau_{m}$
for all $n\in\textrm{Obj}(\mathfrak{M})_{\natural}$, such that
\[
((b_{m,n}^{\mathfrak{M}'})^{-1}\natural id_{\mathfrak{M}})^{*}\circ(\tau_{m}(i_{n}F))=((b_{m,n}^{\mathfrak{M}'})^{-1}\natural id_{\mathfrak{M}})^{*}\circ F(\iota_{n}\natural id_{m}\natural-)=F(id_{m}\natural\iota_{n}\natural-)=i_{n}(\tau_{m}F)
\]
for each object $F$ of $\mathbf{Fct}(\mathfrak{M},\mathscr{A})$.
Hence, the following diagram is commutative for all $n\in\textrm{Obj}(\mathfrak{M})_{\natural}$:
\[
\xymatrix{ &  &  &  & F_{m}\ar[ld]_{\varphi_{m}}\ar[d]^{i_{m}(F_{m})}\ar[rrrd]^{i_{n\natural m}(F_{m})}\\
 &  &  & \tau_{m}(B_{m})\ar[d]^{i_{n}(\tau_{m}(B_{m}))}\ar[dlll]_{\tau_{m}(i_{n}(B_{m}))}\ar[r]^{\tau_{m}(\alpha)} & \tau_{m}(F_{m})\ar[rrr]^{i_{n}(\tau_{m}(F_{m}))} &  &  & \tau_{n}\tau_{m}(F_{m})\cong\tau_{n\natural m}(F_{m}).\\
\tau_{m}\tau_{n}(B_{m})\ar[rrr]_{((b_{m,n}^{\mathfrak{M}'})^{-1}\natural id_{\mathfrak{M}})^{*}}^{\cong} &  &  & \tau_{n}\tau_{m}(B_{m})\ar[rrrru]_{\tau_{n}\tau_{m}(\alpha)}
}
\]
We deduce from the previous commutative diagram and the universal
property of the kernel that there exists an inclusion morphism $\varphi_{m}^{-1}(\tau_{m}(\kappa_{n}(B_{m})))\hookrightarrow\kappa_{n\natural m}(F_{m})$
for all $n\in\textrm{Obj}(\mathfrak{M})_{\natural}$. Using the definition
of $\kappa$ as a filtered colimit, we deduce that $\textrm{Colim}_{n\in\textrm{Obj}(\mathfrak{M})_{\natural}}(\varphi_{m}^{-1}(\tau_{m}(\kappa_{n}(B_{m}))))$
is a subobject of $\kappa(F_{m})$. Hence, we have $\kappa(F_{m})=F_{m}$
and thus $K(\mathfrak{M},\mathscr{A})$ is closed under extension.

Let us now prove that $K(\mathfrak{M},\mathscr{A})$ is closed under
colimits. We recall that coproducts are exact in $\mathbf{Fct(\mathfrak{M},\mathscr{A})}$
since it is a Grothendieck category; see \cite[Corollary 2.8.9]{Popescu}
for instance. Hence, for each object $m$ in $\textrm{Obj}(\mathfrak{M})_{\natural}$,
the evanescence functor $\kappa_{m}$ commutes with the coproducts
of $\mathbf{Fct(\mathfrak{M},\mathscr{A})}$. Since the functor $\kappa$
is the filtered colimit of the functors $\{\kappa_{m}\}_{m\in\textrm{Obj}(\mathfrak{M})_{\natural}}$,
we deduce that $\kappa$ commutes with the coproducts of $\mathbf{Fct(\mathfrak{M},\mathscr{A})}$:
therefore, the category $K(\mathfrak{M},\mathscr{A})$ is closed under
coproducts. As $\mathbf{Fct(\mathfrak{M},\mathscr{A})}$ is a Grothendieck
category, any colimit of $K(\mathfrak{M},\mathscr{A})$ may be expressed
as a quotient of a coproduct. Since $K(\mathfrak{M},\mathscr{A})$
is closed under quotient (see above), this category is thus closed
under colimits.
\end{proof}
\begin{rem}
We see in the proof of Proposition \ref{prop:Snthick} why we require
the category $\mathscr{A}$ to have more properties than just being
an abelian category: it is necessary to assume that the filtered colimits
in the category $\mathscr{A}$ are exact, which is the case for a
Grothendieck category. Actually, we could only assume that the category
$\mathscr{A}$ is an $AB5$-category, ie a cocomplete abelian category
in which filtered colimits of exact sequences are exact, and all the
work of Section \ref{subsec:Weak-polynomiality} extends mutatis mutandis.
However, this more general notion is less standard than the one of
Grothendieck category that we chose to use for sake of simplicity.
\end{rem}

The thickness property of Proposition \ref{prop:Snthick} ensures
that we can consider the quotient category of $\mathbf{Fct(\mathfrak{M},\mathscr{A}})$
by $K(\mathfrak{M},\mathscr{A})$; see \cite[Chapitre III]{gabriel}.
Let $\mathbf{St}(\mathfrak{M},\mathscr{A})$ denote the quotient category
$\mathbf{Fct}(\mathfrak{M},\mathscr{A})/K(\mathfrak{M},\mathscr{A})$.
We consider the canonical functor associated with this quotient by
$\pi_{\mathfrak{M}}:\mathbf{Fct}(\mathfrak{M},\mathscr{A})\rightarrow\mathbf{St}(\mathfrak{M},\mathscr{A})$.
It is exact, essentially surjective and commutes with all colimits.
Since the functor category $\mathbf{Fct}(\mathfrak{M},\mathscr{A})$
is a Grothendieck category, it follows from Proposition \ref{prop:Snthick}
that the subcategory $K(\mathfrak{M},\mathscr{A})$ is \textit{localizing}
(see \cite[Proposition III.4.8]{gabriel} or \cite[4.7.14, p.315]{Popescu2}),
in the sense that the functor $\pi_{\mathfrak{M}}$ admits a right
adjoint functor denoted by $s_{\mathfrak{M}}:\mathbf{Fct}(\mathfrak{M},\mathscr{A})/K(\mathfrak{M},\mathscr{A})\rightarrow\mathbf{Fct}(\mathfrak{M},\mathscr{A})$
and called the \textit{section functor}. 

The following proposition introduces the induced translation and difference
functors on the category $\mathbf{St}(\mathfrak{M},\mathscr{A})$.
Its proof is verbatim that of \cite[Proposition 2.19]{DV3}, using
Proposition \ref{prop:lemmecaract}.
\begin{prop}
\label{prop:piMcommutetauetdelta} Let $m\in\textrm{Obj}(\mathfrak{M})_{\natural}$.
The translation functor $\tau_{m}$ and the difference functor $\delta_{m}$
of $\mathbf{Fct}(\mathfrak{M},\mathscr{A})$ respectively induce an
exact endofunctor of $\mathbf{St}(\mathfrak{M},\mathscr{A})$ which
commute with colimits, again called the translation functor $\tau_{m}$
and the difference functor $\delta_{m}$ respectively. In addition:

\begin{enumerate}
\item The following relations hold: $\delta_{m}\circ\pi_{\mathfrak{M}}=\pi_{\mathfrak{M}}\circ\delta_{m}$
and $\tau_{m}\circ\pi_{\mathfrak{M}}=\pi_{\mathfrak{M}}\circ\tau_{m}$.
\item The exact sequence (\ref{eq:ESCaract-1}) induces a short exact sequence
of endofunctors of $\mathbf{St}(\mathfrak{M},\mathscr{A})$:
\begin{equation}
0\longrightarrow Id\overset{i_{m}}{\longrightarrow}\tau_{m}\overset{\varDelta_{m}}{\longrightarrow}\delta_{m}\longrightarrow0.\label{eq:ESCaract}
\end{equation}
\item For another object $m'$ of $\mathfrak{M}$, the endofunctors $\delta_{m}$,
$\delta_{m'}$, $\tau_{m}$ and $\tau_{m'}$ of $\mathbf{St}(\mathfrak{M},\mathscr{A})$
pairwise commute up to natural isomorphism.
\end{enumerate}
\end{prop}

We can now introduce the notion of a weak polynomial functor, extending
that of \cite[Definition 2.22]{DV3}.
\begin{defn}
\label{def:defweakpoly} We recursively define on $d\in\mathbb{N}$
the category $\mathcal{P}ol_{d}(\mathfrak{M},\mathscr{A})$ of polynomial
functors of degree less than or equal to $d$ to be the full subcategory
of $\mathbf{St}\left(\mathfrak{M},\mathscr{A}\right)$ as follows:

\begin{enumerate}
\item If $d<0$, the objects of $\mathcal{P}ol_{d}(\mathfrak{M},\mathscr{A})$
are those of $K(\mathfrak{M},\mathscr{A})$;
\item if $d\geq0$, the objects of $\mathcal{P}ol_{d}(\mathfrak{M},\mathscr{A})$
are the functors $F$ such that the functor $\delta_{x}(F)$ is an
object of $\mathcal{P}ol_{d-1}(\mathfrak{M},\mathscr{A})$ for all
$x\in\textrm{Obj}(\mathfrak{M})_{\natural}$.
\end{enumerate}
For an object $F$ of $\mathbf{St}(\mathfrak{M},\mathscr{A})$ which
is polynomial of degree less than or equal to $d\in\mathbb{N}$, the
smallest natural number $n\leq d$ for which $F$ is an object of
$\mathcal{P}ol_{d}(\mathfrak{M},\mathscr{A})$ is called the degree
of $F$. An object $F$ of $\mathbf{Fct}(\mathfrak{M},\mathscr{A})$
is \textit{weak polynomial} of degree at most $d$ if its image $\pi_{\mathfrak{M}}(F)$
is an object of $\mathcal{P}ol_{d}(\mathfrak{M},\mathscr{A})$. The
degree of polynomiality of $\pi_{\mathfrak{M}}(F)$ is called the
(weak) degree of $F$.
\end{defn}

Let us give some important properties of the categories of weak polynomial
functors used in Sections \ref{sec:Behaviour-of-the} and \ref{sec:Examples-and-applications}.
Their proofs follow verbatim their analogues in \cite[Section 2]{DV3}.
\begin{prop}
\cite[Propositions 2.24-2.26]{DV3} \label{prop:thickweak}\label{prop:weakpolydeg0}
For $d$ a natural number, the subcategory $\mathcal{P}ol_{d}(\mathfrak{M},\mathscr{A})$
of $\mathbf{St}(\mathfrak{M},\mathscr{A})$ is thick and closed under
limits and colimits. Furthermore, there is an equivalence of categories
$\mathscr{A}\simeq\mathcal{P}ol_{0}(\mathfrak{M},\mathscr{A})$.

\label{enough1} We assume that the category $\mathfrak{M}$ is finitely
generated by the monoidal structure in $(\mathfrak{M}',\natural,0)$
and we denote by $E$ a finite generating set of $\mathfrak{M}$.
Let $F$ be an object of $\mathbf{St}(\mathfrak{M},\mathscr{A})$.
Then, the functor $F$ is an object of $\mathcal{P}ol_{d}(\mathfrak{M},\mathscr{A})$
if and only if the functor $\delta_{e}(F)$ is an object of $\mathcal{P}ol_{d-1}(\mathfrak{M},\mathscr{A})$
for all objects $e$ of $E\cap\textrm{Obj}(\mathfrak{M})_{\natural}$.
\end{prop}

\section{Behaviour of the Long-Moody functors on polynomial functors\label{sec:Behaviour-of-the}}

In this section, we study the effect of some generalized Long-Moody
functors on (very) strong and weak polynomial functors. Indeed, under
some additional assumptions, they have the property to increase by
one both the very strong and the weak polynomial degrees; see Theorems
\ref{Thm:emairesult} and \ref{thm:ResultWeakpoly}.

For all the work of this section, we fix a \textit{coherent} Long-Moody
system $\{\mathcal{A},\mathcal{G},\mathcal{G}',\chi\}$; see Section
\ref{subsec:Functoriality-of-the}. We recall that $\mathfrak{Gr}$
denotes the category of groups and that the free product is denoted
by $*$. Let $\mathcal{G}'_{(0,1)}$ be the small full subgroupoid
of $(\mathcal{G}',\natural,0_{\mathcal{G}'})$ of the finite monoidal
products on the objects $0_{\mathcal{G}'}$, $0$ and $1$ of $\mathcal{G}'$.
Note that the monoidal structure $\natural$ restricts to give $\mathcal{G}'_{(0,1)}$
a braided monoidal structure. We assume that the functors $\mathcal{A}$
and $\chi$ satisfy the following additional properties:
\begin{assumption}
\label{ass:coherenceconditionbnautfn2et3}\label{assu:decomposeAfreeproduct}The
functor $\mathcal{A}:\mathfrak{U}\mathcal{G}\to\mathfrak{Gr}$ is
the restriction of a functor $\mathfrak{U}\mathcal{G}'\to\mathfrak{Gr}$
(that we also denote by $\mathcal{A}$) along the canonical inclusion
$\mathfrak{U}\mathcal{G}\hookrightarrow\mathfrak{U}\mathcal{G}'$,
and there exist two groups $H_{0}$ and $H$ (with $H$ non-trivial),
such that:

\begin{itemize}
\item for all objects $\text{\ensuremath{\underline{n}}}$ of $\mathcal{G}$,
$\mathcal{A}(\underline{n})=H^{*n}*H_{0}$ that we denote by $H_{n}$,
and $\mathcal{A}(1^{\natural m})=H^{*m}$ for all $m\in\mathbb{N}$.
\item for all objects $X$ and $Y$ of $\mathcal{G}'_{(0,1)}$, $\mathcal{A}(\iota_{X}\natural id_{Y})=\iota_{\mathcal{A}(X)}*id_{\mathcal{A}(Y)}$
and $\mathcal{A}(id_{Y}\natural\iota_{X})=id_{\mathcal{A}(Y)}*\iota_{\mathcal{A}(X)}$,
where $\iota_{G}:0_{\mathfrak{Gr}}\rightarrow G$ denotes the unique
morphism from $0_{\mathfrak{Gr}}$ to a group $G$.
\end{itemize}
Moreover, the family of group morphisms $\{\chi_{n}:H_{n}\rightarrow G_{n+1}\}_{n\in\mathbb{N}}$
induced by $\chi$ satisfies the equality (\ref{eq:equiva}) of Proposition
\ref{cond:conditionstability}. Namely we assume that in $G_{n+2}$,
for all elements $h\in H_{n}$:
\begin{eqnarray}
((b_{1,1}^{\mathcal{G}'})^{-1}\natural id_{\underline{n}})\circ(id_{1}\natural\chi_{n}(h)) & = & \chi_{n+1}((\iota_{H}\ast id_{H_{n}})(h))\circ((b_{1,1}^{\mathcal{G}'})^{-1}\natural id_{\underline{n}}).\label{eq:equiva'}
\end{eqnarray}
\end{assumption}

\begin{defn}
\label{def:reliableLongMoody}A coherent Long-Moody system $\{\mathcal{A},\mathcal{G},\mathcal{G}',\chi\}$
is said to be \textit{reliable} if Assumption \ref{assu:decomposeAfreeproduct}
is satisfied.
\end{defn}

Consequences of Assumption \ref{assu:decomposeAfreeproduct} are heavily
used in our study; see the proofs of Lemmas \ref{lem:explicit_def_xsi}
and \ref{lem:balanced_2} and of Proposition \ref{prop:key_relation}.
Some of the results presented in Section \ref{subsec:Relation-with-evanescence}
still hold without the hypotheses of Assumption \ref{assu:decomposeAfreeproduct}.
However, these additional properties are necessary to prove some of
the further results. In particular, we point out some relations which
are used in the proof of Proposition \ref{prop:key_relation}.
\begin{cor}
Let $m$ be a natural number and $X$ be an object of $\mathcal{G}'_{(0,1)}$.
We have:

\begin{equation}
\mathcal{A}((b_{m,X}^{\mathcal{G}'})^{-1})\circ(\iota_{\mathcal{A}(X)}*id_{H^{*m}})=id_{H^{*m}}*\iota_{\mathcal{A}(X)}.\label{eq:key_equality_A_2}
\end{equation}
Also, for $n$ another natural number and for all $g\in G_{n}$, we
have 
\begin{equation}
\mathcal{A}(id_{m}\natural g)\circ(id_{H^{*m}}*\iota_{H_{n}})=id_{H^{*m}}*\iota_{H_{n}}.\label{eq:key_equality_A_1}
\end{equation}
\end{cor}

\begin{proof}
We first recall from Assumption \ref{ass:coherenceconditionbnautfn2et3}
that $\mathcal{A}(\iota_{X}\natural id_{m})=\iota_{\mathcal{A}(X)}*id_{H^{*m}}$.
Since $\mathfrak{U}\mathcal{G}'$ is a pre-braided monoidal category,
we deduce from the relation (\ref{eq:defbraid}) that $(b_{m,X}^{\mathcal{G}'})^{-1}\circ(\iota_{X}\natural id_{m})=id_{m}\natural\iota_{X}$.
Therefore, the equality (\ref{eq:key_equality_A_2}) follows from
the functoriality of $\mathcal{A}$ over the category $\mathfrak{U}\mathcal{G}'$.
Furthermore, we recall that $0_{\mathcal{G}'}$ is a initial object
in the category $\mathfrak{U}\mathcal{G}'$ and that $\iota_{\underline{n}}$
is the unique morphism from $0_{\mathcal{G}'}$ to $\underline{n}$
in $\mathfrak{U}\mathcal{G}'$. We deduce that $(id_{m}\natural g)\circ(id_{m}\natural\iota_{\underline{n}})=id_{m}\natural(g\circ\iota_{\underline{n}})=id_{m}\natural\iota_{\underline{n}}$.
Hence, since $id_{H^{*m}}*\iota_{H_{n}}=\mathcal{A}(id_{m}\natural\iota_{\underline{n}})$,
the functoriality of $\mathcal{A}$ over the category $\mathfrak{U}\mathcal{G}'$
gives the equality (\ref{eq:key_equality_A_1}).
\end{proof}
\textbf{We assume that the fixed coherent Long-Moody system $\{\mathcal{A},\mathcal{G},\mathcal{G}',\chi\}$
is reliable.} Note that such functors $\mathcal{A}$ and $\chi$ always
exist: we can at least consider the functor $\mathcal{A}_{id}$ defined
assigning $\mathcal{A}(g)=id_{\mathcal{A}(\underline{n})}$ to all
$g\in G_{n}$ and all natural numbers $n$, and the trivial functor
$\chi_{tr}$. Also, we show in Section \ref{sec:Examples-and-applications}
that most of the coherent Long-Moody systems introduced in Section
\ref{subsec:Examples} are reliable.

We consider the Long-Moody functor $\mathbf{LM}_{\{\mathcal{A},\mathcal{G},\mathcal{G}',\chi\}}$
associated with the reliable Long-Moody system $\{\mathcal{A},\mathcal{G},\mathcal{G}',\chi\}$,
which is fixed throughout this section. In particular, we omit $\{\mathcal{A},\mathcal{G},\mathcal{G}',\chi\}$
from the notation most of the time. Since the category $\mathfrak{U}\mathcal{G}$
is generated by the objects $\underline{0}$ and $1$ using the monoidal
product $\natural$, it is enough for our work to only consider the
translation functor $\tau_{1}$ by Propositions \ref{prop:proppoln}
and \ref{enough1}.

\subsection{Relation with evanescence and difference functors\label{subsec:Relation-with-evanescence}}

In this section, we describe the decomposition of the Long-Moody functor
$\mathbf{LM}_{\{\mathcal{A},\mathcal{G},\mathcal{G}',\chi\}}$ with
respect to the translation functor $\tau_{1}$; see Corollary \ref{cor:splittingtranslation}.
We then establish the crucial results stated in Theorem \ref{thm:Splitting LM},
describing the behaviour of the Long-Moody functor $\mathbf{LM}_{\{\mathcal{A},\mathcal{G},\mathcal{G}',\chi\}}$
with respect to the evanescence and difference functors.

\subsubsection{Factorization of the natural transformation $i_{1}\mathbf{LM}$ by
$\mathbf{LM}(i_{1})$\label{subsec:Factorisation-of-the}}

Recall from Proposition \ref{prop:lemmecaract} the exact sequence
in the category of endofunctors of $\mathbf{Fct}(\mathfrak{U}\mathcal{G},R\textrm{-}\mathfrak{Mod})$,
which defines the natural transformation $i_{1}$:
\begin{equation}
\xymatrix{0\ar@{->}[r] & \kappa_{1}\ar@{->}[r]^{\Omega_{1}} & Id\ar@{->}[r]^{i_{1}} & \tau_{1}\ar@{->}[r]^{\varDelta_{1}} & \delta_{1}\ar@{->}[r] & 0}
.\label{eq:definges}
\end{equation}
As we are interested in the effect of the considered Long-Moody functor
$\mathbf{LM}$ on (very) strong and weak polynomial functors, our
objective is to study the cokernel of the natural transformation $i_{1}\mathbf{LM}:\mathbf{LM}\rightarrow\tau_{1}\circ\mathbf{LM}$.
We recall from the relation (\ref{eq:def_monoidal_UG}) that $\iota_{1}\natural id_{\underline{n}}=[1,id_{\underline{n+1}}]$.
Then, for $F$ an object of $\mathbf{Fct}(\mathfrak{U}\mathcal{G},R\textrm{-}\mathfrak{Mod})$
and for all natural numbers $n$, $i_{1}\mathbf{LM}$ is defined by
the morphisms:
\[
(i_{1}\mathbf{LM})(F)_{\underline{n}}=\mathbf{LM}(F)(\iota_{1}\natural id_{\underline{n}})=\mathbf{LM}(F)([1,id_{\underline{n+1}}]):\mathbf{LM}(F)(\underline{n})\rightarrow\tau_{1}\mathbf{LM}(F)(\underline{n}).
\]
Observe that, since the associated Long-Moody functor is right-exact
by Proposition \ref{prop:exactnessLM}, we have the following exact
sequence:
\begin{equation}
\xymatrix{\mathbf{LM}\ar@{->}[rr]^{\mathbf{LM}(i_{1})} &  & \mathbf{LM}\circ\tau_{1}\ar@{->}[rr]^{\mathbf{LM}(\varDelta_{1})} &  & \mathbf{LM}\circ\delta_{1}\ar@{->}[r] & 0}
.\label{eq:SESLMi1}
\end{equation}
Moreover, if the groups $H_{0}$ and $H$ are free, as the associated
Long-Moody functor is then exact by Corollary \ref{cor:exactnessLM},
we note that the following sequence is exact:
\begin{equation}
\xymatrix{0\ar@{->}[r] & \mathbf{LM}\circ\kappa_{1}\ar@{->}[rr]^{{\color{white}ooo}\mathbf{LM}(\Omega_{1})} &  & \mathbf{LM}\ar@{->}[rr]^{\mathbf{LM}(i_{1})} &  & \mathbf{LM}\circ\tau_{1}\ar@{->}[rr]^{\mathbf{LM}(\varDelta_{1})} &  & \mathbf{LM}\circ\delta_{1}\ar@{->}[r] & 0}
.\label{eq:LESLMi1}
\end{equation}

First, we prove that the functor $i_{1}\mathbf{LM}$ factors through
$\mathbf{LM}(i_{1})$ via a natural transformation $\xi$ defined
as follows. Let $F$ be an object of $\mathbf{Fct}(\mathfrak{U}\mathcal{G},R\textrm{-}\mathfrak{Mod})$
and $n$ be a natural number. We recall from the definition of a Long-Moody
functor and of the translation endofunctor $\tau_{1}$ that 
\[
(\mathbf{LM}\circ\tau_{1})(F)(\underline{n})=\mathcal{I}_{R[H_{n}]}\underset{R[H_{n}]}{\otimes}F(2\natural\underline{n})\textrm{ and }(\tau_{1}\circ\mathbf{LM})(F)(\underline{n})=\mathcal{I}_{R[H_{n+1}]}\underset{R[H_{n+1}]}{\otimes}F(2\natural\underline{n}).
\]
In particular, for $(\tau_{1}\circ\mathbf{LM})(F)(\underline{n})$,
the $R[G_{n+2}]$-module $F(2\natural\underline{n})$ is an $R[H_{n+1}]$-module
via the composition $F\circ\chi_{n+1}:H_{n+1}\rightarrow G_{n+2}\rightarrow\textrm{Aut}_{R}(F(2\natural\underline{n}))$.
We also note that $F(2\natural\underline{n})$ is an $R[H_{n}]$-module
using $F\circ(id_{1}\natural\chi_{n}(-)):H_{n}\rightarrow G_{n+2}\rightarrow\textrm{Aut}_{R}(F(2\natural\underline{n}))$
for $(\mathbf{LM}\circ\tau_{1})(F)(\underline{n})$. 

By Assumption \ref{assu:decomposeAfreeproduct}, the augmentation
ideal functor $\mathcal{I}_{R[\mathcal{A}]}$ defines the monomorphism
$\mathcal{I}_{R[\mathcal{A}]}(\iota_{1}\natural id_{\underline{n}}):\mathcal{I}_{R[H_{n}]}\hookrightarrow\mathcal{I}_{R[H_{n+1}]}$.
We also consider the automorphism $F((b_{1,1}^{\mathfrak{\mathcal{G}'}})^{-1}\natural id_{\underline{n}})$
of $F(2\natural\underline{n})$. Let $\otimes_{R[H_{n+1}]}$ denote
the canonical $R[H_{n+1}]$-balanced projection from $\mathcal{I}_{R[H_{n+1}]}\times F(2\natural\underline{n})$
to $\mathcal{I}_{R[H_{n+1}]}\otimes_{R[H_{n+1}]}F(2\natural\underline{n})$.
We denote by $\hat{\xi}(F)_{\underline{n}}$ the composition $\otimes_{R[H_{n+1}]}\circ(\mathcal{I}_{R[\mathcal{A}]}(\iota_{1}\natural id_{\underline{n}})\times F((b_{1,1}^{\mathfrak{\mathcal{G}'}})^{-1}\natural id_{\underline{n}}))$.
\begin{lem}
\label{lem:explicit_def_xsi} The morphism $\hat{\xi}(F)_{\underline{n}}$
is $R[H_{n}]$-balanced.
\end{lem}

\begin{proof}
We first note that it is enough to check that $\hat{\xi}(F)_{\underline{n}}$
is $H_{n}$-balanced, the result for $R[H_{n}]$ following by linearity.
We fix $h\in H_{n}$, $i\in\mathcal{I}_{R[H_{n}]}$ and $v\in F(2\natural\underline{n})$.
We deduce from the condition on $\mathcal{A}(\iota_{1}\natural id_{\underline{n}})$
of Assumption \ref{assu:decomposeAfreeproduct} that
\[
\mathcal{I}_{R[\mathcal{A}]}(\iota_{1}\natural id_{\underline{n}})(i\cdot h)=\mathcal{I}_{R[\mathcal{A}]}(\iota_{1}\natural id_{\underline{n}})(i)\cdot(\iota_{H}\ast id_{H_{n}})(h).
\]
Then, it follows from the functoriality of $F$ and from the relation
(\ref{eq:equiva'}) of Assumption \ref{assu:decomposeAfreeproduct}
that:
\begin{eqnarray*}
\hat{\xi}(F)_{\underline{n}}(i\cdot h\underset{R[H_{n}]}{\otimes}v) & = & \mathcal{I}_{R[\mathcal{A}]}(\iota_{1}\natural id_{\underline{n}})(i)\underset{R[H_{n+1}]}{\otimes}F(\chi_{n+1}((\iota_{H}\ast id_{H_{n}})(h)))\circ F((b_{1,1}^{\mathfrak{\mathcal{G}'}})^{-1}\natural id_{\underline{n}})(v)\\
 & = & \mathcal{I}_{R[\mathcal{A}]}(\iota_{1}\natural id_{\underline{n}})(i)\underset{R[H_{n+1}]}{\otimes}F((b_{1,1}^{\mathfrak{\mathcal{G}'}})^{-1}\natural id_{\underline{n}})\circ F(id_{1}\natural\chi_{n}(h))(v)\\
 & = & \hat{\xi}(F)_{\underline{n}}(i\underset{R[H_{n}]}{\otimes}\tau_{1}F(\chi_{n}(h))(v)),
\end{eqnarray*}
which ends the proof.
\end{proof}
\begin{rem}
\label{rem:adduseful0}We stress that the condition (\ref{eq:equiva'})
of Assumption \ref{assu:decomposeAfreeproduct} is required for Lemma
\ref{lem:explicit_def_xsi}.
\end{rem}

It follows from Lemma \ref{lem:explicit_def_xsi} that the universal
property of the tensor product over $R[H_{n}]$ defines a unique morphism
$\xi(F)_{\underline{n}}$ from $(\mathbf{LM}\circ\tau_{1})(F)(\underline{n})$
to $\tau_{1}\mathbf{LM}(F)(\underline{n})$ for each $n\in\mathbb{N}$.
Actually, there is a more explicit description of the morphism $\xi(F)_{\underline{n}}$:
it is a routine to check that the conjugation of $\xi(F)_{\underline{n}}$
by the canonical isomorphism from $\mathcal{I}_{R[H_{n}]}\otimes_{R[H_{n}]}F(2\natural\underline{n})$
to $(\mathcal{I}_{R[H_{n}]}\otimes_{R[H_{n}]}R[H*H_{n}])\otimes_{R[H_{n+1}]}F(2\natural\underline{n})$
is equal to the morphism
\begin{equation}
(\mathcal{I}_{R[\mathcal{A}]}(\iota_{1}\natural id_{\underline{n}})\underset{R[H_{n}]}{\otimes}id_{R[H*H_{n}]})\underset{R[H_{n+1}]}{\otimes}F((b_{1,1}^{\mathfrak{\mathcal{G}'}})^{-1}\natural id_{\underline{n}}).\label{eq:alternative_xsi_n}
\end{equation}
This explicit definition allows in particular to prove the following
property:
\begin{lem}
\label{lem:xsi_mono}The morphism $\xi(F)_{\underline{n}}$ is a monomorphism
for each $n$.
\end{lem}

\begin{proof}
We recall that $t_{G}:G\rightarrow0_{\mathfrak{Gr}}$ denotes the
unique morphism from the group $G$ to $0_{\mathfrak{Gr}}$. We denote
by $\mathcal{I}_{R[\mathcal{A}]}^{-1}(\iota_{1}\natural id_{\underline{n}}):\mathcal{I}_{R[H_{n+1}]}\twoheadrightarrow\mathcal{I}_{R[H_{n}]}$
the $R$-module surjection induced by the group surjection $t_{H}*id_{H_{n}}:H*H_{n}\twoheadrightarrow H_{n}$.
Since $\mathcal{I}_{R[\mathcal{A}]}^{-1}(\iota_{1}\natural id_{\underline{n}})$
is a left inverse for $\mathcal{I}_{R[\mathcal{A}]}(\iota_{1}\natural id_{\underline{n}})$,
we deduce that the morphism $(\mathcal{I}_{R[\mathcal{A}]}^{-1}(\iota_{1}\natural id_{\underline{n}})\otimes_{R[H_{n}]}id_{R[H*H_{n}]})\otimes_{R[H_{n+1}]}F(b_{1,1}^{\mathfrak{\mathcal{G}'}}\natural id_{\underline{n}})$
defines a left inverse for $\xi(F)_{\underline{n}}$.
\end{proof}
The morphisms of the form $\xi(F)_{\underline{n}}$ are the key to
define the natural transformation $\xi$:
\begin{prop}
\label{prop:subfunclmtm}The morphisms $\{\xi(F)_{\underline{n}}\}$
define a natural transformation $\xi(F):(\mathbf{LM}\circ\tau_{1})(F)\rightarrow(\tau_{1}\circ\mathbf{LM})(F)$.
Then this yields a natural transformation $\xi:\mathbf{LM}\circ\tau_{1}\rightarrow\tau_{1}\circ\mathbf{LM}$.
\end{prop}

\begin{proof}
Let $n$ and $n'$ be natural numbers such that $n'\geq n$, and $[n'-n,g]\in\textrm{Hom}_{\mathfrak{U}\mathcal{G}}(\underline{n},\underline{n'})$.
Since $\mathfrak{s}_{\mathcal{A}}^{*}\circ\chi^{*}=1\natural-$ as
endofunctors of $\mathfrak{U}\mathcal{G}$, we recall from Lemma \ref{lem:explicit_def_T}
that for all $i\in\mathcal{I}_{R[H_{n}]}$ and $v\in F(2\natural\underline{n})$:
\[
(\mathbf{LM}\circ\tau_{1})(F)([n'-n,g])(i\underset{R[H_{n}]}{\otimes}v)=\mathcal{I}_{R[\mathcal{A}]}([n'-n,g])(i)\underset{R[H_{n'}]}{\otimes}F(id_{1}\natural id_{1}\natural[n'-n,g])(v),
\]
\[
(\tau_{1}\circ\mathbf{LM})(F)([n'-n,g])(i\underset{R[H_{n}]}{\otimes}v)=\mathcal{I}_{R[\mathcal{A}]}(id_{1}\natural[n'-n,g])(i)\underset{R[H_{n'+1}]}{\otimes}F(id_{1}\natural id_{1}\natural[n'-n,g])(v).
\]
It follows from the defining equivalence relation (\ref{eq:equivalence relation})
of $\mathfrak{U}\mathcal{G}'$ that $[1+n'-n,(b_{1,n'-n}^{\mathfrak{\mathcal{G}'}})^{-1}\natural id_{\underline{n}}]=[1+n'-n,id_{\underline{n'+1}}]$.
We then compute that $(id_{1}\natural[n'-n,g])\circ(\iota_{1}\natural id_{\underline{n}})=[1+n'-n,id_{1}\natural g]$.
Since $\mathcal{I}_{R[\mathcal{A}]}$ is a functor over $\mathfrak{U}\mathcal{G}'$,
we deduce that:
\[
\mathcal{I}_{R[\mathcal{A}]}(id_{1}\natural[n'-n,g])\circ(\iota_{1}\natural id_{\underline{n}})=\mathcal{I}_{R[\mathcal{A}]}([1+n'-n,id_{1}\natural g])=\mathcal{I}_{R[\mathcal{A}]}(\iota_{1}\natural id_{\underline{n'}})\circ\mathcal{I}_{R[\mathcal{A}]}([n'-n,g]).
\]
Also, as a consequence of the properties of a monoidal structure,
we compute that:
\[
(id_{2}\natural[n'-n,g])\circ((b_{1,1}^{\mathfrak{\mathcal{G}'}})^{-1}\natural id_{\underline{n}})=(b_{1,1}^{\mathfrak{\mathcal{G}'}})^{-1}\natural[n'-n,g]=((b_{1,1}^{\mathfrak{\mathcal{G}'}})^{-1}\natural id_{\underline{n'}})\circ(id_{2}\natural[n'-n,g]).
\]
Hence, we deduce from the definitions that for all $i\in\mathcal{I}_{R[H_{n}]}$
and $v\in F(2\natural\underline{n})$
\begin{eqnarray*}
((\tau_{1}\circ\mathbf{LM})(F)([n'-n,g]))\circ\xi(F)_{\underline{n}}(i\underset{R[H_{n}]}{\otimes}v) & = & \mathcal{I}_{R[\mathcal{A}]}([1+n'-n,id_{1}\natural g])(i)\underset{R[H_{n'+1}]}{\otimes}F((b_{1,1}^{\mathfrak{\mathcal{G}'}})^{-1}\natural[n'-n,g])(v)\\
 & = & \xi(F)_{\underline{n'}}\circ((\mathbf{LM}\circ\tau_{1})(F)([n'-n,g]))(i\underset{R[H_{n}]}{\otimes}v).
\end{eqnarray*}
Therefore, $\xi(F)$ is a natural transformation from $(\tau_{1}\circ\mathbf{LM})(F)$
to $(\mathbf{LM}\circ\tau_{1})(F)$. Let us now check that assembling
all these natural transformations for all the objects of $\mathbf{Fct}(\mathfrak{U}\mathcal{G},R\textrm{-}\mathfrak{Mod})$
defines a natural transformation from $\tau_{1}\circ\mathbf{LM}$
to $\mathbf{LM}\circ\tau_{1}$. Let $M$ and $N$ be two objects of
$\mathbf{Fct}(\mathfrak{U}\mathcal{G},R\textrm{-}\mathfrak{Mod})$
and $\eta:M\rightarrow N$ be a natural transformation. It follows
from Lemma \ref{lem:explicit_def_T} that the natural transformations
$(\mathbf{LM}\circ\tau_{1})(\eta)$ and $(\tau_{1}\circ\mathbf{LM})(\eta)$
are respectively given by:
\[
(\mathbf{LM}\circ\tau_{1})(\eta)_{\underline{n}}=id_{\mathcal{I}_{R[H_{n}]}}\underset{R[H_{n}]}{\varotimes}\eta_{2\natural\underline{n}}\textrm{ and }(\tau_{1}\circ\mathbf{LM})(\eta)_{\underline{n}}=id_{\mathcal{I}_{R[H_{1+n}]}}\underset{R[H_{1+n}]}{\varotimes}\eta_{2\natural\underline{n}}
\]
for all natural numbers $n$. Since $\eta$ is a natural transformation,
we have $N((b_{1,1}^{\mathfrak{\mathcal{G}'}})^{-1}\natural id_{\underline{n}})\circ\eta_{2\natural\underline{n}}=\eta_{2\natural\underline{n}}\circ M((b_{1,1}^{\mathfrak{\mathcal{G}'}})^{-1}\natural id_{\underline{n}})$
and we deduce that $\xi(N)_{\underline{n}}\circ(\mathbf{LM}\circ\tau_{1})(\eta)_{\underline{n}}=(\tau_{1}\circ\mathbf{LM})(\eta)_{\underline{n}}\circ\xi(M)_{\underline{n}}$.
Hence $\xi(M)\circ(\mathbf{LM}\circ\tau_{1})(\eta)$ and $(\tau_{1}\circ\mathbf{LM})(\eta)\circ\xi(N)$
are equal as natural transformations from $(\mathbf{LM}\circ\tau_{1})(M)$
to $(\mathbf{LM}\circ\tau_{1})(N)$, which ends the proof.
\end{proof}
Finally, we prove that the natural transformation $i_{1}\mathbf{LM}$
factors across $\mathbf{LM}(i_{1})$. As a result, we obtain that
the endofunctor $\mathbf{LM}\circ\delta_{1}$ is a subfunctor of $\delta_{1}\circ\mathbf{LM}$;
see diagram (\ref{eq:diagram_Xsi}).
\begin{prop}
\label{lem:relationLMi1andi1LM}As natural transformations from $\mathbf{LM}$
to $\tau_{1}\circ\mathbf{LM}$, the following equality holds:
\[
\xi\circ(\mathbf{LM}(i_{1}))=i_{1}\mathbf{LM}.
\]
Moreover, there exists a unique natural transformation $\mathbf{LM}\circ\delta_{1}\rightarrow\delta_{1}\circ\mathbf{LM}$
such that the following diagram is commutative and the rows are exact
sequences in the category of endofunctors of $\mathbf{Fct}(\mathfrak{U}\mathcal{G},R\textrm{-}\mathfrak{Mod})$:
\begin{equation}
\xymatrix{0\ar@{->}[r] & \mathbf{LM}\circ\tau_{1}\ar@{->}[r]^{\xi}\ar@{->>}[d] & \tau_{1}\circ\mathbf{LM}\ar@{->}[r]^{\varUpsilon}\ar@{->>}[d] & \textrm{Coker}(\xi)\ar@{->}[r]\ar@{=}[d] & 0\\
0\ar@{->}[r] & \mathbf{LM}\circ\delta_{1}\ar@{->}[r] & \delta_{1}\circ\mathbf{LM}\ar@{->}[r] & \textrm{Coker}(\xi)\ar@{->}[r] & 0.
}
\label{eq:diagram_Xsi}
\end{equation}
\end{prop}

\begin{proof}
Let $F$ be an object of $\mathbf{Fct}(\mathfrak{U}\mathcal{G},R\textrm{-}\mathfrak{Mod})$
and $n$ be a natural number. It follows from the definition of a
Long-Moody functor on morphisms and natural transformations (see Lemma
\ref{lem:explicit_def_T}) and the definition of the natural transformation
$i_{1}$ (see Section \ref{subsec:Prerequisite-on-strong}) that

\[
\mathbf{LM}(i_{1}F)_{\underline{n}}=id_{\mathcal{I}_{R[H_{n}]}}\underset{R[H_{1+n}]}{\varotimes}F(\iota_{1}\natural id_{1}\natural id_{\underline{n}}),
\]

\[
(i_{1}\mathbf{LM})(F)_{\underline{n}}=\mathbf{LM}(F)(\iota_{1}\natural id_{\underline{n}})=\mathcal{I}_{\mathcal{A}}(\iota_{1}\natural id_{\underline{n}})\underset{R[H_{1+n}]}{\varotimes}F(id_{1}\natural\iota_{1}\natural id_{\underline{n}}).
\]
Since $\mathfrak{U}\mathcal{G}'$ is pre-braided, we have $(b_{1,1}^{\mathfrak{\mathcal{G}'}})^{-1}\circ(\iota_{1}\natural id_{1})=id_{1}\natural\iota_{1}$
by the relation (\ref{eq:defbraid}) and we deduce from the definition
of $\xi$ that $\xi(F)_{\underline{n}}\circ\mathbf{LM}(i_{1}F)_{\underline{n}}=(i_{1}\mathbf{LM})(F)_{\underline{n}}$.
Hence the following diagram is commutative for any object $F$ of
$\mathbf{Fct}(\mathfrak{U}\mathcal{G},R\textrm{-}\mathfrak{Mod})$
\[
\xymatrix{\mathbf{LM}(F)\ar@{->}[d]_{(i_{1}\mathbf{LM})(F)}\ar@{->}[dr]^{\mathbf{LM}(i_{1}F)}\\
(\mathbf{LM}\circ\tau_{1})(F)\ar@{->}[r]_{\xi(F)} & (\tau_{1}\circ\mathbf{LM})(F),
}
\]
which proves the first statement. Then, we recall that for two morphisms
$f$ and $g$ of an abelian category such that the target of $f$
is the source of $g$ and $g$ is a monomorphism, the snake lemma
provides a short exact sequence for the cokernels $0\rightarrow\textrm{Coker}(f)\rightarrow\textrm{Coker}(g\circ f)\rightarrow\textrm{Coker}(g)\rightarrow0$.
Since $\mathbf{Fct}(\mathfrak{U}\mathcal{G},R\textrm{-}\mathfrak{Mod})$
is an abelian category, the bottom short exact sequence of (\ref{eq:diagram_Xsi})
is an instance of that result assigning $f=\mathbf{LM}(i_{1})$ and
$g=\xi$ by using the equality $\xi\circ(\mathbf{LM}(i_{1}))=i_{1}\mathbf{LM}$
and Lemma (\ref{lem:xsi_mono}). Finally, in the diagram (\ref{eq:diagram_Xsi}),
the top short exact sequence is obtained from the definition of $\xi$
and Lemma (\ref{lem:xsi_mono}), and the commutativities for the two
squares formally follow from the definitions of the arrows of the
bottom short exact sequence.
\end{proof}

\subsubsection{Study of $\textrm{Coker}(\xi)$\label{subsec:Studycoker}}

By Proposition \ref{lem:relationLMi1andi1LM}, we know that the endofunctor
$\tau_{1}\circ\mathbf{LM}$ is an extension of $\textrm{Coker}(\xi)$
by $\mathbf{LM}\circ\tau_{1}$; see diagram (\ref{eq:diagram_Xsi}).
In this section, we prove that this extension is actually trivial;
see Corollary \ref{cor:splittingtranslation}. For this purpose, we
first make the following observations.

Let\textit{ $F$} be an object of $\mathbf{Fct}(\mathfrak{U}\mathcal{G},R\textrm{-}\mathfrak{Mod})$
and $n$ be a natural number. We recall a classical decomposition
for the augmentation ideal of a free group, which may be found in
\cite[Proposition 6.2.9]{Weibel1} (stated for $R=\mathbb{Z}$, but
which can be straightforwardly generalized to any commutative ring
$R$) or in \cite[Section 4, Lemma 4.3 and Theorem 4.7]{cohencohomo}
(we note that these statements are done for groups which are not necessarily
of cohomological dimension at most $1$). For two groups $A$ and
$B$, there is a right $R[A*B]$-module isomorphism:
\begin{equation}
(\mathcal{I}_{R[A]}\underset{R[A]}{\otimes}R[A*B])\oplus(\mathcal{I}_{R[B]}\underset{R[B]}{\otimes}R[A*B])\cong\mathcal{I}_{R[A*B]}.\label{eq:general_decomposition_ideal}
\end{equation}
where $R[A*B]$ is an $R[A]$-module via the canonical inclusion $id_{A}*\iota_{B}:A\rightarrow A*B$
and an $R[B]$-module via $\iota_{A}\ast id_{B}:B\rightarrow A*B$.
We recall that $t_{G}:G\rightarrow0_{\mathfrak{Gr}}$ denotes the
unique morphism from the group $G$ to $0_{\mathfrak{Gr}}$. We denote
by $\mathcal{I}(id_{A}*\iota_{B}):\mathcal{I}_{R[A]}\hookrightarrow\mathcal{I}_{R[A*B]}$
and $\mathcal{I}(\iota_{A}*id_{B}):\mathcal{I}_{R[B]}\hookrightarrow\mathcal{I}_{R[A*B]}$
the $R$-module injections induced by the group injections $id_{A}*\iota_{B}:A\hookrightarrow A*B$
and $\iota_{A}*id_{B}:B\hookrightarrow A*B$. The isomorphism (\ref{eq:general_decomposition_ideal})
is explicitly defined by the direct sum of the injections $\mathcal{I}(id_{A}*\iota_{B})\varotimes_{R[A]}id_{R[A*B]}$
and $\mathcal{I}(\iota_{A}*id_{B})\varotimes_{R[B]}id_{R[A*B]}$;
we refer the reader to \cite[Proposition 6.2.9]{Weibel1} for further
details.

Applying the isomorphism (\ref{eq:general_decomposition_ideal}) provides
the following Lemma \ref{lem:application_ideal}. Before this, we
recall and introduce some notations:
\begin{notation}
\label{nota:complement_IA}Let $n$ and $n'$ be natural numbers such
that $n'\geq n$. We recall from Assumption \ref{assu:decomposeAfreeproduct}
that the functor $\mathcal{A}$ is defined over the category $\mathfrak{U}\mathcal{G}'$
and that $\mathcal{A}(id_{n}\natural\iota_{\underline{n'-n}})$ is
equal to $id_{H^{*n}}*\iota_{H_{n'-n}}$. Also, by definition (see
Section \ref{subsec:Framework-of-the}), the augmentation ideal functor
$\mathcal{I}_{R[\mathcal{A}]}$ is defined over the category $\mathfrak{U}\mathcal{G}'$
and in particular defines the $R$-module injection $\mathcal{I}_{R[\mathcal{A}]}(id_{n}\natural\iota_{\underline{n'-n}}):\mathcal{I}_{R[H^{*n}]}\hookrightarrow\mathcal{I}_{R[H_{n'}]}$
induced by the group injection $id_{H^{*n}}*\iota_{H_{n'-n}}$. We
denote by $\mathcal{I}_{R[\mathcal{A}]}^{-1}(id_{n}\natural\iota_{\underline{n'-n}}):\mathcal{I}_{R[H_{n'}]}\twoheadrightarrow\mathcal{I}_{R[H^{*n}]}$
the $R$-module surjection induced by the group surjection $id_{H^{*n}}*t_{H_{n'-n}}:H_{n'}\twoheadrightarrow H^{*n}$.
\end{notation}

\begin{lem}
\label{lem:application_ideal}The direct sum of $(\mathcal{I}_{R[\mathcal{A}]}(id_{1}\natural\iota_{\underline{n}})\varotimes_{R[H]}id_{R[H*H_{n}]})\varotimes_{R[H_{1+n}]}id_{F(2\natural\underline{n})}$
and $\xi(F)_{\underline{n}}$ provides an $R$-module isomorphism:
\begin{eqnarray}
((\mathcal{I}_{R[H]}\underset{R[H]}{\otimes}R[H*H_{n}])\underset{R[H_{n+1}]}{\otimes}F(2\natural\underline{n}))\oplus(\mathbf{LM}\circ\tau_{1})(F)(\underline{n}) & \overset{\cong}{\longrightarrow} & \tau_{1}\mathbf{LM}(F)(\underline{n}).\label{eq:(*)-1}
\end{eqnarray}
\end{lem}

\begin{proof}
We deduce from the isomorphism (\ref{eq:general_decomposition_ideal}),
from the definition of $\tau_{1}\mathbf{LM}(F)(\underline{n})$ and
from the distributivity of the tensor product that the direct sum
of the morphisms $(\mathcal{I}_{R[\mathcal{A}]}(id_{1}\natural\iota_{\underline{n}})\otimes_{R[H]}id_{R[H*H_{n}]})\otimes_{R[H_{n+1}]}id_{F(2\natural\underline{n})}$
and $(\mathcal{I}_{R[\mathcal{A}]}(\iota_{1}\natural id_{\underline{n}})\otimes_{R[H_{n}]}id_{R[H*H_{n}]})\otimes_{R[H_{n+1}]}id_{F(2\natural\underline{n})}$
defines an $R$-module isomorphism:
\[
\left((\mathcal{I}_{R[H]}\underset{R[H]}{\otimes}R[H*H_{n}])\underset{R[H_{n+1}]}{\otimes}F(2\natural\underline{n})\right)\oplus\left((\mathcal{I}_{R[H_{n}]}\underset{R[H_{n}]}{\otimes}R[H*H_{n}])\underset{R[H_{n+1}]}{\otimes}F(2\natural\underline{n})\right)\overset{\cong}{\longrightarrow}\tau_{1}\mathbf{LM}(F)(\underline{n}).
\]
Also, precomposing the right-hand summand with the automorphism $(id_{R[H_{n}]}\otimes_{R[H_{n}]}id_{R[H*H_{n}]})\otimes_{R[H_{n+1}]}F((b_{1,1}^{\mathfrak{\mathcal{G}'}})^{-1}\natural id_{\underline{n}})$,
the injection associated with the right hand summand is equal to the
morphism (\ref{eq:alternative_xsi_n}). Hence, using the canonical
isomorphism between $\mathcal{I}_{R[H_{n}]}\otimes_{R[H_{n}]}F(2\natural\underline{n})$
and $(\mathcal{I}_{R[H_{n}]}\otimes_{R[H_{n}]}R[H*H_{n}])\otimes_{R[H_{n+1}]}F(2\natural\underline{n})$,
we identify the right-hand summand with the $R$-module $(\mathbf{LM}\circ\tau_{1})(F)(\underline{n})$
and associated injection $\xi(F)_{\underline{n}}$.
\end{proof}
It follows from Lemma \ref{lem:application_ideal} and from the universal
property of a cokernel that the $R$-module $\textrm{Coker}(\xi)(F(\underline{n}))$
is isomorphic to the left hand summand of (\ref{eq:(*)-1}). We consider
the natural transformation $\varUpsilon:\tau_{1}\circ\mathbf{LM}\rightarrow\textrm{Coker}(\xi)$
of Proposition \ref{lem:relationLMi1andi1LM}. We deduce from the
isomorphism (\ref{eq:(*)-1}) that the morphism $\varUpsilon(F)_{\underline{n}}$
is equal to $(\mathcal{I}_{R[\mathcal{A}]}^{-1}(id_{1}\natural\iota_{\underline{n}})\varotimes_{R[H]}id_{R[H*H_{n}]})\varotimes_{R[H_{1+n}]}id_{F(2\natural\underline{n})}$.
This leads ineluctably to wonder if the left-hand side of the isomorphism
(\ref{eq:(*)-1}) is a direct sum of two endofunctors of $\mathbf{Fct}(\mathfrak{U}\mathcal{G},R\textrm{-}\mathfrak{Mod})$.

\paragraph{Identification with a translation functor.}

We aim at identifying the left-hand summand the isomorphism (\ref{eq:(*)-1})
as the defining input of a certain endofunctor of $\mathbf{Fct}(\mathfrak{U}\mathcal{G},R\textrm{-}\mathfrak{Mod})$
for each object $F$ of $\mathbf{Fct}(\mathfrak{U}\mathcal{G},R\textrm{-}\mathfrak{Mod})$
and each natural number $n$. Beforehand, we highlight the following
technical relation, which is used in the proof of Proposition \ref{prop:I_tens_tau_well_defined}.
We recall that $e_{G}$ denotes the unit element of a group $G$.
\begin{prop}
\label{prop:key_relation}Let $n$ and $n'$ be natural numbers such
that $n'\geq n$, let $[n'-n,g]\in\textrm{Hom}_{\mathfrak{U}\mathcal{G}}(\underline{n},\underline{n'})$
and let $h\in H$. Then:
\begin{equation}
\chi_{n'+1}(h*e_{H_{n'}})\circ(id_{2}\natural[n'-n,g])=(id_{2}\natural[n'-n,g])\circ\chi_{n+1}(h*e_{H_{n}}).\label{eq:highlighted_relation}
\end{equation}
\end{prop}

\begin{proof}
We deduce from the relation (\ref{eq:equiva'}) that the following
square of the groupoid $\mathcal{G}$ is commutative:
\[
\xymatrix{2\natural\underline{n'}\ar@{->}[rrrr]^{id_{n'-n-k}\natural\chi_{n+k}(e_{H^{*k-1}}*h*e_{H_{n}})}\ar@{->}[d]_{id_{n'-n-k}\natural(b_{1,1}^{\mathcal{G}'})^{-1}\natural id_{\underline{k+n}}} &  &  &  & 2\natural\underline{n'}\ar@{->}[d]^{id_{n'-n-k}\natural(b_{1,1}^{\mathcal{G}'})^{-1}\natural id_{\underline{k+n}}}\\
2\natural\underline{n'}\ar@{->}[rrrr]_{id_{n'-n-k}\natural\chi_{n+k+1}(e_{H^{*k}}*h*e_{H_{n}})} &  &  &  & 2\natural\underline{n'}
}
\]
for each $1\leq k\leq n'-n$. Composing the above squares for all
$1\leq k\leq n'-n$ and recalling that we have
\begin{eqnarray*}
b_{1,n'-n}^{\mathcal{G}'} & = & (id_{n'-n-1}\natural b_{1,1}^{\mathcal{G}'})\circ(id_{n'-n-2}\natural b_{1,1}^{\mathcal{G}'}\natural id_{1})\circ\cdots\circ(id_{1}\natural b_{1,1}^{\mathcal{G}'}\natural id_{n'-n-2})\circ(b_{1,1}^{\mathcal{G}'}\natural id_{n'-n-1})
\end{eqnarray*}
by the definition of a braiding, we deduce the following equality:
\begin{equation}
((b_{1,n'-n}^{\mathcal{G}'})^{-1}\natural id_{1}\natural id_{\underline{n}})\circ(id_{n'-n}\natural\chi_{n+1}(h*e_{H_{n}}))=\chi_{n'+1}(e_{H^{*n'-n}}*h*e_{H_{n}})\circ((b_{1,n'-n}^{\mathcal{G}'})^{-1}\natural id_{1}\natural id_{\underline{n}})\label{eq:proofcompatibilitytensor0}
\end{equation}
Furthermore, it follows from the relation (\ref{eq::equivcond:coherenceconditionsigmanan})
induced by Lemma \ref{lem:equivcond:coherenceconditionsigmanan} that
\begin{eqnarray*}
 &  & (id_{1}\natural((b_{1,n'-n}^{\mathcal{G}'})^{-1}\natural id_{\underline{n}}))\circ\chi_{n'+1}(e_{H^{*n'-n}}*h*e_{H_{n}})\\
 & = & \chi_{n'+1}(\mathcal{A}((b_{1,n'-n}^{\mathcal{G}'})^{-1}\natural id_{\underline{n}})(e_{H^{*n'-n}}*h*e_{H_{n}}))\circ(id_{1}\natural((b_{1,n'-n}^{\mathcal{G}'})^{-1}\natural id_{\underline{n}})).
\end{eqnarray*}
We deduce from the relation (\ref{eq:key_equality_A_2}) that $\mathcal{A}((b_{1,n'-n}^{\mathcal{G}'})^{-1}\natural id_{\underline{n}})(e_{H^{*n'-n}}*h*e_{H_{n}})=h*e_{H_{n'}}$
and thus note that
\begin{equation}
(id_{1}\natural((b_{1,n'-n}^{\mathcal{G}'})^{-1}\natural id_{\underline{n}}))\circ\chi_{n'+1}(e_{H^{*n'-n}}*h*e_{H_{n}})=\chi_{n'+1}(h*e_{H_{n'}})\circ(id_{1}\natural((b_{1,n'-n}^{\mathcal{G}'})^{-1}\natural id_{\underline{n}})).\label{eq:proofcompatibilitytensor-1}
\end{equation}
We recall from the definition of the monoidal structure in $\mathfrak{U}\mathcal{G}$
(see the relation (\ref{eq:def_monoidal_UG})) that:
\[
(id_{2}\natural[n'-n,id_{\underline{n'}}])\circ\chi_{n+1}(h*e_{H_{n}})=[n'-n,((b_{2,n'-n}^{\mathcal{G}'})^{-1}\natural id_{\underline{n}})\circ(id_{n'-n}\natural\chi_{n+1}(h*e_{H_{n}})].
\]
By the definition of a braiding, we have $b_{2,n'-n}^{\mathcal{G}'}=(b_{1,n'-n}^{\mathcal{G}'}\natural id_{1})\circ(id_{1}\natural b_{1,n'-n}^{\mathcal{G}'})$.
Therefore, using the relations (\ref{eq:proofcompatibilitytensor0})
and (\ref{eq:proofcompatibilitytensor-1}), we prove the following
key equality as morphisms in $\mathfrak{U}\mathcal{G}$:
\begin{equation}
(id_{2}\natural[n'-n,id_{\underline{n'}}])\circ\chi_{n+1}(h*e_{H_{n}})=\chi_{n'+1}(h*e_{H_{n'}})\circ(id_{2}\natural[n'-n,id_{\underline{n'}}]).\label{eq:proofcompatibilitytensor1}
\end{equation}
On another note, since $(id_{2}\natural g)\circ\chi_{n'+1}(h*e_{H_{n'}})=\chi_{n'+1}(\mathcal{A}(id_{1}\natural g)(h*e_{H_{n'}}))\circ(id_{2}\natural g)$
by the relation (\ref{eq::equivcond:coherenceconditionsigmanan})
induced from Lemma \ref{lem:equivcond:coherenceconditionsigmanan}
, we deduce from the relation (\ref{eq:key_equality_A_1}) that
\begin{equation}
(id_{2}\natural g)\circ\chi_{n'+1}(h*e_{H_{n'}})=\chi_{n'+1}(h*e_{H_{n'}})\circ(id_{2}\natural g).\label{eq:proofcompatibilitytesnor2}
\end{equation}
Finally, the equality (\ref{eq:highlighted_relation}) follows from
the combination of the relations (\ref{eq:proofcompatibilitytensor1})
and (\ref{eq:proofcompatibilitytesnor2}) since $id_{2}\natural[n'-n,g]=(id_{2}\natural g)\circ[n'-n,id_{\underline{n'}}]$.
\end{proof}
We now introduce the endofunctor $\mathcal{I}_{R[H]}\otimes_{R[H]}\tau_{2}$
of $\mathbf{Fct}(\mathfrak{U}\mathcal{G},R\textrm{-}\mathfrak{Mod})$
to which we will identify $\textrm{Coker}(\xi)$. For each object
$F$ of $\mathbf{Fct}(\mathfrak{U}\mathcal{G},R\textrm{-}\mathfrak{Mod})$,
we assign $(\mathcal{I}_{R[H]}\otimes_{R[H]}\tau_{2})(F)(\underline{n})$
to be $\mathcal{I}_{R[H]}\otimes_{R[H]}(\tau_{2}F)(\underline{n})$
for each $n\in\mathbb{N}$. In particular, $(\tau_{2}F)(\underline{n})=F(2\natural\underline{n})$
is an $R[H]$-module via the composition $H\rightarrow G_{n+2}\rightarrow\textrm{Aut}_{R}(F(2\natural\underline{n}))$
defined by $F\circ\chi_{n+1}\circ(id_{H}\ast\iota_{H_{n}})$. For
each morphism $[n'-n,g]$ of $\mathfrak{U}\mathcal{G}$, we assign
$(\mathcal{I}_{R[H]}\otimes_{R[H]}\tau_{2})(F)([n'-n,g])$ to be the
tensor product morphism $id_{\mathcal{I}_{R[H]}}\otimes_{R[H]}F(id_{2}\natural[n'-n,g])$.
Finally, for each natural transformation $\eta:F\to G$ of $\mathbf{Fct}(\mathfrak{U}\mathcal{G},R\textrm{-}\mathfrak{Mod})$,
we assign $(\mathcal{I}_{R[H]}\otimes_{R[H]}\tau_{2})(\eta)$ to be
the natural transformation defined by $id_{\mathcal{I}_{R[H]}}\otimes_{R[H]}\eta_{2\natural\underline{n}}$
for each $n\in\mathbb{N}$.
\begin{prop}
\label{prop:I_tens_tau_well_defined}The above assignments define
an endofunctor $\mathcal{I}_{R[H]}\otimes_{R[H]}\tau_{2}$ of $\mathbf{Fct}(\mathfrak{U}\mathcal{G},R\textrm{-}\mathfrak{Mod})$.
\end{prop}

\begin{proof}
Let us prove that $\mathcal{I}_{R[H]}\otimes_{R[H]}\tau_{2}F$ is
an object of $\mathbf{Fct}(\mathfrak{U}\mathcal{G},R\textrm{-}\mathfrak{Mod})$
for any functor $F:\mathfrak{U}\mathcal{G}\to R\textrm{-}\mathfrak{Mod}$.
First, we check that the assignment $(\mathcal{I}_{R[H]}\otimes_{R[H]}\tau_{2})(F)([n'-n,g])$
is well-defined with respect to the tensor product over $R[H]$ for
any morphism $[n'-n,g]$ of $\mathfrak{U}\mathcal{G}$ as follows.
We fix $h\in H$, $i\in\mathcal{I}_{R[H_{n}]}$ and $v\in F(2\natural\underline{n})$.
Then, it follows from the relation (\ref{eq:highlighted_relation})
of Proposition \ref{prop:key_relation} and from the functoriality
of $F$ that:
\begin{eqnarray*}
(\mathcal{I}_{R[H]}\underset{R[H]}{\otimes}\tau_{2}F)([n'-n,g])(i\cdot h\underset{R[H]}{\otimes}v) & = & i\underset{R[H]}{\otimes}F(\chi_{n'+1}(h\ast e_{H_{n'}}))\circ F(id_{2}\natural[n'-n,g])(v)\\
 & = & i\underset{R[H]}{\otimes}F(id_{2}\natural[n'-n,g])\circ F(\chi_{n+1}(h\ast e_{H_{n}}))(v)\\
 & = & (\mathcal{I}_{R[H]}\underset{R[H]}{\otimes}\tau_{2}F)([n'-n,g])(i\underset{R[H]}{\otimes}F(\chi_{n+1}(h\ast e_{H_{n}}))(v)),
\end{eqnarray*}
which proves the consistency of the assignment on morphisms with respect
to the tensor product structure. Then, it is clear from the functoriality
of $F$ that the identity and composition axioms for the morphisms
are satisfied by the assignments defining $(\mathcal{I}_{R[H]}\otimes_{R[H]}\tau_{2})(F)$,
which is therefore a functor $\mathfrak{U}\mathcal{G}\rightarrow R\textrm{-}\mathfrak{Mod}$.

Furthermore, the fact that $\eta:F\to G$ is a natural transformation
directly implies that the assignment for $(\mathcal{I}_{R[H]}\otimes_{R[H]}\tau_{2})(\eta)$
defines a natural transformation from $(\mathcal{I}_{R[H]}\otimes_{R[H]}\tau_{2})(F)$
to $(\mathcal{I}_{R[H]}\otimes_{R[H]}\tau_{2})(G)$. Also, the identity
and composition axioms for the natural transformations (ie the morphisms
of the category $\mathbf{Fct}(\mathfrak{U}\mathcal{G},R\textrm{-}\mathfrak{Mod})$)
are straightforwardly checked from the definitions.
\end{proof}
Let us now introduce a natural transformation from $\mathcal{I}_{R[H]}\otimes_{R[H]}\tau_{2}$
to $\tau_{1}\circ\mathbf{LM}$. We fix an object $F$ of $\mathbf{Fct}(\mathfrak{U}\mathcal{G},R\textrm{-}\mathfrak{Mod})$
and a natural number $n$. We consider the $R$-module injection $\mathcal{I}_{R[\mathcal{A}]}(id_{1}\natural\iota_{\underline{n}}):\mathcal{I}_{R[H]}\hookrightarrow\mathcal{I}_{R[H_{n+1}]}$
recalled in Notation \ref{nota:complement_IA}. We recall that $\otimes_{R[H_{n+1}]}$
denotes the canonical $R[H_{n+1}]$-balanced projection from $\mathcal{I}_{R[H_{n+1}]}\times F(2\natural\underline{n})$
to $\mathcal{I}_{R[H_{n+1}]}\otimes_{R[H_{n+1}]}F(2\natural\underline{n})$.
In particular, we recall that the $R[H_{n+1}]$-module structure of
$F(2\natural\underline{n})$ is defined by the composition $F\circ\chi_{n+1}:H_{n+1}\rightarrow G_{n+2}\rightarrow\textrm{Aut}_{R}(F(2\natural\underline{n}))$.
We denote by $\hat{\upsilon}(F)_{\underline{n}}$ the composition
$\otimes_{R[H_{n+1}]}\circ(\mathcal{I}_{R[\mathcal{A}]}(id_{1}\natural\iota_{\underline{n}})\times id_{F(2\natural\underline{n})})$.
\begin{lem}
\label{lem:balanced_2}The morphism $\hat{\upsilon}(F)_{\underline{n}}$
is $R[H]$-balanced.
\end{lem}

\begin{proof}
Again, it is enough to check that $\hat{\upsilon}(F)_{\underline{n}}$
is $H$-balanced, the result for $R[H]$ following by linearity. We
fix $h\in H$, $i\in\mathcal{I}_{R[H]}$ and $v\in F(2\natural\underline{n})$.
We deduce from the condition on $\mathcal{A}(id_{1}\natural\iota_{\underline{n}})$
of Assumption \ref{assu:decomposeAfreeproduct} that $\mathcal{I}_{R[\mathcal{A}]}(id_{1}\natural\iota_{\underline{n}})(i\cdot h)=\mathcal{I}_{R[\mathcal{A}]}(id_{1}\natural\iota_{\underline{n}})(i)\cdot(h\ast e_{H_{n}})$.
It then follows from the definitions that:
\[
\hat{\upsilon}(F)_{\underline{n}}(i\cdot h\underset{R[H]}{\otimes}v)=\mathcal{I}_{R[\mathcal{A}]}(id_{1}\natural\iota_{\underline{n}})(i)\underset{R[H_{n}]}{\otimes}F(\chi_{n+1}(h\ast e_{H_{n}}))(v)=\hat{\upsilon}(F)_{\underline{n}}(i\underset{R[H]}{\otimes}F(\chi_{n+1}(h\ast e_{H_{n}}))(v))
\]
which ends the proof.
\end{proof}
It follows from Lemma \ref{lem:balanced_2} that the universal property
of the tensor product over $R[H]$ defines a unique morphism $\upsilon(F)_{\underline{n}}:\mathcal{I}_{R[H]}\otimes_{R[H]}F(2\natural\underline{n})\rightarrow\tau_{1}\mathbf{LM}(F)(\underline{n})$
for each $n\in\mathbb{N}$. It is a routine to check that the conjugation
of $\upsilon(F)_{\underline{n}}$ by the canonical isomorphism from
$(\mathcal{I}_{R[H]}\otimes_{R[H]}R[H*H_{n}])\otimes_{R[H_{n+1}]}F(2\natural\underline{n})$
to $\mathcal{I}_{R[H]}\otimes_{R[H]}F(2\natural\underline{n})$ is
equal to the morphism $(\mathcal{I}_{R[\mathcal{A}]}(id_{1}\natural\iota_{\underline{n}})\varotimes_{R[H]}id_{R[H*H_{n}]})\varotimes_{R[H_{1+n}]}id_{F(2\natural\underline{n})}$
(used in Lemma \ref{lem:application_ideal}). We can now introduce
the natural transformation $\upsilon$:
\begin{prop}
\label{prop:identificationtau2} We define a natural transformation
$\upsilon(F):\mathcal{I}_{R[H]}\otimes_{R[H]}\tau_{2}F\rightarrow(\tau_{1}\circ\mathbf{LM})(F)$
from the monomorphisms $\{\upsilon(F)_{\underline{n}}\}_{n\in\mathbb{N}}$.
This yields a natural transformation $\upsilon:\mathcal{I}_{R[H]}\otimes_{R[H]}\tau_{2}\rightarrow\tau_{1}\circ\mathbf{LM}$.
\end{prop}

\begin{proof}
Let $n$ and $n'$ be natural numbers such that $n'\geq n$, and $[n'-n,g]\in\textrm{Hom}_{\mathfrak{U}\mathcal{G}}(\underline{n},\underline{n'})$.
Since $\mathfrak{s}_{\mathcal{A}}^{*}\circ\chi^{*}=1\natural-$ as
endofunctors of $\mathfrak{U}\mathcal{G}$, we recall from the definitions
that for all $i\in\mathcal{I}_{R[H]}$ and $v\in F(2\natural\underline{n})$:
\[
(\mathcal{I}_{R[H]}\otimes_{R[H]}\tau_{2})(F)([n'-n,g])(i\underset{R[H]}{\otimes}v)=i\underset{R[H]}{\otimes}F(id_{1}\natural id_{1}\natural[n'-n,g])(v),
\]
\[
(\tau_{1}\circ\mathbf{LM})(F)([n'-n,g])(i\underset{R[H]}{\otimes}v)=\mathcal{I}_{R[\mathcal{A}]}(id_{1}\natural[n'-n,g])(i)\underset{R[H_{n'+1}]}{\otimes}F(id_{1}\natural id_{1}\natural[n'-n,g])(v).
\]
We recall from the definition of the monoidal structure in $\mathfrak{U}\mathcal{G}$
(see the relation (\ref{eq:def_monoidal_UG})) that $id_{1}\natural[n'-n,g]=(id_{1}\natural g)\circ((b_{1,n'-n}^{\mathfrak{\mathcal{G}'}})^{-1}\natural id_{\underline{n}})\circ(\iota_{n'-n}\natural id_{\underline{n+1}})$.
First, we deduce from the relation (\ref{eq:defbraid}) that $((b_{1,n'-n}^{\mathfrak{\mathcal{G}'}})^{-1}\natural id_{\underline{n}})\circ(\iota_{n'-n}\natural id_{1}\natural\iota_{\underline{n}})=id_{1}\natural\iota_{\underline{n'}}$.
Then, it follows from the fact that $\mathcal{I}_{R[\mathcal{A}]}$
is a functor over $\mathfrak{U}\mathcal{G}'$ that:
\begin{equation}
\mathcal{I}_{R[\mathcal{A}]}((b_{1,n'-n}^{\mathfrak{\mathcal{G}'}})^{-1}\natural id_{\underline{n}})\circ\mathcal{I}_{R[\mathcal{A}]}(\iota_{n'-n}\natural id_{\underline{n+1}})\circ\mathcal{I}_{R[\mathcal{A}]}(id_{1}\natural\iota_{\underline{n}})=\mathcal{I}_{R[\mathcal{A}]}(id_{1}\natural\iota_{\underline{n'}}).\label{eq:I_A_final}
\end{equation}
Furthermore, we recall that $(id_{1}\natural g)\circ(id_{1}\natural\iota_{\underline{n}})=id_{m}\natural(g\circ\iota_{\underline{n}})=id_{1}\natural\iota_{\underline{n}}$
since $0_{\mathcal{G}'}$ is a initial object in the category $\mathfrak{U}\mathcal{G}'$
and that $\iota_{\underline{n}}$ is the unique morphism from $0_{\mathcal{G}'}$
to $\underline{n}$ in $\mathfrak{U}\mathcal{G}'$. We deduce from
the functoriality of $\mathcal{I}_{R[\mathcal{A}]}$ over $\mathfrak{U}\mathcal{G}'$
that:
\begin{equation}
\mathcal{I}_{R[\mathcal{A}]}(id_{1}\natural g)\circ\mathcal{I}_{R[\mathcal{A}]}(id_{1}\natural\iota_{\underline{n}})=\mathcal{I}_{R[\mathcal{A}]}(id_{1}\natural\iota_{\underline{n}}).\label{eq:I_A_final_2}
\end{equation}
Then, combining the relations (\ref{eq:I_A_final}) and (\ref{eq:I_A_final_2}),
we prove the following key equality as morphisms in $\mathfrak{U}\mathcal{G}'$:

\begin{equation}
\mathcal{I}_{R[\mathcal{A}]}(id_{1}\natural[n'-n,g])\circ\mathcal{I}_{R[\mathcal{A}]}(id_{1}\natural\iota_{\underline{n}})=\mathcal{I}_{R[\mathcal{A}]}(id_{1}\natural\iota_{\underline{n}}).\label{eq:last_key}
\end{equation}
Therefore, we deduce from the definitions and from the equality (\ref{eq:last_key})
that for all $i\in\mathcal{I}_{R[H_{n}]}$ and $v\in F(2\natural\underline{n})$:

\begin{eqnarray*}
((\tau_{1}\circ\mathbf{LM})(F)([n'-n,g]))\circ\upsilon(F)_{\underline{n}}(i\underset{R[H_{n}]}{\otimes}v) & = & \mathcal{I}_{R[\mathcal{A}]}(id_{1}\natural\iota_{\underline{n}})(i)\underset{R[H_{n'+1}]}{\otimes}F(id_{1}\natural id_{1}\natural[n'-n,g])(v)\\
 & = & \upsilon(F)_{\underline{n'}}\circ((\mathcal{I}_{R[H]}\otimes_{R[H]}\tau_{2}(F)([n'-n,g]))(i\underset{R[H_{n}]}{\otimes}v).
\end{eqnarray*}
Hence, $\upsilon(F)$ is a natural transformation from $\mathcal{I}_{R[H]}\otimes_{R[H]}\tau_{2}F$
to $(\mathbf{LM}\circ\tau_{1})(F)$.

Let us now check that assembling all these natural transformations
for all the objects of $\mathbf{Fct}(\mathfrak{U}\mathcal{G},R\textrm{-}\mathfrak{Mod})$
defines a natural transformation from $\mathcal{I}_{R[H]}\otimes_{R[H]}\tau_{2}$
to $\mathbf{LM}\circ\tau_{1}$. Let $M$ and $N$ be two objects of
$\mathbf{Fct}(\mathfrak{U}\mathcal{G},R\textrm{-}\mathfrak{Mod})$
and $\eta:M\rightarrow N$ be a natural transformation. We recall
that the natural transformations $(\mathcal{I}_{R[H]}\otimes_{R[H]}\tau_{2})(\eta)$
and $(\tau_{1}\circ\mathbf{LM})(\eta)$ are respectively given by:
\[
(\mathcal{I}_{R[H]}\otimes_{R[H]}\tau_{2})(\eta)_{\underline{n}}=id_{\mathcal{I}_{R[H]}}\underset{R[H]}{\varotimes}\eta_{2\natural\underline{n}}\textrm{ and }(\tau_{1}\circ\mathbf{LM})(\eta)_{\underline{n}}=id_{\mathcal{I}_{R[H_{1+n}]}}\underset{R[H_{1+n}]}{\varotimes}\eta_{2\natural\underline{n}}
\]
for all natural numbers $n$. We straightforwardly deduce from the
definitions that $\upsilon(N)_{\underline{n}}\circ(\mathbf{LM}\circ\tau_{1})(\eta)_{\underline{n}}=(\tau_{1}\circ\mathbf{LM})(\eta)_{\underline{n}}\circ\upsilon(M)_{\underline{n}}$.
Hence $\upsilon(M)\circ(\mathbf{LM}\circ\tau_{1})(\eta)$ and $(\tau_{1}\circ\mathbf{LM})(\eta)\circ\upsilon(N)$
are equal as natural transformations from $(\mathcal{I}_{R[H]}\otimes_{R[H]}\tau_{2})(M)$
to $(\mathbf{LM}\circ\tau_{1})(N)$, which ends the proof. 
\end{proof}
\begin{rem}
\label{def:basisfunctor} Let us assume that $H$ is a free group
which rank is denoted by $r(H)$. Let $M$ be an $R[H]$-module. Since
$H$ is free, $\mathcal{I}_{R[H]}$ is a free $R[H]$-module of rank
$r(H)$ (see \cite[Proposition 6.2.6]{Weibel1} for instance), hence
there are canonical isomorphisms of $R$-modules $\mathcal{I}_{R[H]}\otimes_{R[H]}M\cong(R[H])^{\oplus r(H)}\otimes_{R[H]}M\cong M^{\oplus r(H)}$
by the distributive property of the tensor product.
\end{rem}

\begin{cor}
\label{cor:splittingtranslation}For $\{\mathcal{A},\mathcal{G},\mathcal{G}',\chi\}$
a reliable Long-Moody system, as endofunctors of $\mathbf{Fct}(\mathfrak{U}\mathcal{G},R\textrm{-}\mathfrak{Mod})$,
there is an isomorphism $\tau_{1}\circ\mathbf{LM}\cong(\mathcal{I}_{R[H]}\otimes_{R[H]}\tau_{2})\oplus(\mathbf{LM}\circ\tau_{1})$.
Furthermore, if the group $H$ is free of rank $r(H)$, there is a
natural equivalence $\mathcal{I}_{R[H]}\otimes_{R[H]}\tau_{2}\cong\tau_{2}^{\oplus r(H)}$.
\end{cor}

\begin{proof}
Since $\mathcal{I}_{R[\mathcal{A}]}^{-1}(id_{1}\natural\iota_{\underline{n}})\circ\mathcal{I}_{R[\mathcal{A}]}(id_{1}\natural\iota_{\underline{n}})=id_{\mathcal{I}_{R[H]}}$
for all natural numbers $n$, the natural transformation $\upsilon:\mathcal{I}_{R[H]}\otimes_{R[H]}\tau_{2}\rightarrow\tau_{1}\circ\mathbf{LM}$
is a right inverse of the natural transformation $\varUpsilon:\tau_{1}\circ\mathbf{LM}\rightarrow\textrm{Coker}(\xi)$.

We now assume that the group $H$ is free and consider an object $F$
of $\mathbf{Fct}(\mathfrak{U}\mathcal{G},R\textrm{-}\mathfrak{Mod})$.
The isomorphisms of Remark \ref{def:basisfunctor} provide the $R$-module
isomorphism $\mathcal{I}_{R[H]}\otimes_{R[H]}\tau_{2}F(\underline{n})\cong\tau_{2}^{\oplus r(H)}F(\underline{n})$
for all natural numbers $n$. The naturality with respect to the morphisms
of $\mathfrak{U}\mathcal{G}$ and that we define a natural transformation
in the category of endofunctors of $\mathbf{Fct}(\mathfrak{U}\mathcal{G},R\textrm{-}\mathfrak{Mod})$
both directly follow from the definitions of $\mathcal{I}_{R[H]}\otimes_{R[H]}\tau_{2}$
and $\tau_{2}^{\oplus r(H)}$.
\end{proof}

\subsubsection{Key relations with the difference and evanescence functors\label{sec:A-keystone-relation}}

This section presents the key commutation relations of the generalized
Long-Moody functors with the evanescence and difference functors.
Considering the endofunctor $\delta_{1}\circ\mathbf{LM}$ as an extension
of $\textrm{Coker}(\xi)$ by $\mathbf{LM}\circ\delta_{1}$ (see diagram
(\ref{eq:diagram_Xsi})), we prove that this extension is trivial
by using Corollary \ref{cor:splittingtranslation}.
\begin{thm}
\label{thm:Splitting LM}\label{thm:commutationkappaLM}Let $\{\mathcal{A},\mathcal{G},\mathcal{G}',\chi\}$
be a reliable Long-Moody system. There is a natural equivalence in
the category of endofunctors of $\mathbf{Fct}(\mathfrak{U}\mathcal{G},R\textrm{-}\mathfrak{Mod})$
\begin{equation}
\delta_{1}\circ\mathbf{LM}\cong(\mathcal{I}_{R[H]}\underset{R[H]}{\varotimes}\tau_{2})\oplus(\mathbf{LM}\circ\delta_{1}).\label{eq:decompostiondelta}
\end{equation}
Moreover, if we assume that the groups $H_{0}$ and $H$ are free,
then the evanescence endofunctor $\kappa_{1}$ commutes with the endofunctor
$\mathbf{LM}$ and the isomorphisms of Remark \ref{def:basisfunctor}
provide a natural equivalence:
\begin{equation}
\delta_{1}\circ\mathbf{LM}\cong\tau_{2}^{\oplus r(H)}\oplus(\mathbf{LM}\circ\delta_{1}).\label{eq:decompostiondeltafree}
\end{equation}
\end{thm}

\begin{proof}
To prove the first statement, we consider the diagram (\ref{eq:diagram_Xsi}).
Then the isomorphism $\tau_{1}\circ\mathbf{LM}\cong(\mathcal{I}_{R[H]}\otimes_{R[H]}\tau_{2})\oplus(\mathbf{LM}\circ\tau_{1})$
of Corollary \ref{cor:splittingtranslation} provides a splitting
of the top short exact sequence, which automatically induces a splitting
for the bottom short exact sequence and thus proves the relation (\ref{eq:decompostiondelta}).

Furthermore, assuming that the groups $H_{0}$ and $H$ are free,
we have the exact sequence (\ref{eq:LESLMi1}). We recall that for
two morphisms $f$ and $g$ of an abelian category such that the target
of $f$ is the source of $g$ and $g$ is a monomorphism, the snake
lemma provides an isomorphism $\ker(f)\cong\ker(g\circ f)$. Assigning
$f=\mathbf{LM}(i_{1})$ and $g=\xi$ and using the equality $\xi\circ(\mathbf{LM}(i_{1}))=i_{1}\mathbf{LM}$,
we thus conclude that $\mathbf{LM}\circ\kappa_{1}\cong\kappa_{1}\circ\mathbf{LM}$.
Also, the relation (\ref{eq:decompostiondeltafree}) is obtained from
(\ref{eq:decompostiondelta}) using the second statement of Corollary
\ref{cor:splittingtranslation}.
\end{proof}
Finally, the above work on the evanescence endofunctor $\kappa_{1}$
straightforwardly generalizes to any other evanescence endofunctor.
This property is used to prove the further Lemma \ref{lem:LMtau2sn}.
\begin{prop}
\label{rem:LMcommutetoutkappaHfree}Let $m\geq1$ be a natural number
and we assume that the groups $H_{0}$ and $H$ are free. Then the
evanescence endofunctor $\kappa_{m}$ commutes with the Long-Moody
functor.
\end{prop}

\begin{proof}
The result follows repeating verbatim the work on the evanescence
endofunctor of Section \ref{subsec:Relation-with-evanescence} simply
by modifying the index $1$ into any $m\geq1$. Indeed, we define
a natural transformation $\xi(F):(\mathbf{LM}\circ\tau_{m})(F)\rightarrow(\tau_{m}\circ\mathbf{LM})(F)$
analogously Proposition \ref{prop:subfunclmtm} by assigning to each
$n\in\mathbb{N}$ the morphism $\xi(F)_{\underline{n}}$ to be $(\mathcal{I}_{\mathcal{A}}(\iota_{m}\natural id_{\underline{n}})\otimes_{R[H_{m+n}]}F((b_{1,m}^{\mathfrak{\mathcal{G}'}})^{-1}\natural id_{\underline{n}}))$.
Then the analogue of Proposition \ref{lem:relationLMi1andi1LM} is
satisfied, the key point being that $(b_{1,m}^{\mathfrak{\mathcal{G}'}})^{-1}\circ(\iota_{m}\natural id_{1})=id_{1}\natural\iota_{m}$
by Relation (\ref{eq:defbraid}) since $\mathfrak{U}\mathcal{G}'$
is pre-braided. Then the hypotheses of Assumption \ref{ass:coherenceconditionbnautfn2et3}
are sufficient to prove the analogue results of that of Section \ref{subsec:Studycoker}:
in other words we define a natural equivalence $\tau_{m}\circ\mathbf{LM}\cong\tau_{m+1}^{\oplus m\cdot r(H)}\oplus(\mathbf{LM}\circ\tau_{m})$
in the category $\mathbf{Fct}(\mathfrak{U}\mathcal{G},R\textrm{-}\mathfrak{Mod)}$.
The result thus follows repeating mutatis mutandis the proof of Theorem
\ref{thm:Splitting LM}.
\end{proof}

\subsection{Effect on strong polynomial functors\label{subsec:Strong-polynomial-functors}}

In this section, we focus on the behaviour of the generalized Long-Moody
functor on (very) strong polynomial functors. We recover in particular
the results of \cite[Section 4]{soulieLMBilan} when $(\mathfrak{U}\mathcal{G}',\natural,0_{\mathcal{G}'})=(\mathfrak{U}\mathcal{G},\natural,0_{\mathcal{G}})=(\mathfrak{U}\boldsymbol{\beta},\natural,0)$.
First, we have the following property:
\begin{lem}
\label{lem:diffevantau2}The functor $\mathcal{I}_{R[H]}\otimes_{R[H]}\tau_{2}$
commutes with the difference functor $\delta_{1}$. Moreover, if $H$
is free, then $\mathcal{I}_{R[H]}\otimes_{R[H]}\tau_{2}$ commutes
with the evanescence functor $\kappa_{m}$ for all natural numbers
$m\geq1$.
\end{lem}

\begin{proof}
The commutation result with the difference functor $\delta_{1}$ is
a consequence of the right-exactness of the functor $\mathcal{I}_{R[H]}\otimes_{R[H]}-:R\textrm{-}\mathfrak{Mod}\rightarrow R\textrm{-}\mathfrak{Mod}$
and of the exactness and the commutation property of the translation
functor $\tau_{2}$ of Proposition \ref{prop:lemmecaract}. Assuming
that the group $H$ is free, the functor $\mathcal{I}_{R[H]}\otimes_{R[H]}-:R\textrm{-}\mathfrak{Mod}\rightarrow R\textrm{-}\mathfrak{Mod}$
is exact as a consequence of Lemma \ref{lem:Swan}. Hence, the claim
follows from the commutation of the evanescence functor $\kappa_{m}$
with the translation functor $\tau_{2}$ of Proposition \ref{prop:lemmecaract}.
\end{proof}
\begin{thm}
\label{Thm:emairesult}Let $d$ be a natural number and $F$ be an
object of $\mathbf{Fct}(\mathfrak{U}\mathcal{G},R\textrm{-}\mathfrak{Mod})$.
We recall that we consider a reliable Long-Moody system $\{\mathcal{A},\mathcal{G},\mathcal{G}',\chi\}$.
If the functor $F$ is strong polynomial of degree $d$, then:
\begin{itemize}
\item the functor $\mathcal{I}_{R[H]}\otimes_{R[H]}\tau_{2}(F)$ belongs
to $\mathcal{P}ol_{d}^{strong}(\mathfrak{U}\mathcal{G},R\textrm{-}\mathfrak{Mod})$;
\item the functor $\mathbf{LM}(F)$ belongs to $\mathcal{P}ol_{d+1}^{strong}(\mathfrak{U}\mathcal{G},R\textrm{-}\mathfrak{Mod})$.
\end{itemize}
Moreover, if the groups $H_{0}$ and $H$ are free and $F$ is very
strong polynomial of degree $d$, then the functor $\mathbf{LM}(F)$
is a very strong polynomial functor of degree equal to $d+1$.
\end{thm}

\begin{proof}
The result on \textit{$\mathcal{I}_{R[H]}\otimes_{R[H]}\tau_{2}(F)$}
follows from Lemma \ref{lem:diffevantau2}. By induction on the polynomial
degree, we deduce the first result on \textit{$\mathbf{LM}(F)$} using
the relation (\ref{eq:decompostiondelta}).

Assume now that the groups $H_{0}$ and $H$ are free. We recall that
$H$ is non-trivial. For a very strong polynomial functor $F$ of
degree $d$, we deduce from Lemma \ref{lem:diffevantau2} that $\mathcal{I}_{R[H]}\otimes_{R[H]}\tau_{2}F\cong\tau_{2}^{\oplus r(H)}F$
is also very strong polynomial of degree $d$. A fortiori, the result
follows from a clear induction using the relation (\ref{eq:decompostiondeltafree})
and the commutation property of the evanescence endofunctor $\kappa_{1}$
with the endofunctor $\mathbf{LM}$ of Theorem \ref{thm:Splitting LM}.
\end{proof}

\subsection{Effect on weak polynomial functors\label{subsec:Weak-polynomial-functors}}

We investigate the effect on weak polynomial functors of the Long-Moody
functor associated with the \textit{reliable} Long-Moody system $\{\mathcal{A},\mathcal{G},\mathcal{G}',\chi\}$.
The first step of this study consists in defining the Long-Moody functor
on the quotient category $\mathbf{St}(\mathfrak{U}\mathcal{G},R\textrm{-}\mathfrak{Mod})$.
First, we note the following property.
\begin{lem}
\label{lem:LMtau2sn}Let $F$ be an object of $\mathbf{Fct}(\mathfrak{U}\mathcal{G},R\textrm{-}\mathfrak{Mod})$.
Assume that the groups $H_{0}$ and $H$ are free. If the functor
$F$ is in $K(\mathfrak{U}\mathcal{G},R\textrm{-}\mathfrak{Mod})$,
then the functors $\mathbf{LM}(F)$ and $\mathcal{I}_{R[H]}\otimes_{R[H]}\tau_{2}(F)$
are in $K(\mathfrak{U}\mathcal{G},R\textrm{-}\mathfrak{Mod})$.
\end{lem}

\begin{proof}
We recall that the functor $\kappa$ is the colimit $\sum_{m\in\textrm{Obj}(\mathfrak{U}\mathcal{G})_{\natural}}\kappa_{m}$.
By Proposition \ref{rem:LMcommutetoutkappaHfree} and Lemma \ref{lem:diffevantau2},
the endofunctors $\mathbf{LM}$ and $\mathcal{I}_{R[H]}\otimes_{R[H]}\tau_{2}$
commute with the evanescence functor $\kappa_{m}$ for all natural
numbers $m\geq1$. We recall from Proposition \ref{prop:exactnessLM}
that the endofunctor $\mathbf{LM}$ commutes with all colimits, and
thus commutes with $\kappa$. Furthermore, we deduce the commutation
with all colimits of the functor $\mathcal{I}_{R[H]}\otimes_{R[H]}\tau_{2}(-)$
from the exactness of the translation functor $\tau_{2}$ (see Proposition
\ref{prop:lemmecaract}) and the exactness of the functor $\mathcal{I}_{R[H]}\otimes_{R[H]}-:R\textrm{-}\mathfrak{Mod}\rightarrow R\textrm{-}\mathfrak{Mod}$
(since the augmentation ideal $\mathcal{I}_{R[H]}$ is a projective
$R[H]$-module by Lemma \ref{lem:Swan}). Therefore, the functor $\mathcal{I}_{R[H]}\otimes_{R[H]}\tau_{2}(-)$
also commutes with $\kappa$.
\end{proof}
\textbf{From now until the end of Section \ref{subsec:Weak-polynomial-functors},
we assume that the groups $H_{0}$ and $H$ are free.} By Lemma \ref{lem:LMtau2sn},
the endofunctors $\mathbf{LM}$ and $\mathcal{I}_{R[H]}\otimes_{R[H]}\tau_{2}$
induce two functors on the quotient category $\mathbf{St}(\mathfrak{U}\mathcal{G},R\textrm{-}\mathfrak{Mod})$,
denoted by
\[
\mathbf{LM}_{\mathbf{St}}:\mathbf{St}(\mathfrak{U}\mathcal{G},R\textrm{-}\mathfrak{Mod})\rightarrow\mathbf{St}(\mathfrak{U}\mathcal{G},R\textrm{-}\mathfrak{Mod})\,\textrm{ and }\,(\mathcal{I}_{R[H]}\underset{R[H]}{\varotimes}\tau_{2})_{\mathbf{St}}:\mathbf{St}(\mathfrak{U}\mathcal{G},R\textrm{-}\mathfrak{Mod})\rightarrow\mathbf{St}(\mathfrak{U}\mathcal{G},R\textrm{-}\mathfrak{Mod}).
\]

\begin{prop}
\label{lem:behaviourdeltaLMweak} The induced functor $(\mathcal{I}_{R[H]}\varotimes_{R[H]}\tau_{2})_{\mathbf{St}}$
is equivalent to the functor $\tau_{2}^{\oplus r(H)}$, where the
last $\tau_{2}$ is the translation endofunctor of $\mathbf{St}(\mathfrak{U}\mathcal{G},R\textrm{-}\mathfrak{Mod})$.

Furthermore, for $F$ an object of $\mathbf{St}(\mathfrak{U}\mathcal{G},R\textrm{-}\mathfrak{Mod})$,
there are natural equivalences as objects of $\mathbf{St}(\mathfrak{U}\mathcal{G},R\textrm{-}\mathfrak{Mod})$:

\begin{equation}
\delta_{1}(\mathcal{I}_{R[H]}\underset{R[H]}{\varotimes}\tau_{2}(F))_{\mathbf{St}}\cong(\mathcal{I}_{R[H]}\underset{R[H]}{\varotimes}\tau_{2})_{\mathbf{St}}(\delta_{1}F),\label{eq:deltatau2}
\end{equation}

\begin{equation}
\delta_{1}\mathbf{LM}_{\mathbf{St}}(F)\cong(\mathcal{I}_{R[H]}\underset{R[H]}{\varotimes}\tau_{2}(F))_{\mathbf{St}}\oplus\mathbf{LM}_{\mathbf{St}}(\delta_{1}F).\label{eq:deltapiLM}
\end{equation}
\end{prop}

\begin{proof}
By Corollary \ref{cor:splittingtranslation}, we have a natural equivalence
$\mathcal{I}_{R\left[H\right]}\varotimes_{R[H]}\tau_{2}\cong\tau_{2}^{\oplus r(H)}$
and that $K(\mathfrak{U}\mathcal{G},R\textrm{-}\mathfrak{Mod})$ is
closed under colimits by Proposition \ref{prop:Snthick}. Let $G$
be an object of $\mathbf{Fct}(\mathfrak{U}\mathcal{G},R\textrm{-}\mathfrak{Mod})$.
We recall that $\kappa$ is left-exact and that $\kappa(\kappa_{2}G)=\kappa_{2}G$
(see Lemma \ref{rem:fourre-toutfilteredcolim} and its proof). We
deduce that $G$ is in $K(\mathfrak{U}\mathcal{G},R\textrm{-}\mathfrak{Mod})$
if the functor $(\mathcal{I}_{R[H]}\varotimes_{R[H]}\tau_{2})(G)$
is in $K(\mathfrak{U}\mathcal{G},R\textrm{-}\mathfrak{Mod})$. This
fact and Lemma \ref{lem:LMtau2sn} prove the first statement.

As a consequence of the definitions of the induced difference functor
of Proposition \ref{prop:piMcommutetauetdelta} and of the induced
functors $(\mathcal{I}_{R[H]}\otimes_{R[H]}\tau_{2})_{\mathbf{St}}$
and $\mathbf{LM}_{\mathbf{St}}$, we obtain the following two natural
equivalences $\delta_{1}(\mathcal{I}_{R[H]}\otimes_{R[H]}\tau_{2})_{\mathbf{St}}\cong(\delta_{1}(\mathcal{I}_{R[H]}\otimes_{R[H]}\tau_{2}))_{\mathbf{St}}$
and $\delta_{1}\mathbf{LM}_{\mathbf{St}}\cong(\delta_{1}\circ\mathbf{LM})_{\mathbf{St}}$.
The result then follows from Lemma \ref{lem:diffevantau2} and Theorem
\ref{thm:Splitting LM}.
\end{proof}
Then we can now prove:
\begin{thm}
\label{thm:ResultWeakpoly}Let $d$ be a natural number and recall
that the groups $H_{0}$ and $H$ are assumed to be free. Let $F$
be an object of $\mathbf{Fct}(\mathfrak{U}\mathcal{G},R\textrm{-}\mathfrak{Mod})$
which is weak polynomial of degree $d$. Then the functor $\mathcal{I}_{R[H]}\varotimes_{R[H]}\tau_{2}(F)$
is a weak polynomial functor of degree $d$ and the functor $\mathbf{LM}(F)$
is a weak polynomial functor of degree $d+1$.
\end{thm}

\begin{proof}
We proceed by induction on the degree of polynomiality of $F$. If
$F$ is weak polynomial of degree $0$, then there exists a constant
functor $C$ of \textit{$\mathbf{St}(\mathfrak{U}\mathcal{G},R\textrm{-}\mathfrak{Mod})$}
such that $\pi_{\mathfrak{U}\mathcal{G}}(F)\cong C$ by Proposition
\ref{prop:weakpolydeg0}. Hence, we deduce from Proposition \ref{lem:behaviourdeltaLMweak}
that $(\mathcal{I}_{R[H]}\varotimes_{R[H]}\tau_{2})_{\mathbf{St}}(C)\cong C^{\oplus r(H)}$,
which is a degree $0$ weak polynomial functor. Now, assume that $F$
is weak polynomial functor of degree $d\geq0$. Then, the result follows
from the relation (\ref{eq:deltatau2}) and the inductive hypothesis.

For the endofunctor $\mathbf{LM}$, we also proceed by induction as
follows. Assume that $F$ is a weak polynomial functor of degree $0$.
By the equivalence (\ref{eq:deltapiLM}), we obtain that $\delta_{1}(\pi_{\mathfrak{U}\mathcal{G}}(\mathbf{LM}(F)))\cong(\mathcal{I}_{R\left[H\right]}\varotimes_{R[H]}\tau_{2})(\pi_{\mathfrak{U}\mathcal{G}}(F))$.
Following the result on $\mathcal{I}_{R\left[H\right]}\varotimes_{R[H]}\tau_{2}$,
the functor $\delta_{1}(\pi_{\mathfrak{U}\mathcal{G}}(\mathbf{LM}(F)))$
is polynomial of degree $0$. Therefore, the functor $\mathbf{LM}\left(F\right)$
is therefore weak polynomial of degree $1$. Now, assume that $F$
is a weak polynomial functor of degree $d\geq1$. By the equivalence
(\ref{eq:deltapiLM}):\textit{
\[
\delta_{1}(\pi_{\mathfrak{U}\mathcal{G}}(\mathbf{LM}(F)))\cong(\mathcal{I}_{R\left[H\right]}\underset{R[H]}{\varotimes}\tau_{2})(\pi_{\mathfrak{U}\mathcal{G}}(F))\oplus\mathbf{LM}_{\mathbf{St}}(\delta_{1}(\pi_{\mathfrak{U}\mathcal{G}}(F))).
\]
}The result follows from the inductive hypothesis and the result on
$\mathcal{I}_{R\left[H\right]}\varotimes_{R[H]}\tau_{2}$.
\end{proof}

\section{Examples and applications\label{sec:Examples-and-applications}}

This last section presents applications of the results of Section
\ref{sec:Behaviour-of-the}. Namely, the generalized Long-Moody functors
provide very strong and weak polynomial functors in any degree for
the families of groups of Section \ref{subsec:Examples}. In particular,
they give twisted coefficients for which homological stability is
satisfied (see Section \ref{subsec:Strong-polynomial-functors-1})
and introduce a tool for classifying weak polynomial functors with
$\mathfrak{U}\mathcal{G}$ as source category.

\subsection{Strong polynomial functors\label{subsec:Strong-polynomial-functors-1}}
\begin{prop}
\label{prop:coherent_reliable}The coherent Long-Moody systems of
Sections \ref{subsec:Modifyingpunctures} and \ref{subsec:Surface-braid-groups:}
are reliable.
\end{prop}

\begin{proof}
We recall from Lemma \ref{lem:symplecticstrongmon} that the functor
$\pi_{1}(-,p)$ is strict monoidal. Hence, in particular, the assignments
conditions of Assumption \ref{assu:decomposeAfreeproduct} on the
functor $\mathcal{A}$ are satisfied. Also, for the families of morphisms
$\{\chi_{n,1}\}_{n\in\mathbb{N}}$, $\{\chi_{n,2}\}_{n\in\mathbb{N}}$,
$\{\chi_{n,1}^{\mathfrak{b}}\}_{n\in\mathbb{N}}$ and $\{\chi_{n,2}^{\mathfrak{b}}\}_{n\in\mathbb{N}}$,
the equality (\ref{eq:equalitylemma2.54}) of Proposition \ref{lem:sigma1satisfiesCondition}
implies that the  relation (\ref{eq:equiva'}) of Assumption \ref{assu:decomposeAfreeproduct}
is satisfied. A fortiori, it follows from the definition of the functor
$\pi_{1}(-,p)$ that Assumption \ref{assu:decomposeAfreeproduct}
is satisfied.
\end{proof}
Hence, applying a Long-Moody functor on the constant functor $\mathbb{Z}$,
we prove:
\begin{cor}
\label{cor:example strong}For each $d\geq0$, the functors $\mathbf{LM}_{\mathfrak{M}_{2}^{+,g}}^{\circ d}(H_{1}(\varSigma_{g,1}^{-},\mathbb{Z}))$,
$\mathbf{LM}_{\mathfrak{B}_{2}^{g}}^{\circ d}(H_{1}(\varSigma_{g,1}^{-},\mathbb{Z}))$,
$\text{\ensuremath{\mathbf{LM}_{\mathfrak{M}_{2}^{-,h}}^{\circ d}}(\ensuremath{H_{1}}(\ensuremath{\mathscr{N\varSigma}_{h,1}^{-}},\ensuremath{\mathbb{Z}}))}$
and $\text{\ensuremath{\mathbf{LM}_{\mathfrak{B}_{2}^{-,h}}^{\circ d}}(\ensuremath{H_{1}}(\ensuremath{\mathscr{N\varSigma}_{h,1}^{-}},\ensuremath{\mathbb{Z}}))}$
are very strong polynomial functors of degree $d+1$.
\end{cor}

Moreover, we have the following application for the Long-Moody functor
associated with surface braid group of Section \ref{subsec:Applications}.
Although $\mathbf{LM}_{\mathfrak{B}_{2}^{g}}(\mathbb{Z}[\mathfrak{B}_{2}^{g}/\varGamma_{3}])$
is not an object of $\mathbf{Fct}(\mathfrak{UB}_{2}^{g},\mathbb{Z}\textrm{-}\mathfrak{Mod})$,
the following slight modification of that functor defines a new functor
with $\mathfrak{UB}_{2}^{g}$ as source category. An analogous manipulation
was actually made in \cite[Section 2.3.1]{soulieLMBilan}. Let $\mathbb{Z}[\mathfrak{B}_{2}^{g}/\varGamma_{3}]^{*}:\mathfrak{B}_{2}^{g}\rightarrow\textrm{\ensuremath{\mathbb{Z}}-}\mathfrak{Mod}$
be the functor induced by the dual representations of that defined
by $\mathbb{Z}[\mathfrak{B}_{2}^{g}/\varGamma_{3}]$. Let $\widetilde{\mathbf{LM}}_{\mathfrak{B}_{2}^{g}}(\mathbb{Z}[\mathfrak{B}_{2}^{g}/\varGamma_{3}])$
denote the object $\mathbf{LM}_{\mathfrak{B}_{2}^{g}}(\mathbb{Z}[\mathfrak{B}_{2}^{g}/\varGamma_{3}])\otimes_{\mathbb{Z}[(\mathbb{Z}\times\mathbb{Z}^{g})\rtimes\mathbb{Z}^{g}]}\mathbb{Z}[\mathfrak{B}_{2}^{g}/\varGamma_{3}]^{*}$
of the category $\mathbf{Fct}(\mathfrak{B}_{2}^{g},\mathbb{Z}\textrm{-}\mathfrak{Mod})$,
where $\otimes_{\mathbb{Z}[(\mathbb{Z}\times\mathbb{Z}^{g})\rtimes\mathbb{Z}^{g}]}$
is the pointwise tensor product for the functor category. It is a
routine to check that, by assigning to all natural numbers $n$ and
$n'$ such that $n'\geq n$ the morphism $\widetilde{\mathbf{LM}}_{\mathfrak{B}_{2}^{g}}(\mathbb{Z}[\mathfrak{B}_{2}^{g}/\varGamma_{3}])([\varSigma_{g,1}^{n'-n},id_{\varSigma_{g,1}^{n'}}])$
to be the embedding
\[
\mathcal{I}_{\pi_{1}(\varSigma_{g,1}^{-},p)^{\mathfrak{b}}}([\varSigma_{g,1}^{n'-n},id_{\varSigma_{g,1}^{n'}}])\underset{\pi_{1}(\varSigma_{g,1}^{n'},p)}{\otimes}id_{\mathbb{Z}[\mathfrak{B}_{2}^{g}/\varGamma_{3}](\varSigma_{g,1}^{n'+1})},
\]
the relations (\ref{eq:criterion}) and (\ref{eq:criterion'}) are
then satisfied. It follows from Lemma \ref{lem:criterionfamilymorphismsfunctor}
that the functor $\widetilde{\mathbf{LM}}_{\mathfrak{B}_{2}^{g}}(\mathbb{Z}[\mathfrak{B}_{2}^{g}/\varGamma_{3}]):\mathfrak{B}_{2}^{g}\rightarrow\mathbb{Z}\textrm{-}\mathfrak{Mod}$
thus defines an object of $\mathbf{Fct}(\mathfrak{UB}_{2}^{g},\mathbb{Z}\textrm{-}\mathfrak{Mod})$.
Hence, we deduce from Proposition \ref{prop:coherent_reliable}:
\begin{prop}
\label{prop:Bur(S)_polynomial_one}For each $d\geq0$, the functor
$\mathbf{LM}_{\mathfrak{B}_{2}^{g}}^{\circ d}(\widetilde{\mathbf{LM}}_{\mathfrak{B}_{2}^{g}}(\mathbb{Z}[\mathfrak{B}_{2}^{g}/\varGamma_{3}]))$
is very strong polynomial of degree $d+1$.
\end{prop}

\begin{proof}
Let $n$ be a natural number. We note that the application $i_{1}\widetilde{\mathbf{LM}}_{\mathfrak{B}_{2}^{g}}(\mathbb{Z}[\mathfrak{B}_{2}^{g}/\varGamma_{3}])([0,id_{\underline{n}}])$
is a monomorphism which cokernel is $\mathbb{Z}[(\mathbb{Z}\times\mathbb{Z}^{g})\rtimes\mathbb{Z}^{g}]$.
Hence $\kappa_{1}\widetilde{\mathbf{LM}}_{\mathfrak{B}_{2}^{g}}(\mathbb{Z}[\mathfrak{B}_{2}^{g}/\varGamma_{3}])$
is the null functor. For a natural number $n'\geq n$ and $[\varSigma_{g,1}^{n'-n},\sigma]\in\textrm{Hom}_{\mathfrak{U}\mathfrak{B}_{2}^{g}}(\varSigma_{g,1}^{n},\varSigma_{g,1}^{n'})$.
It follows from the universal property of the cokernel and naturality
that $\delta_{1}\widetilde{\mathbf{LM}}_{\mathfrak{B}_{2}^{g}}(\mathbb{Z}[\mathfrak{B}_{2}^{g}/\varGamma_{3}])([\varSigma_{g,1}^{n'-n},\sigma])=id_{\mathbb{Z}[(\mathbb{Z}\times\mathbb{Z}^{g})\rtimes\mathbb{Z}^{g}]}.$
Hence, $\delta_{1}\widetilde{\mathbf{LM}}_{\mathfrak{B}_{2}^{g}}(\mathbb{Z}[\mathfrak{B}_{2}^{g}/\varGamma_{3}])$
is the constant functor equal to $\mathbb{Z}[(\mathbb{Z}\times\mathbb{Z}^{g})\rtimes\mathbb{Z}^{g}]$,
which is very strong polynomial of degree $0$. Hence the functor
$\widetilde{\mathbf{LM}}_{\mathfrak{B}_{2}^{g}}(\mathbb{Z}[\mathfrak{B}_{2}^{g}/\varGamma_{3}])$
is very strong polynomial of degree one. The result is thus a consequence
of Theorem \ref{Thm:emairesult} and Proposition \ref{prop:coherent_reliable}.
\end{proof}
In \cite[Section 5]{WahlRandal-Williams}, Randal-Williams and Wahl
prove homological stability for the families of mapping class groups
of surfaces and surface braid groups considered in Section \ref{subsec:Examples},
with twisted coefficients given by very strong polynomial functors.
Namely, for all the groupoids $\mathcal{G}$ and $\mathcal{G}'$ introduced
in the examples of Section \ref{subsec:Examples}, they show that
if $F:\mathfrak{U}\mathcal{G}\rightarrow\mathbb{Z}\textrm{-}\mathfrak{Mod}$
is a very strong polynomial functor of degree $d$, then the canonical
maps
\[
H_{*}(G_{n},F(\underline{n}))\rightarrow H_{*}(G_{n+1},F(\underline{n+1}))
\]
are isomorphisms for $N(*,d)\leq n$ with $N(*,d)\in\mathbb{N}$ depending
on $*$ and $d$; see \cite[Theorem A]{WahlRandal-Williams}. The
representation theory of mapping class groups of surfaces is wild
and an active research topic; see \cite[Section 4.6]{BirmanBrendlesurvey}
or \cite{Margalit} for example. Hence there are very few known examples
of functors appropriate for homological stability. Using Long-Moody
functors, we thus construct very strong polynomial functors in any
degree for these families of groups.

\subsection{Weak polynomial functors\label{subsec:Weak-polynomial-functors-1}}

By Proposition \ref{prop:weakpolydeg0}, the constant functor $\mathbb{Z}$
is weak polynomial of degree $0$. Also, it is clear that $\widetilde{\mathbf{LM}}_{\mathfrak{B}_{2}^{g}}(\mathbb{Z}[\mathfrak{B}_{2}^{g}/\varGamma_{3}])$
is weak polynomial of degree $1$ since all its evanescence functors
are trivial. It follows from Theorem \ref{thm:ResultWeakpoly} that:
\begin{prop}
For each $d\geq0$, the functors of Corollary \ref{cor:example strong}
and Proposition \ref{prop:Bur(S)_polynomial_one} are weak polynomial
of degree $d+1$.
\end{prop}

A strong polynomial functor of degree $d$ is always weak polynomial
of degree less than or equal to $d$ by the first property of Proposition
\ref{prop:piMcommutetauetdelta}. The converse is false; see \cite[Example 5.4]{DV3}
for a counterexample. Also, the weak polynomial degree of a strong
polynomial functor can be strictly smaller than its strong polynomial
degree as the following example shows. We recall from \cite[Section 1.3]{soulieLMBilan}
the functor $\overline{\mathfrak{Bur}}:\mathfrak{U}\boldsymbol{\beta}\rightarrow\mathbb{C}[t^{\pm1}]\textrm{-}\mathfrak{Mod}$
which encodes the family of reduced Burau representations.
\begin{prop}
The functor $\overline{\mathfrak{Bur}}:\mathfrak{U}\boldsymbol{\beta}\rightarrow\mathbb{C}[t^{\pm1}]\textrm{-}\mathfrak{Mod}$
is a strong polynomial functor of degree $2$ and weak polynomial
of degree $1$.
\end{prop}

\begin{proof}
Let $\mathbb{C}[t^{\pm1}]_{\geq1}$ be the subfunctor of the constant
functor $\mathbb{C}[t^{\pm1}]$ which is null at $0$ and equal to
$R$ elsewhere. The result on the strong polynomial degree is proved
in \cite[Proposition 3.28]{soulieLMBilan}, using the short exact
sequence $0\rightarrow\overline{\mathfrak{Bur}}_{t}\rightarrow\tau_{1}\overline{\mathfrak{Bur}}_{t}\rightarrow\mathbb{C}[t^{\pm1}]_{\geq1}\rightarrow0$
of $\mathbf{Fct}(\mathfrak{U}\boldsymbol{\beta},\mathbb{C}[t^{\pm1}]\textrm{-}\mathfrak{Mod})$.
Since $\pi_{\mathfrak{U}\mathcal{G}}$ is exact, we deduce that $\delta_{1}(\pi_{\mathfrak{U}\mathcal{G}}(\overline{\mathfrak{Bur}}_{t}))\cong\pi_{\mathfrak{U}\mathcal{G}}(\mathbb{C}[t^{\pm1}]_{\geq1})$.
The functor $\mathbb{C}[t^{\pm1}]_{\geq1}$ is a subfunctor of a weak
polynomial functor of degree $0$ and $\kappa(\mathbb{C}[t^{\pm1}]_{\geq1})\neq\mathbb{C}[t^{\pm1}]_{\geq1}$.
So, we deduce from Proposition \ref{prop:thickweak} that $\mathbb{C}[t^{\pm1}]_{\geq1}$
is weak polynomial of degree $0$ and then the functor $\overline{\mathfrak{Bur}}_{t}$
is weak polynomial of degree $1$.
\end{proof}
On the contrary, it is clear that the translation $\tau_{1}\overline{\mathfrak{Bur}}$
is both very strong and weak polynomial of degree $1$. This exemplifies
a benefit of the notion of weak polynomiality compared to the strong
one: it reflects more accurately the behaviour of functors, in particular
for large values.

Another fundamental reason for the notion of weak polynomial functors
to be introduced is that, contrary to the category $\mathcal{P}ol_{d}^{strong}(\mathfrak{M},\mathscr{A})$,
the category $\mathcal{P}ol_{d}(\mathfrak{M},\mathscr{A})$ is \textit{localizing}
by Proposition \ref{prop:thickweak}. This allows one to define the
quotient categories $\mathcal{P}ol_{d+1}(\mathfrak{M},\mathscr{A})/\mathcal{P}ol_{d}(\mathfrak{M},\mathscr{A})$.
Generally speaking, a refined description of the category $\mathcal{P}ol_{d}^{strong}(\mathfrak{M},\mathscr{A})$
is difficult and seems out of reach even for small $d$. On the contrary,
understanding the quotient categories $\mathcal{P}ol_{d+1}(\mathfrak{M},\mathscr{A})/\mathcal{P}ol_{d}(\mathfrak{M},\mathscr{A})$
is more attainable: for example, when $\mathfrak{M}$ is the category
$FI$ of finite sets and bijections, \cite[Proposition 5.9]{DV3}
gives a general equivalence of these quotients in terms of module
categories.

Also, considering $\mathfrak{M}=\mathfrak{U}\mathcal{G}$, these quotients
thus provide a new classifying tool for families of representations
of the groups $\{G_{n}\}_{n\in\mathbb{N}}$. Even if such a description
seems generally speaking out of reach for this kind of pre-braided
monoidal categories, the Long-Moody functors give a new implement
to understand these quotients. Indeed, as a consequence of Theorem
\ref{thm:ResultWeakpoly}, we obtain:
\begin{prop}
\label{prop:classificationweakpoly} For a reliable Long-Moody system
$\{\mathcal{A},\mathcal{G},\mathcal{G}',\chi\}$, if the groups $H_{0}$
and $H$ are free, the associated Long-Moody functor induces a functor:
\[
\mathcal{P}ol_{d}(\mathfrak{U}\mathcal{G},R\textrm{-}\mathfrak{Mod})/\mathcal{P}ol_{d-1}(\mathfrak{U}\mathcal{G},R\textrm{-}\mathfrak{Mod})\rightarrow\mathcal{P}ol_{d+1}(\mathfrak{U}\mathcal{G},R\textrm{-}\mathfrak{Mod})/\mathcal{P}ol_{d}(\mathfrak{U}\mathcal{G},R\textrm{-}\mathfrak{Mod}),
\]
\end{prop}

The first quotient category $\mathcal{P}ol_{1}(\mathfrak{U}\mathcal{G},R\textrm{-}\mathfrak{Mod})/\mathcal{P}ol_{0}(\mathfrak{U}\mathcal{G},R\textrm{-}\mathfrak{Mod})$
being the most reasonably computable directly (it is for instance
achievable for $\mathcal{G}=\boldsymbol{\beta}$), the Long-Moody
functors thus allow to describe subcategories of the further quotients
and can therefore be keys to understand them.

\bibliographystyle{plain}
\bibliography{bibliographiethese}

\begin{thebibliography}{10}

\bibitem{bellingeri2004presentations}
Paolo Bellingeri.
\newblock On presentations of surface braid groups.
\newblock {\em Journal of Algebra}, 274(2):543--563, 2004.

\bibitem{BellingeriGervais}
Paolo Bellingeri and Sylvain Gervais.
\newblock Surface framed braids.
\newblock {\em Geom. Dedicata}, 159:51--69, 2012.

\bibitem{bellingeriguaschigodelle}
Paolo Bellingeri, Eddy Godelle, and John Guaschi.
\newblock Abelian and metabelian quotient groups of surface braid groups.
\newblock {\em Glasg. Math. J.}, 59(1):119--142, 2017.

\bibitem{BirmanMCGBraid}
Joan~S. Birman.
\newblock Mapping class groups and their relationship to braid groups.
\newblock {\em Comm. Pure Appl. Math.}, 22:213--238, 1969.

\bibitem{birmanbraids}
Joan~S. Birman.
\newblock {\em Braids, links, and mapping class groups}.
\newblock Princeton University Press, Princeton, N.J.; University of Tokyo
  Press, Tokyo, 1974.
\newblock Annals of Mathematics Studies, No. 82.

\bibitem{BirmanBrendlesurvey}
Joan~S. Birman and Tara~E. Brendle.
\newblock Braids: a survey.
\newblock {\em Handbook of knot theory}, pages 19--103, 2005.

\bibitem{boedigtilman2}
Carl-Friedrich B\"odigheimer and Ulrike Tillmann.
\newblock Stripping and splitting decorated mapping class groups.
\newblock In {\em Cohomological methods in homotopy theory ({B}ellaterra,
  1998)}, volume 196 of {\em Progr. Math.}, pages 47--57. Birkh\"auser, Basel,
  2001.

\bibitem{boedigtilman1}
Carl-Friedrich B\"odigheimer and Ulrike Tillmann.
\newblock Embeddings of braid groups into mapping class groups and their
  homology.
\newblock In {\em Configuration spaces}, volume~14 of {\em CRM Series}, pages
  173--191. Ed. Norm., Pisa, 2012.

\bibitem{brown}
Kenneth~S Brown.
\newblock {\em Cohomology of groups}, volume~87.
\newblock Springer Science \& Business Media, 2012.

\bibitem{callegarosalvetti2017}
Filippo Callegaro and Mario Salvetti.
\newblock Homology of the family of hyperelliptic curves.
\newblock {\em Israel J. Math.}, 230(2):653--692, 2019.

\bibitem{Chillingworth}
D.~R.~J. Chillingworth.
\newblock Winding numbers on surfaces. {I}.
\newblock {\em Math. Ann.}, 196:218--249, 1972.

\bibitem{Chillingworth2}
D.~R.~J. Chillingworth.
\newblock Winding numbers on surfaces. {II}.
\newblock {\em Math. Ann.}, 199:131--153, 1972.

\bibitem{ChurchEllenbergFarbstabilityFI}
Thomas Church, Jordan~S. Ellenberg, and Benson Farb.
\newblock F{I}-modules and stability for representations of symmetric groups.
\newblock {\em Duke Math. J.}, 164(9):1833--1910, 2015.

\bibitem{cohencohomo}
Daniel~E. Cohen.
\newblock {\em Groups of cohomological dimension one}.
\newblock Lecture Notes in Mathematics, Vol. 245. Springer-Verlag, Berlin-New
  York, 1972.

\bibitem{djamentcongruence}
Aur{\'e}lien Djament.
\newblock On stable homology of congruence groups.
\newblock {\em arXiv:1707.07944}, 2017.

\bibitem{DV3}
Aur\'{e}lien Djament and Christine Vespa.
\newblock Foncteurs faiblement polynomiaux.
\newblock {\em Int. Math. Res. Not. IMRN}, (2):321--391, 2019.

\bibitem{EilenbergMacLane}
Samuel Eilenberg and Saunders Mac~Lane.
\newblock On the groups {$H(\Pi,n)$}. {II}. {M}ethods of computation.
\newblock {\em Ann. of Math. (2)}, 60:49--139, 1954.

\bibitem{farb2011primer}
Benson Farb and Dan Margalit.
\newblock {\em A Primer on Mapping Class Groups (PMS-49)}.
\newblock Princeton University Press, 2011.

\bibitem{FranjouFriedlanderScorichenkoSuslin}
Vincent Franjou, Eric~M. Friedlander, Alexander Scorichenko, and Andrei Suslin.
\newblock General linear and functor cohomology over finite fields.
\newblock {\em Ann. of Math. (2)}, 150(2):663--728, 1999.

\bibitem{gabriel}
Pierre Gabriel.
\newblock Des cat{\'e}gories ab{\'e}liennes.
\newblock {\em Bulletin de la Soci{\'e}t{\'e} Math{\'e}matique de France},
  90:323--448, 1962.

\bibitem{gramain1973type}
Andr{\'e} Gramain.
\newblock Le type d'homotopie du groupe des diff{\'e}omorphismes d'une surface
  compacte.
\newblock In {\em Annales scientifiques de l'{\'E}cole Normale Sup{\'e}rieure},
  volume~6, pages 53--66. Elsevier, 1973.

\bibitem{graysonQuillen}
Daniel Grayson.
\newblock Higher algebraic {$K$}-theory: {II} (after {D}aniel {Q}uillen).
\newblock In {\em Algebraic {$K$}-theory}, pages 217--240. Lectures Notes in
  Math., Vol.551, Springer, Berlin, 1976.

\bibitem{GuaschiJuan-Pineda}
John Guaschi and Daniel Juan-Pineda.
\newblock A survey of surface braid groups and the lower algebraic {$K$}-theory
  of their group rings.
\newblock In {\em Handbook of group actions. {V}ol. {II}}, volume~32 of {\em
  Adv. Lect. Math. (ALM)}, pages 23--75. Int. Press, Somerville, MA, 2015.

\bibitem{Hamstrom}
Mary-Elizabeth Hamstrom.
\newblock Homotopy groups of the space of homeomorphisms on a {$2$}-manifold.
\newblock {\em Illinois J. Math.}, 10:563--573, 1966.

\bibitem{hatcher2002algebraic}
Allen Hatcher.
\newblock Algebraic topology. 2002.
\newblock {\em Cambridge UP, Cambridge}, 606(9), 2002.

\bibitem{HennLannesSchwartz}
Hans-Werner Henn, Jean Lannes, and Lionel Schwartz.
\newblock The categories of unstable modules and unstable algebras over the
  {S}teenrod algebra modulo nilpotent objects.
\newblock {\em Amer. J. Math.}, 115(5):1053--1106, 1993.

\bibitem{Kuhn}
Nicholas~J. Kuhn.
\newblock A stratification of generic representation theory and generalized
  {S}chur algebras.
\newblock {\em $K$-Theory}, 26(1):15--49, 2002.

\bibitem{Long1}
D.~D. Long.
\newblock Constructing representations of braid groups.
\newblock {\em Comm. Anal. Geom.}, 2(2):217--238, 1994.

\bibitem{MacLane1}
Saunders Mac~Lane.
\newblock {\em Categories for the working mathematician}, volume~5.
\newblock Springer Science \& Business Media, 2013.

\bibitem{MacLaneMoerdijk}
Saunders Mac~Lane and Ieke Moerdijk.
\newblock {\em Sheaves in geometry and logic}.
\newblock Universitext. Springer-Verlag, New York, 1994.
\newblock A first introduction to topos theory, Corrected reprint of the 1992
  edition.

\bibitem{Margalit}
Dan Margalit.
\newblock Problems, questions, and conjectures about mapping class groups.
\newblock In {\em Breadth in contemporary topology}, volume 102 of {\em Proc.
  Sympos. Pure Math.}, pages 157--186. Amer. Math. Soc., Providence, RI, 2019.

\bibitem{palmer2017comparison}
Martin Palmer.
\newblock A comparison of twisted coefficient systems.
\newblock {\em arXiv:1712.06310}, 2017.

\bibitem{Pirashvili}
Teimuraz Pirashvili.
\newblock Hodge decomposition for higher order {H}ochschild homology.
\newblock {\em Ann. Sci. \'{E}cole Norm. Sup. (4)}, 33(2):151--179, 2000.

\bibitem{Popescu}
N.~Popescu.
\newblock {\em Abelian categories with applications to rings and modules}.
\newblock Academic Press, London-New York, 1973.
\newblock London Mathematical Society Monographs, No. 3.

\bibitem{Popescu2}
Nicolae Popescu and Liliana Popescu.
\newblock {\em Theory of categories}.
\newblock Martinus Nijhoff Publishers, The Hague; Sijthoff \& Noordhoff
  International Publishers, Alphen aan den Rijn, 1979.

\bibitem{WahlRandal-Williams}
Oscar Randal-Williams and Nathalie Wahl.
\newblock Homological stability for automorphism groups.
\newblock {\em Adv. Math.}, 318:534--626, 2017.

\bibitem{soulieLMBilan}
Arthur Souli\'e.
\newblock The long--moody construction and polynomial functors.
\newblock {\em Annales de l'Institut Fourier}, 69(4):1799--1856, 2019.

\bibitem{swan1969groups}
Richard~G Swan.
\newblock Groups of cohomological dimension one.
\newblock {\em Journal of Algebra}, 12(4):585--610, 1969.

\bibitem{Trapp}
Rolland Trapp.
\newblock A linear representation of the mapping class group $\mathscr{M}$ and
  the theory of winding numbers.
\newblock {\em Topology Appl.}, 43(1):47--64, 1992.

\bibitem{wada1992group}
Masaaki Wada.
\newblock Group invariants of links.
\newblock {\em Topology}, 31(2):399--406, 1992.

\bibitem{Weibel1}
Charles~A. Weibel.
\newblock {\em An introduction to homological algebra}, volume~38 of {\em
  Cambridge Studies in Advanced Mathematics}.
\newblock Cambridge University Press, Cambridge, 1994.

\end{thebibliography}

\lyxaddress{\textit{E-mail address: }\texttt{artsou@hotmail.fr}}
\end{document}